\documentclass[11pt]{article}

\newcounter{colt}
\usepackage[sc]{mathpazo}

\usepackage[margin=1in]{geometry}
\usepackage[english]{babel}
\usepackage[utf8x]{inputenc}
\usepackage[compact]{titlesec}

\usepackage{cmap}
\usepackage[T1]{fontenc}
\usepackage{bm}

\usepackage{amsmath}
\usepackage{amsfonts}
\usepackage{amssymb}
\usepackage{amsbsy}
\usepackage{amsthm}
\usepackage{dsfont}

\usepackage{mathtools}
\usepackage{xspace}

\usepackage{graphicx, ucs}

\usepackage{subcaption}
\usepackage{rotating}
\usepackage{float}
\usepackage{tikz}

\usepackage[ruled,linesnumbered]{algorithm2e}

\usepackage{listings}

\usepackage{enumitem}
\usepackage{hyperref}
\hypersetup{
  colorlinks = true,
  urlcolor = {blueGrotto},
  linkcolor = {royalBlue},
  citecolor = {navyBlueP}
}
\usepackage{multirow}
\usepackage{array}

\usepackage{accents}

\usepackage[outline]{contour}
\usepackage{xcolor}

\usepackage[linewidth=2pt,
linecolor=gray,
middlelinecolor= black,
middlelinewidth=0.4pt,
roundcorner=1pt,
topline = false,
rightline = false,
bottomline = false,
rightmargin=0pt,
skipabove=0pt,
skipbelow=0pt,
leftmargin=0pt,
innerleftmargin=4pt,
innerrightmargin=0pt,
innertopmargin=0pt,
innerbottommargin=0pt,
]{mdframed}

\usepackage{chngcntr}
\usepackage{soul}
\usepackage{arydshln}

\newcommand{\R}{\ensuremath{\mathbb{R}}}

 \renewcommand{\vec}[1]{\boldsymbol{#1}}

\newcommand{\vecw}{\ensuremath{\boldsymbol{w}}}
\newcommand{\vecx}{\ensuremath{\boldsymbol{x}}}
\newcommand{\vecy}{\ensuremath{\boldsymbol{y}}}
\newcommand{\vecz}{\ensuremath{\boldsymbol{z}}}

 \ifnum \value{colt} < 1{
\theoremstyle{plain}            \newtheorem{theorem}{Theorem}[section]
\newtheorem{lemma}[theorem]{Lemma}
\newtheorem{corollary}[theorem]{Corollary}
}
\fi
\newtheorem{claim}[theorem]{Claim}

\ifnum \value{colt} < 1{
\theoremstyle{definition}       \newtheorem{definition}[theorem]{Definition}
}
\fi
\newtheorem{assumption}[theorem]{Assumption}

\ifnum \value{colt} < 1{
\theoremstyle{remark}           
}
\fi

\numberwithin{equation}{section}

\contourlength{0.1pt}
\contournumber{10}

\newif\ifnotes\notestrue

\ifnotes
\usepackage{color}
\definecolor{mygrey}{gray}{0.50}
\newcommand{\notename}[2]{{\textcolor{red}{\footnotesize{\bf (#1:} {#2}{\bf
) }}}}

\else

\newcommand{\notename}[2]{{}}

\fi

\DeclareMathOperator*{\Exp}{{\mathbb{E}}}
\DeclareMathOperator*{\Prob}{{\mathbb{P}}}

\DeclareMathOperator*{\Var}{\mathrm{Var}}

 \newcommand{\calA}{\ensuremath{\mathcal{A}}}
\newcommand{\calB}{\ensuremath{\mathcal{B}}}

\newcommand{\calD}{\ensuremath{\mathcal{D}}}
\newcommand{\calE}{\ensuremath{\mathcal{E}}}
\newcommand{\calF}{\ensuremath{\mathcal{F}}}
\newcommand{\calG}{\ensuremath{\mathcal{G}}}
\newcommand{\calH}{\ensuremath{\mathcal{H}}}
\newcommand{\calI}{\ensuremath{\mathcal{I}}}

\newcommand{\calK}{\ensuremath{\mathcal{K}}}
\newcommand{\calL}{\ensuremath{\mathcal{L}}}

\newcommand{\calO}{\ensuremath{\mathcal{O}}}
\newcommand{\calP}{\ensuremath{\mathcal{P}}}
\newcommand{\calQ}{\ensuremath{\mathcal{Q}}}
\newcommand{\calR}{\ensuremath{\mathcal{R}}}
\newcommand{\calS}{\ensuremath{\mathcal{S}}}
\newcommand{\calT}{\ensuremath{\mathcal{T}}}
\newcommand{\calU}{\ensuremath{\mathcal{U}}}
\newcommand{\calV}{\ensuremath{\mathcal{V}}}

\newcommand{\calX}{\ensuremath{\mathcal{X}}}
\newcommand{\calY}{\ensuremath{\mathcal{Y}}}
\newcommand{\calZ}{\ensuremath{\mathcal{Z}}}

\mathchardef\mdash="2D
\newcommand{\eps}{\varepsilon}

\renewcommand{\epsilon}{\varepsilon}

\def\compactify{\itemsep=0pt \topsep=0pt \partopsep=0pt \parsep=0pt}
\let\latexusecounter=\usecounter

\newenvironment{Enumerate}
  {\def\usecounter{\compactify\latexusecounter}
   \begin{enumerate}}
  {\end{enumerate}\let\usecounter=\latexusecounter}

\DeclareMathOperator*{\argmax}{argmax}
\DeclareMathOperator*{\argmin}{argmin}

\def\floor#1{\mathop{\left\lfloor#1\right\rfloor}}

\def\abs#1{\left|#1\right|}
\def\p#1{\left(#1\right)}
\def\b#1{\left[#1\right]}
\def\set#1{\left\{#1\right\}}

\newcommand{\paragr}[1]{\noindent \textbf{#1}}

\def\norm#1{\left\|#1\right\|}

\definecolor{secinhead}{RGB}{249,196,95}
\definecolor{niceRed}{RGB}{190,38,38}
\definecolor{blueGrotto}{HTML}{059DC0}
\definecolor{royalBlue}{HTML}{057DCD}
\definecolor{navyBlueP}{HTML}{0B579C}
\definecolor{limeGreen}{HTML}{81B622}

\newcommand{\chara}{\mathds{1}}

\newcommand{\queue}{\mathbf{queue}}
\newcommand{\sh}{\text{star}}
\newcommand{\train}{\mathrm{D}}
\renewcommand{\bar}{\overline}

\newcommand{\level}{\ensuremath{\mathrm{level}}}
\newcommand{\nodes}{\ensuremath{\mathrm{nodes}}}

\begin{document}

\title{Estimation and Inference with Trees and Forests \\ in High Dimensions}
\author{ Vasilis Syrgkanis \\
 Microsoft Research \\
 \url{vasy@microsoft.com}
 \and
 Manolis Zampetakis\\
 MIT \\
 \url{mzampet@mit.edu}
}
\date{}

\maketitle
\begin{abstract}  We analyze the finite sample mean squared error (MSE) performance of regression trees and forests in the high dimensional regime with binary features, under a sparsity constraint. We prove that if only $r$ of the $d$ features are relevant for the mean outcome function, then shallow trees built greedily via the CART empirical MSE criterion achieve MSE rates that depend only logarithmically on the ambient dimension $d$. We prove upper bounds, whose exact dependence on the number relevant variables $r$ depends on the correlation among the features and on the degree of relevance. For strongly relevant features, we also show that fully grown honest forests achieve fast MSE rates and their predictions are also asymptotically normal, enabling asymptotically valid inference that adapts to the sparsity of the regression function.\end{abstract}

\section{Introduction} \label{sec:intro}

Regression Trees \cite{breiman1984classification} and their ensemble counterparts, Random Forests \cite{Breiman01}, are one of the most widely used estimation methods by machine learning practitioners. Despite their widespread use, their theoretical underpinnings are far from being fully understood. Early works established sample complexity bounds of decision trees and other data-adaptive partitioning estimators \cite{nobel1996,lugosi1996,Mansour2000}. However, sample complexity bounds do not address the computational aspect of how to choose the best tree in the space. In practice, trees and forests are constructed in a greedy fashion, typically identifying the most empirically informative split at east step; an approach pioneered by \cite{breiman1984classification,Breiman01}. The consistency and estimation rates of such greedily built trees has proven notoriously more difficult to analyze.

Recent breakthrough advances has shown that such greedily built trees are asymptotically consistent \cite{Biau2010,Denil2014,scornet2015} in the low dimensional regime, where the number of features is a constant independent of the sample size. In another line of work, \cite{Mentch2016,WagerA18} provide asymptotic normality results for honest versions of Random Forests, where each tree is constructed using a random sub-sample of the original data and further each tree construction algorithm sub-divides the sub-sample into a random half-sample that is used for construction of the tree structure and a separate half-sample used for the leaf estimates. However, these results are typically asymptotic or their finite sample guarantees scale exponentially with the number of features. Random Forests are used in practice to address estimation with high-dimensional features. Hence, these works, though of immense theoretical importance in our understanding of adaptively built trees and forests, they do not provide theoretical foundations of the finite sample superiority of these algorithms in practice.

In this work, we analyze the performance of regression trees and forests in the high-dimensional regime, where the number of features can grow exponentially with the number of samples. To focus on the high-dimensionality of the features (as opposed to the ability of forests to sub-divide continuous variable spaces), we constrain our analysis to the case when all features are binary.

We show that trees and forests built greedily based on the original CART criterion, provably adapt to sparsity: when only a subset $R$, of size $r$, of the features are relevant, then the mean squared error of appropriately shallow trees, or fully grown honest forests, scales exponentially only with the number of relevant features and depends only logarithmically on the overall number of features. We analyze two variants of greedy tree algorithms: in the level-split variant, the same variable is chosen at all the nodes of each level of the tree and is greedily chosen so as to maximize the overall variance reduction. In the second variant, which is the most popular in practice, the choice of the next variable to split on is locally decided at each node of the tree.

We identify three regimes, each providing different dependence on the number of relevant features. When the relevant variables are ``weakly'' relevant (in the sense that there is not strong separation between the relevant and irrelevant variables in terms of their ability to reduce variance), then shallow trees achieve ``slow rates'' on the mean squared error of the order of $2^r/\sqrt{n}$, when variables are independent, and $1/n^{1/(r+2)}$, when variables are dependent. When the relevant variables are ``strongly'' relevant, in that there is a separation in their ability to reduce variance as compared to the irrelevant ones, by a constant $\beta_{\min}$, then we show that greedily built shallow trees and fully grown honest forests can achieve fast parametric mean squared error rates of the order of $2^r/(\beta_{\min}\, n)$.

When variables are strongly relevant, we also show that the predictions of sub-sampled honest forests have an asymptotically normal distribution centered around their true values and whose variance scales at most as $O(2^r \log(n)/(\beta_{\min}\, n))$. Thus sub-sampled honest forests are provably a data-adaptive method for non-parametric inference, that adapts to the latent sparsity dimension of the data generating distribution, as opposed to classical non-parametric regression approaches, whose variance would deteriorate drastically with the overall number of features. Our results show that, at least for the case of binary features, forest based algorithms can offer immense improvement on the statistical power of non-parametric hypothesis tests in high-dimensional regimes.

The main crux of our analysis is showing bounds on the decay of the bias of decision trees, constructed via the mean-squared-error criterion. In particular, we show that either a relevant variable leads to a large decrease in the mean squared error, in which case we prove that with high probability it is chosen in the first few levels of the tree or if not then it's impact on the mean squared due to the fact that the algorithm failed to choose it can be controlled. For achieving the fast rates of $1/n$ for shallow trees, we also develop a new localized Rademacher analysis \cite{Bartlett2002,wainwright_2019} for adaptive partitioning estimators \cite{gyorfi2006distribution} to provide fast rates for the ``variance'' part of the MSE. Our results on honest forests, utilize recent work on the concentration and asymptotic normality of sub-sampled estimators \cite{Mentch2016,WagerA18,Fan2018} and combine it with our proof of the bias decay, which for the case of strongly relevant features, as we show, is exponential in the number of samples.

Several theoretical aspects of variants of CART trees and forests have been analyzed in the recent years \cite{Lin02randomforests,Meinshausen2006,Arlot2014,Breiman04consistencyfor,Scornet2016}. The majority of these works deal with the low dimensional regime and with few exceptions, these results deal with trees built with random splitting criteria or make no use of the fact that splits are chosen to minimize the CART mean-squared-error criterion. Arguably, closest to our work is that of \cite{Wager2015}, who consider a high dimensional regime with continuous variables, distributed according to a distribution with continuous density and uniformly upper and lower bounded. The main focus of this work is proving a uniform concentration bound on the mean squared error objective locally at every node of the adaptively constructed tree and for this reason makes several assumptions not present in our work: e.g. minimum leaf size constraints, approximately balanced splits. Crucially, their results on random forest consistency require an analogue of our $\beta_{\min}$ condition, and do not offer results without strong relevance. Moreover, their results on the consistency of forests requires a strong modification of the CART algorithm: split variables are selected based on initial median splits, subject to a lower bound on the decrease in variance, and then only the chosen variables are used in subsequent splits, not based on a CART criterion, but rather simply choosing random median splits; invoking an analysis of such random median trees in low dimensions by \cite{Duroux2016}. Moreover, they provide no results on asymptotic normality.

  Apart from the literature related to Random Forests and the CART criterion, there has
been a great amount of work on the sparse non-parametric regression problem that we
consider in this paper. A lot of heuristic methods have been proposed such as: $C_p$
and AIC for additive models \cite{HastieTF09}, MARS \cite{Friedman91}, Bayesian methods
\cite{GeorgeM97} and Gaussian Processes \cite{SmolaB01}. All these methods are
very successful in many practical scenarios but our theoretical understanding of their
performance is limited. Our work is closer to the works of
\cite{LaffertyW08, LiuC09, CommingesD12, YangT15} and citations therein, where they
propose and theoretically analyze greedy algorithms that exploit the sparsity of the
input regression function and hence they provide a way to overcome the curse of
dimensionality of high dimensional data in non-parametric regression. The main
difference of this line of work with our paper is that we do not propose a new
algorithm but instead our goal is to analyze the performance of the heuristically
proposed CART trees in the setting of sparse high-dimensional non-parametric regression
with binary features.

\section{Model and Preliminaries} \label{sec:model}

  In this work we consider the non-parametric regression model with binary
features. More precisely, we assume that we have access to a training set
$\train_n = \{(\vecx^{(1)}, y^{(1)}), \dots, (\vecx^{(n)}, y^{(n)})\}$, which
consists of $n$ i.i.d. samples of the form $(\vecx^{(i)}, y^{(i)})$, sampled
independently from a common distribution $\calD$. Each sample is generated according to the
following steps:
\begin{Enumerate}
  \item $\vecx^{(i)}$ is sampled from a distribution $\calD_x$ with support
  $\{0, 1\}^d$,
  \item $\eps^{(i)}$ is sampled from a zero mean error distribution $\calE$ with
  support $[-1/2, 1/2]$, i.e. $\Exp_{\eps \sim \calE}\b{\eps} = 0$ and
  $\eps^{(i)} \in [-1/2, 1/2]$,
  \item $y^{(i)} = m(\vecx^{(i)}) + \eps^{(i)}$, where
  $m : \{0, 1\}^d \to [-1/2, 1/2]$.
\end{Enumerate}
\noindent The goal of the regression task is to estimate the
\textit{target function} $m$. Observe from the definition of the non-parametric
regression model we have that $y^{(i)} \in [-1, 1]$. Our results apply
to any case where both the error distribution and the values of the target
function are bounded, i.e. $\abs{\eps^{(i)}} \le H$ and $\abs{m(\vecx)} \le H$.
In this case the sample complexity bounds and the rates should be multiplied by
$H$. For simplicity we present the result for the case $\abs{y^{(i)}} \le 1$.

For any vector $\vecx \in \{0, 1\}^d$ we define the vector $\vecx_S$ as the
sub-vector of $\vecx$, where we keep only the coordinates with indices in
$S \subseteq [d]$. Additionally, we define $\calD_{x, S}$ as the marginal
distribution of $\calD_x$ in the coordinates $S$. Also, let $\vecx^{(K)}$ be an
arbitrary $\vecx^{j}$ such that $j \in K$. For any training set $\train_n$ we
define the set $\train_{n, x} = \{\vecx^{(1)}, \dots, \vecx^{(n)}\}$. For any set
$S \subseteq [d]$ and an index $i \in [d]$ we sometimes use the notation
$S \cup i$ to refer to $S \cup \{i\}$.

  All of the results that we present in this paper are in the ``high-dimensional''
regime, where the number of features is very big but the number of
\textit{relevant features} is small. When this is true we say that the function $m$
is \textit{sparse} as we explain in the following definition.

\begin{definition}[\textsc{Sparsity}]
    We say that the target function $m : \{0, 1\}^d \to \R$ is $r$-\textit{sparse} if and
  only if there exists a set $R \subseteq [d]$, with $\abs{R} = r$ and a function
  $h : \{0, 1\}^r \to \R$ such that for every $\vecz \in \{0, 1\}^d$ it holds that
  $m(\vecz) = h(\vecz_R)$. The set $R$ is called the set of relevant features.
\end{definition}

  Some of the results that we have in this paper significantly improve if we
make the assumption that the feature vector distribution $\calD_x$ is a product
distribution. For this reason we define the ``independence of features''
assumption.

\begin{assumption}[\textsc{Independence of Features}] \label{asp:independentFeature}
    We assume that there exist Bernoulli distributions $\calB_1, \dots, \calB_d$
 such that $x^{(i)}_j$ is distributed independently according to $\calB_j$.
\end{assumption}

Now we give some definitions, related to the structure of a binary regression tree, that are important for the presentation of our results.

\begin{definition}[\textsc{Partitions, Cells and Subcells}]
    A \textit{partition} $\calP$ of $\{0, 1\}^d$ is a family of sets
  $\{A_1, \dots, A_s\}$ such that $A_j \subseteq \{0, 1\}^d$,
  $A_j \cap A_k = \emptyset$ for all $j, k \in [s]$, and
  $\bigcup_{j = 1}^s A_j = \{0, 1\}^d$.

    Let $\calP$ be a partition of $\{0, 1\}^d$. Every element $A$ of
  $\calP$ is called a \textit{cell of $\calP$} or just \textit{cell},
  if $\calP$ is clear from the context. Every cell $A$ has two
  \textit{subcells} $A^i_0$, $A^i_1$ with respect to any direction $i$,
  which are defined as $A^i_0 = \{\vecx \in A \mid x_i = 0\}$ and
  $A^i_1 = \{\vecx \in A \mid x_i = 1\}$.

    For any $\vecx \in \{0, 1\}^d$ and any partition $\calP$, we define
  $\calP(\vecx) \in \calP$ as the cell of $\calP$ that contains
  $\vecx$.
\end{definition}

\begin{definition}[\textsc{Split Operator and Refinement}] \label{def:splitOperator}
    For any partition $\calP$ of $\{0, 1\}^d$, any cell
  $A \in \calP$ and any $i \in [d]$ we define the
  \textit{split operator} $\calS(\calP, A, i)$ that outputs a
  partition with the cell $A$ split with respect to direction
  $i$. More formally
  $\calS(\calP, A, i) = \p{\calP \setminus \{A\}} \cup \{A^i_0, A^i_1\}$.
  We can also extend the definition of the split operator to splits over
  sets of dimensions $I \subseteq [d]$, inductively as follows: if
  $i \in I$ then
  $\calS(\calP, A, I) = \calS(\calS(\calS(\calP, A, i), A^i_0, I \setminus \{i\}), A^i_1, I \setminus \{i\})$.

    A partition $\calP'$ is a \textit{refinement} of a partition $\calP$
  if every element of $\calP'$ is a subset of an element of $\calP$. Then
  we say that $\calP'$ is \textit{finer} than $\calP$ and $\calP$ is
  \textit{coarser} than $\calP'$ and we use the notation
  $\calP' \sqsubseteq \calP$.
\end{definition}

\section{Consistency of Level-Splits Algorithm for Sparse Functions} \label{sec:levelSplitShallow}

  In this section we present our analysis for the case when we run
\textit{a level split} greedy algorithm to build a tree or a forest that
approximates the target function $m$. We start with the necessary definitions
to present the algorithm that we use. Then in Section
\ref{sec:app:population:levelSplitShallow} we present an analysis of the
population version of the algorithm, that is useful to gain intuition on the
finite sample proof that we present in full detail in the Appendix.
\smallskip

  Given a set of splits $S$, we define the \textit{expected mean squared error}
of $S$ as follows:
\begin{align} \label{eq:MSEPopulationBinaryLevelSplit}
  \bar{L}(S) & = \Exp_{\vecx \sim \calD_x}\b{\p{m(\vecx) - \Exp_{\vecw \sim \calD_x}\b{m(\vecw) \mid \vecw_S = \vecx_S}}^2} \\
  & = \Exp_{\vecx \sim \calD_x}\b{m^2(\vecx)} - \Exp_{\vecz_S \sim \calD_{x, S}} \b{\p{\Exp_{\vecw \sim \calD_x}\b{m(\vecw) \mid \vecw_S = \vecz_S}}^2} \nonumber \\
  & \triangleq \Exp_{\vecx \sim \calD_x}\b{m^2(\vecx)} - \bar{V}(S). \label{eq:MSEPopulationBinaryLevelSplit:2}
\end{align}

It is easy to see that $\bar{L}$ is a monotone decreasing function of $S$ and
hence $\bar{V}$ is a monotone increasing function of $S$. $\bar{V}$ can be viewed as a \emph{measure of heterogeneity} of the within-leaf mean values of the target function $m$, from the leafs created by split $S$.

  We present results based on either one of two main assumption about $\bar{V}$: \textit{approximate submodularity} of $\bar{V}$ or \textit{strong sparsity}. These assumptions play a crucial role in the analysis
of the performance of the random forest algorithm, both in the finite sample
regime and the population regime (presented in Section
\ref{sec:app:population:levelSplitShallow}). It is not difficult to see that without any assumption, no meaningful result about the consistency of greedily grown trees in the high dimensional setting is
possible, as we illustrate in Appendix \ref{sec:app:lowerBound:submodularity}.

\begin{assumption}[\textsc{Approximate Submodularity}] \label{asp:submodularityLevelSplit}
    Let $C \ge 1$, we say that the function $\bar{V}$ is
  $C$-\textit{approximate submodular}, if and only if for any
  $T, S \subseteq [d]$, such that $S \subseteq T$, and any $i \in [d]$, it holds
  that
  $\bar{V}(T \cup \{i\}) - \bar{V}(T) \le C \cdot (\bar{V}(S \cup \{i\}) - \bar{V}(S))$.
\end{assumption}

\begin{assumption}[\textsc{Strong Sparsity}] \label{asp:strongSparsityLevelSplit}
    A target function $m : \{0, 1\}^d \to [-1/2, 1/2]$
  is \textit{$(\beta, r)$-strongly sparse} if $m$ is $r$-sparse with relevant
  features $R$ and the function $\bar{V}$ satisfies:
  $\bar{V}(T \cup \{j\}) - \bar{V}(T) + \beta \le \bar{V}(T \cup \{i\}) - \bar{V}(T)$,
  for all $i \in \calR$, $j \in [d] \setminus \calR$ and
  $T \subset [d] \setminus \{i\}$.
\end{assumption}

  We next need to define the estimator that is produced by a level-split tree
with a set of splits $S$. Given a set of splits $S$, a training set $\train_n$ and
an input $\vecx$ we can define the estimate $m(\vecx; S, \train_n)$ as follows (for simplicity, we use $m_n(\cdot; \cdot)$, $N_n(\cdot; \cdot)$, and $\calT_n(\cdot, \cdot)$ instead of $m(\cdot; \cdot, \train_n)$,
$N(\cdot; \cdot, \train_n)$ and $\calT(\cdot, \cdot; \train_n)$):
\begin{align} \label{eq:estimationGivenSplitsLevelSplitFiniteSample}
  m_n(\vecx; S) = \frac{1}{N_n(\vecx; \calT_n(S, \vecx))} \sum_{j \in [n]} \chara\{\vecx_{\calT_n(S, \vecx)}^{(j)}=\vecx_{\calT_n(S, \vecx)}\} \cdot y^{(j)}
\end{align}
\noindent where $N_n(\cdot; \cdot)$,
$\calT_n(\cdot, \cdot)$ are defined as follows
\begin{align*}   N_n(\vecx; T) = \sum_{j \in [n]} \chara\{\vecx_T^{(j)}=\vecx_T\}, ~~
  \calT_n(S, \vecx) = \argmax_{T \subseteq S, ~ N_n(\vecx; T) > 0} \abs{T}.
\end{align*}

\noindent In words, the function $\calT_n(S, \vecx)$ returns the subset $T$ of the
splits $S$ that we used to create the leaf of the tree that contains $\vecx$, until the leaf that corresponds to $T$ contains at least one training point. The function $N_n(\vecx; T)$ is the number of training points in the leaf that contains $\vecx$, when we split across the coordinates $T$.

  For the presentation of the algorithm we also need the definition of the
\textit{empirical mean square error}, given as set of splits $S$, as follows:
\begin{align} \label{eq:MSEFiniteSampleBinaryLevelSplit}
  L_n(S) & = \frac{1}{n} \sum_{j \in [n]} \p{y^{(j)} - m_n(\vecx^{(j)}; S)}^2 = \frac{1}{n} \sum_{j \in [n]} \p{y^{(j)}}^2 - \frac{1}{n} \sum_{j \in [n]} m^2_n(\vecx^{(j)}; S) \\
  & \triangleq \frac{1}{n} \sum_{j \in [n]} \p{y^{(j)}}^2 - V_n(S). \label{eq:MSEFiniteSampleBinaryLevelSplit:2}
\end{align}
It is easy to see that $V_n$ is a monotone increasing function and $L_n$ is a
monotone decreasing function. We are now ready to present the level-split
algorithm both with and without honesty. For this we use the honesty flag $h$,
where $h = 1$ means we use honesty and $h = 0$ means we don't.
\smallskip

\begin{algorithm}[H]
\caption{Level Split Algorithm} \label{alg:mainLevelSplitFiniteSample}
  \KwIn{maximum number of splits $\log(t)$, a training data set $\train_n$, honesty flag $h$.}
  \KwOut{tree approximation of $m$.}

  $\calV \leftarrow D_{n, x}$

  \lIf{$h = 1$}{Split randomly $\train_n$ in half, $\train_{n/2}$, $\train'_{n/2}$, set $n \leftarrow n/2$, set $\calV \leftarrow D'_{n ,x}$}

  Set $\calP_0 = \{\{0, 1\}^d\}$ the partition associated with the root of the tree.

  For all $1 \leq \ell \leq n$, set $\calP_{\ell} = \varnothing$.

  $\level \leftarrow -1$, $S \leftarrow \emptyset$.

  \While{$\level < \log(t)$}{
    $\level \leftarrow \level + 1$, $\calP_{\level+1}=\emptyset$.

    Select $i \in [d]$ that maximizes $V_n(S \cup \{i\})$ (\textit{see \eqref{eq:MSEPopulationBinaryLevelSplit:2}}) with ties broken randomly.

    \For{\emph{\textbf{all}} $A \in \calP_{\level}$}{
      Cut the cell $A$ to cells
      $A^i_k = \{\vecx \mid \vecx \in A \wedge x_i = k\}$, $k = 0, 1$.

      \uIf{$\abs{\calV \cap A^i_0} >= 1$ \textbf{and} $\abs{\calV \cap A^i_1} >= 1$}{
        $\calP_{\level + 1} \leftarrow \calP_{\level + 1} \cup \{A^i_0, A^i_1\}$
      }
      \Else{
        $\calP_{\level + 1} \leftarrow \calP_{\level + 1} \cup \{A\}$
      }
    }
    $S \leftarrow S \cup \{i\}$
  }
  \Return $(\calP_n, m_n) = \p{\calP_{\level+1}, \vecx \mapsto m_n(\vecx; S)}$ ~ [\textit{see \eqref{eq:estimationGivenSplitsLevelSplitFiniteSample}}].
\end{algorithm}
\smallskip

  Now we are ready to state our main result for the consistency of
shallow trees with level splits as described in Algorithm~\ref{alg:mainLevelSplitFiniteSample}. The proof of this theorem can be found in Appendix~\ref{sec:app:proofs:levelSplitsShallow}.

\begin{theorem} \label{thm:finalRecoverySubmodularIndependentBalancedProperLevelSplitFiniteSample}
    Let $\train_n$ be i.i.d. samples from the non-parametric regression model
  $y = m(\vecx) + \eps$, where $m(\vecx) \in [-1/2, 1/2]$, $\eps \sim \calE$,
  $\Exp_{\eps \sim \calE}[\eps] = 0$ and $\eps \in [-1/2, 1/2]$. Let also
  $S_n$ be the set of splits chosen by Algorithm~\ref{alg:mainLevelSplitFiniteSample}, with input $h = 0$. Then the
  following statements hold.
  \begin{Enumerate}
    \item Under the submodularity Assumption \ref{asp:submodularityLevelSplit},
    assuming that $m$ is $r$-sparse and if we set as input the number of splits to
    be $\log(t) = \frac{C \cdot r}{C \cdot r + 2}\p{\log(n) - \log(\log(d/\delta))}$, then it
    holds that
    \begin{align*}
      \Prob_{\train_n \sim \calD^n}\p{\Exp_{\vecx \sim \calD_x} \b{\p{m(\vecx) - m_n(\vecx; S_n)}^2} > \tilde{\Omega}\p{C \cdot r \cdot \sqrt[C \cdot r + 2]{\frac{\log(d/\delta)}{n}}}} \le \delta.
    \end{align*}
    \item Under the submodularity Assumption \ref{asp:submodularityLevelSplit},
    the independence of features Assumption \ref{asp:independentFeature} and
    assuming that $m$ is $r$-sparse, if $\log(t) = r$ then it holds that
    \begin{align*}
      \Prob_{\train_n \sim \calD^n}\p{\Exp_{\vecx \sim \calD_x} \b{\p{m(\vecx) - m_n(\vecx; S_n)}^2} > \tilde{\Omega}\p{C \cdot \sqrt{\frac{2^r \cdot \log(d/\delta))}{n}}}} \le \delta.
    \end{align*}
    \item If $m$ is $(\beta, r)$-strongly sparse as per Assumption
    \ref{asp:strongSparsityLevelSplit}, and
    $n \ge \tilde{\Omega}\p{\frac{2^r\, \log(d/\delta)}{\beta^2}}$, and
    we set $\log(t) = r$, then we have
    \begin{align*}
      \Prob_{\train_n \sim \calD^n}\p{\Exp_{\vecx \sim \calD_x} \b{\p{m(\vecx) - m_n(\vecx; S_n)}^2} > \tilde{\Omega}\p{\frac{2^r \log(d/\delta) \log(n)}{n}}} \le \delta.
    \end{align*}
  \end{Enumerate}
\end{theorem}

  As we can see the rates, naturally, are better as we make our assumptions stronger.
The fastest rate is achievable when the $(\beta, r)$ strong sparsity holds (even without the submodularity or the independence condition), the
second fastest rate when the features are independent but only the submodularity
holds, and we have the slowest rate when only the submodularity holds and there is
arbitrary correlation between the features.

\subsection{Fully Grown Honest Forests with Level Splits}
\label{sec:levelSplitsFullyGrown}

  In this section we consider the case of fully grown honest trees. For this case
it is necessary to consider forest estimators instead of trees because a fully
grown tree has very high variance. For this reason we use the
\textit{subsampling technique} and \textit{honesty}. For any subset
$\train_s$ of size $s$ of the set of samples $\train_n$, we build one tree estimator
$m(\cdot; \train_s)$ according to Algorithm \ref{alg:mainLevelSplitFiniteSample}
with inputs, $\log(t)$ large enough so that every leaf has two or three samples (fully grown tree),
training set $\train_s$ and $h = 1$. Then our final estimator $m_{n, s}$ can be
computed as follows
\begin{equation} \label{eq:levelSplitFullyGrownSubsampled}
  m_{n, s}(\vecx) = \frac{1}{\binom{n}{s}} \sum_{\train_s \subseteq \train_n, \abs{\train_s} = s} \Exp_{\omega}[m(\vecx; \train_s)].
\end{equation}

 Where $\omega$ is any internal randomness to the tree building algorithm (e.g. the sample splitting). We note that even though we phrase results for the latter estimator, where all sub-samples are being averaged over and the expectation over the randomness $\omega$ is computed, our results carry over to the monte-carlo approximation of this estimator, where only $B$ trees are created, each on a randomly drawn sub-sample and for a random draw of the randomness (see e.g. \cite{WagerA18,oprescu2019orthogonal}), assuming $B$ is large enough (for our guarantees to hold it suffices to take $B=\Theta(d\, n^2)$).

 For the estimator $m_{n, s}$ and under the strong sparsity Assumption
\ref{asp:strongSparsityLevelSplit} we have the following consistency and asymptotic normality theorems. The
proofs of the theorems are presented in Appendices~\ref{sec:app:proofs:levelSplits:variance:fullyGrown} and \ref{app:normality}.

\begin{theorem}
  \label{thm:finalRecoverySubmodularIndependentBalancedProperLevelSplitFiniteSampleFullyGrown}
    Let $\train_n$ be i.i.d. samples from the non-parametric regression model
  $y = m(\vecx) + \eps$, where $m(\vecx) \in [-1/2, 1/2]$, $\eps \sim \calE$,
  $\Exp_{\eps \sim \calE}[\eps] = 0$ and $\eps \in [-1/2, 1/2]$. Let $m_{n, s}$
  be the forest estimator that is built with sub-sampling of size $s$ from the
  training set and where every tree $m(\vecx; \train_s)$ is built using
  Algorithm~\ref{alg:mainLevelSplitFiniteSample}, with inputs: $\log(t)$ large
  enough so that every leaf has two or three samples and $h = 1$.
  Under the Assumption \ref{asp:strongSparsityLevelSplit}, if $R$ is the
  set of relevant features and for every $\vecw \in \{0, 1\}^r$ it holds for
  the marginal probability that
  $\Prob_{\vecz \sim \calD_x}\p{\vecz_R = \vecw} \not\in (0, \zeta/2^r)$ and if
  $s = \tilde{\Theta}\p{\frac{2^r(\log(d/\delta))}{\beta^2} + \frac{2^r \log(1/\delta)}{\zeta}}$
  then it holds that
  \begin{align*}
    \Prob_{\train_n \sim \calD^n}\p{\Exp_{\vecx \sim \calD_x}[(m_{n, s}(\vecx) - m(\vecx))^2] \ge \tilde{\Omega}\p{\frac{2^r \log(1/\delta)}{n}\p{\frac{\log(d)}{\beta^2} + \frac{1}{\zeta}}}} \le \delta.
  \end{align*}
\end{theorem}

  Our next goal is to prove the asymptotic normality of the estimate $m_{n, s}$. To do so we need
that our estimation algorithm treats samples, a-priori symmetrically (i.e. the estimate is
invariant to permutations of the sample indices). Since for simplicity, we have presented
$m_{n, s}$ based on a deterministic algorithm, this might be violated. For this reason, for the
normality result, before computing the $m_{n, s}$ we apply a random permutation $\tau \in S_n$ in
the training set $\train_n$. The permutation $\tau$ is part of the internal randomness $\omega$ of
the algorithm. Given the permutation $\tau$ we denote estimate that we compute by $m_{n, s, \tau}$.
Ideally we would like to compute the expected value of $m_{n, s, \tau}$ over a uniform
choice of $\tau$ which we denote by $\bar{m}_{n, s}$. However this is computationally very
expensive since we need to repeat the estimate for all the $n!$ permutations. Similarly, it might be computationally intensive to compute an average over all possible sub-samples of size $s$. Instead we compute
a Monte Carlo approximation of $\bar{m}_{n, s}$ by sampling $B$ permutations from $S_n$ and sub-samples of size $s$ (where each sub-sample is drawn uniformly from the set of all possible sub-samples of size $s$) and then taking the empirical average of those. We denote this estimator as $m_{n, s, B}$.

\begin{theorem}\label{thm:normality-level}
    Under the same conditions of Theorem
  \ref{thm:finalRecoverySubmodularIndependentBalancedProperLevelSplitFiniteSampleFullyGrown} and
  with the further assumption that:
  $\sigma^2(\vecx) = \Var(y^{(i)}\mid \vecx^{(i)} = \vecx) \geq \sigma^2 > 0$ and that
  for an a priori fixed $\vecx$ it holds $\Prob_{\vecz \sim \calD_x}\p{\vecz_R = \vecx_R} \ge \zeta/2^r$, if we set: $\tilde{\Theta}\left(\frac{2^r (\log(d) + \log(n))}{\beta^2} + \frac{2^r \log(n)}{\zeta}\right) \leq s \leq o(\sqrt{n})$, then for $\sigma_n^2(\vecx) = O\left(\frac{s^2}{n}\right)$, it holds that:
  \begin{equation}
    \sigma_n^{-1}(\vecx)\, (m_{n, s, B}(\vecx) - m(\vecx)) \rightarrow_d N(0, 1)
  \end{equation}
  where $B \ge n^2 \log(n)$.
\end{theorem}

  The latter asymptotic normality theorem enables the construction of
asymptotically valid normal based intervals, using estimates of the variance of
the prediction. These estimates can be constructed either via the bootstrap or
via methods particular to random forests such as the infinitesimal jacknife
proposed in \cite{Wager14}.

\subsection{Level-Splits Algorithm in the Population Model} \label{sec:app:population:levelSplitShallow}

  In this section we present the population versions of the level-splits
algorithm together its convergence analysis. We believe that this section is a
good step before presenting the full proofs of the aforementioned theorems in
the Appendix.

  We start with the presentation of the level-splits algorithm in the population
model.

\begin{algorithm}[H]
\caption{Level Split Algorithm -- Population Model} \label{alg:mainLevelSplitPopulation}
  \KwIn{maximum number of splits $\log(t)$.}

  \KwOut{tree approximation of $m$.}

  Set $\calP_0 = \{\{0, 1\}^d\}$ the partition associated with the root of the tree.

  $\level \leftarrow -1$, $S \leftarrow \emptyset$.

  \While{$\level < \log(t)$}{
    $\level \leftarrow \level + 1$, $\calP_{\level+1}=\emptyset$.

    Select $i \in [d]$ that maximizes $\bar{V}(S \cup \{i\})$ (\textit{see \eqref{eq:MSEPopulationBinaryLevelSplit:2}}) with ties broken randomly.

    \For{\emph{\textbf{all}} $A \in \calP_{\level}$}{
      Cut the cell $A$ to cells
      $A^i_k = \{\vecx \mid \vecx \in A \wedge x_i = k\}$, $k = 0, 1$.

      $\calP_{\level + 1} \leftarrow \calP_{\level + 1} \cup \{A^i_0, A^i_1\}$
    }
    $S \leftarrow S \cup \{i\}$
  }
  \Return $(\bar{\calP}, \bar{m}) = \p{\calP_{\mathrm{level}}, \vecx \mapsto \Exp_{(\vecz, y) \sim \calD}\b{y \mid \vecz \in \bar{\calP}(\vecx)}}$
\end{algorithm}

\begin{definition}[\textsc{Relevant Variables}]
\label{def:relevantVariablesForCellLevelSplit}
    Given a set $S$, we define the set of remaining relevant features
  $\calR(S) = \{i \in [d] \mid \bar{V}(S \cup \{i\}) > \bar{V}(S)\}$.
\end{definition}

\begin{lemma} \label{lem:optimalityWithoutSplitsPopulationLevelSplit}
    For every set $S \subseteq [d]$, under the Assumption
  \ref{asp:submodularityLevelSplit}, if $\calR(S) = \emptyset$, then
  for any $\vecx, \vecx' \in \{0, 1\}^d$ such that
  $\vecx_S = \vecx_{S'}$ it holds that $m(\vecx) = m(\vecx')$.
\end{lemma}

\begin{proof}
    We prove this by contradiction. If there exist
  $\vecx, \vecx' \in \{0, 1\}^d$ such that $\vecx_S = \vecx'_S$ and
  $m(\vecx) \neq m(\vecx')$ then obviously there exists an
  $\tilde{\vecx} \in A$ such that
    \( m(\tilde{\vecx}) \neq \Exp_{\vecx_{-S}}\b{m((\vecx_S, \vecx_{-S}))} \), where $A$
  is the cell of the input space that contains all vectors $\vecz$ with $\vecz_S= \vecx_S$.
  Therefore it holds that $\bar{L}(S) > \bar{L}([d]) = 0$ and hence
  $\bar{V}([d]) > \bar{V}(S)$. Now let's assume an arbitrary
  enumeration $\{i_1, \dots, i_k\}$ of the set
  $S^c = [d] \setminus S$. Because the function
  $\bar{V}$ is monotone and $\bar{V}([d]) > \bar{V}(S)$, there has to
  be a number $j \in [k]$ such that
  $\bar{V}(S \cup \{i_1, \dots, i_j\}) > \bar{V}(S \cup \{i_1, \dots, i_{j - 1}\})$.
  But because of the approximate submodularity of $\bar{V}$ it holds that
  \( \bar{V}(S \cup \{i_1, \dots, i_j\}) - \bar{V}(S \cup \{i_1, \dots, i_{j - 1}\}) \le C \cdot (\bar{V}(S \cup \{i_j\}) - \bar{V}(S)) \),
  which implies that
  $\bar{V}(S \cup \{i_j\}) > \bar{V}(S)$ and this
  contradicts our assumption that $\calR(S) = \emptyset$.
\end{proof}

\begin{theorem} \label{thm:recoverySubmodularIndependentBalancedProperLevelSplits:population}
    Consider the non-parametric regression model $y = m(\vecx) + \eps$, where
  $m : \{0, 1\}^d \to \b{-\frac{1}{2}, \frac{1}{2}}$ is a $r$-sparse function
  and $\eps \sim \calE$, with $\eps \in \b{-\frac{1}{2}, \frac{1}{2}}$ and
  $\Exp\b{\eps} = 0$. Let $\bar{m}$ be the function that the Algorithm
  \ref{alg:mainLevelSplitPopulation} returns with input
  $t \ge \p{1/\eta}^{C \cdot r}$, then under the Assumption
  \ref{asp:submodularityLevelSplit} it holds that
  \[ \Exp_{\vecx}\b{\p{m(\vecx) - \bar{m}(\vecx)}^2} \le \eta. \]
  Moreover under the Independence of Features Assumption
  \ref{asp:independentFeature} if $\log(t) \ge r$ then $\bar{m} = m$.
\end{theorem}

\begin{proof}
    Let $R \subseteq [d]$ be the set of size $\abs{R} = r$ of the relevant
  features of the target function $m$. Let $S$ the set of splits that Algorithm
  \ref{alg:mainLevelSplitPopulation} chooses. Observe that it holds that
  $\bar{L}(S \cup R) = 0$ and hence $\bar{V}(S \cup R) \triangleq V^*$ is
  maximized. Since $m(\vecx) \in [-1/2, 1/2]$, the maximum value of $\bar{V}$
  is $1$.

    For the first part of the theorem, let $\{i_1, \dots, i_r\}$ and be an
  arbitrary enumeration of $R$ and let $R_j = \{i_1, \dots, i_j\}$ then by
  adding and subtracting terms of the form $\bar{V}(S \cup R_j)$ we have the
  following equality
  \[ \p{\bar{V}(S \cup R) - \bar{V}(S \cup R_{r - 1})} + \cdots + \p{\bar{V}(S \cup R_2) - \bar{V}(S \cup \{i_1\})} + \bar{V}(S \cup \{i_1\}) = V^*. \]
  \noindent From the approximate submodularity of $\bar{V}$ we hence have that
  \[ \p{\bar{V}(S \cup \{i_r\}) - \bar{V}(S)} + \cdots + \p{\bar{V}(S \cup \{i_2\}) - \bar{V}(S)} + \p{\bar{V}(S \cup \{i_1\}) - \bar{V}(S)} \ge \frac{V^* - \bar{V}(S)}{C} \]
  \noindent which implies
  \[ \max_{j \in [r]} \p{\bar{V}(S \cup \{i_j\}) - \bar{V}(S)} \ge \frac{V^* - \bar{V}(S)}{C \cdot r}. \]
  \noindent Let $i^{\level}$ be the coordinate that the algorithm chose to
  split at level $\level$. Now from the greedy criterion that Algorithm
  \ref{alg:mainLevelSplitPopulation} uses we get that the coordinate
  $i^{\level}$ that we picked to split was at least as good as the best of the
  coordinates in $R$, hence it holds that
  \[ \bar{V}\p{S \cup \{i^{\level}\}} \ge \bar{V}(S) + \frac{V^* - \bar{V}(S)}{C \cdot r} \]
  which in turn using $L^* \triangleq \bar{L}(S \cup R) = 0$ implies that
  \begin{equation} \label{eq:potentialProgressPopulationLevelSplit}
    \bar{L}\p{S \cup \{i^{\level}\}} \le \bar{L}(S) \p{1 - \frac{1}{C \cdot r}}.
  \end{equation}
  \noindent Again we fix $S_{\level}$ to be the set of splits after the step
  $\level$ of Algorithm \ref{alg:mainLevelSplitPopulation}, it holds that
  \[ \bar{L}\p{S_{\level + 1}} \le \bar{L}\p{S_{\level}} \p{1 - \frac{1}{C \cdot r}}. \]
  \noindent Inductively and using the fact that $m(\vecx) \in [-1/2, 1/2]$
  implies that
  \begin{equation} \label{eq:potentialProgressPopulationLevel}
    \bar{L}\p{S_{\level}} \le \bar{L}\p{\emptyset} \p{1 - \frac{1}{C \cdot r}}^{\level} \le \p{1 - \frac{1}{C \cdot r}}^{\level}.
  \end{equation}
  \noindent Finally from the choice of $t$ we have that for
  $\level = C \cdot r \ln\p{1/\eta}$ it holds $\bar{L}(S_{\level}) \le \eta$ and
  since $\bar{L}(S)$ is a decreasing function of $S$ the first part of the
  theorem follows.

     For the second part, we observe that for any coordinate
  $i \in [d] \setminus R$ and for any $S \subseteq [d]$ it holds that
  $\bar{V}(S \cup \{i\}) - \bar{V}(S) = 0$ and hence the Algorithm
  \ref{alg:mainLevelSplitPopulation} will pick a coordinate in $[d] \setminus R$
  only after it picks all the coordinates in $R$. Hence for $\log(t) \ge r$ we
  have that $\calR(S) = \emptyset$ and from Lemma
  \ref{lem:optimalityWithoutSplitsPopulationLevelSplit} the second part of the
  theorem follows.
\end{proof}

\section{Consistency of Breiman's Algorithm for Sparse Functions} \label{sec:Breiman}

  In this section we present our analysis for the case when the tree construction
algorithm, at every iteration, chooses a different direction to split on at every
cell of the current partition. We start with the necessary definitions to present
the algorithm that we use. We start with the necessary definitions
to present the algorithm that we use. Then, in Section
\ref{sec:app:population:Breiman} we present an analysis of the
population version of the algorithm, that is useful to gain intuition on the
finite sample proof that we present in full detail in the Appendix.

We define the \textit{total expected mean squared error} that is achieved by a
partition $\calP$ of $\{0, 1\}^d$, in the population model, as follows:
\begin{align} \label{eq:MSEPopulationBinaryBreiman}
  \bar{L}(\calP) & \triangleq \Exp_{\vecx \sim \calD_x}\b{\p{m(\vecx) - \Exp_{\vecz \sim \calD_x}\b{m(\vecz) \mid \vecz \in \calP(\vecx)}}^2} \\
  & = \Exp_{\vecx}\b{m^2(\vecx)} - \Exp_{\vecx \sim \calD_x}\b{\p{\Exp_{\vecz \sim \calD_x}\b{m(\vecz) \mid \vecz \in \calP(\vecx)}}^2} \nonumber \\
  & \triangleq \Exp_{\vecx}\b{m^2(\vecx)} - \bar{V}(\calP). \label{eq:MSEPopulationBinaryBreiman:2}
\end{align}
For simplicity, we use the shorthand notation $\bar{V}(\calP, A, i)$ and
$\bar{L}(\calP, A, i)$ to denote $\bar{V}(\calS(\calP, A, i))$ and
$\bar{L}(\calS(\calP, A, i))$. Observe that the function $\bar{L}$ is \textit{increasing} with respect to $\calP$ is the sense that if $\calP' \sqsubseteq \calP$ then
$\bar{L}(\calP') \le \bar{L}(\calP)$ and hence $\bar{V}$ is \textit{decreasing} with
respect to $\calP$.

In order to define the splitting criterion of our algorithm, we will need to define a local version of the mean squared error, locally at a cell $A$, as follows:
\begin{align} \label{eq:cellMSEPopulationBinaryBreiman}
  \bar{L}_{\ell}(A, \calP) & \triangleq \Exp_{\vecx \sim \calD_x} \b{\p{m(\vecx) - \Exp_{\vecz \sim \calD_x}\b{m(\vecz) \mid \vecz \in \calP(x)}}^2 \mid \vecx \in A} \\
  & = \Exp_{\vecx \sim \calD_x}\b{m^2(\vecx) \mid \vecx \in A} - \Exp_{\vecx}\b{\Exp_{\vecz}\b{m(\vecz) \mid \vecz \in \calP(x)}^2 \mid \vecx \in A} \nonumber \\
  & \triangleq \Exp_{\vecx}\b{m^2(\vecx) \mid \vecx \in A} - \bar{V}_{\ell}(A, \calP). \label{eq:cellMSEPopulationBinaryBreiman:2}
\end{align}
For shorthand notation, for any $\calP$ that contains $A$, we will use $\bar{V}_{\ell}(A, I) = \bar{V}_{\ell}(A, \calS(\calP, A, I))$, for any set of directions
$I \subseteq [d]$ (observe that the quantity is independent of the choice of $\calP$, as long as it contains $A$). Similarly, we will use the shorthand notation $\bar{L}_\ell(A, I)$. Finally, we will use the shorthand notation: $\bar{L}_{\ell}(A)$, $\bar{V}_{\ell}(A)$ for $\bar{L}_{\ell}(A, \emptyset)$ and $\bar{V}_{\ell}(A, \emptyset)$ correspondingly.
\smallskip

We now need to define the corresponding property of submodularity and strong sparsity in this more
complicated setting. Inspired by the economics literature we call the analogue of submodularity for this setting the
\textit{diminishing returns} property. Moreover, we call the analogue of strong sparsity, \text{strong partition sparsity}.

\begin{assumption}[\textsc{Approximate Diminishing Returns}] \label{asp:diminishingReturnsBreimanFull}
    For $C \ge 1$, we say that the function $\bar{V}$ has the
  \textit{approximate diminishing returns} property if for any cells $A$, $A'$, any $i \in [d]$
  and any $T \subseteq [d]$ such that $A' \subseteq A$ it holds that
    $\bar{V}_{\ell}(A', T \cup \{i\}) - \bar{V}_{\ell}(A', T) \le C \cdot (\bar{V}_{\ell}(A, i) - \bar{V}_{\ell}(A))$.
\end{assumption}

\begin{assumption}[\textsc{Strong Partition Sparsity}] \label{asp:strongPartitionSparsity}
     A target function $m : \{0, 1\}^d \to [-1/2, 1/2]$
  is \textit{$(\beta, r)$-strongly partition sparse} if $m$ is $r$-sparse with relevant
  features $R$ and the function $\bar{V}$ satisfies:
  $\bar{V}_{\ell}(A, T\cup j) - \bar{V}_{\ell}(A, T) + \beta \le \bar{V}_{\ell}(A, T\cup i) - \bar{V}_{\ell}(A, T)$,
  for all possible cells $A$ and for all $i \in \calR$, $j \in [d] \setminus \calR$.
\end{assumption}

  For some of the results in this section we need to assume that the density or the
marginal density with respect to $\vecx$ is lower bounded by some constant. The reason that we
need this assumption is that Algorithm
\ref{alg:mainBreimanFiniteSample} makes a greedy split decision separately at every leaf of the tree. Therefore we need to
make sure that, at least in the first few important splits, every leaf has enough samples to choose
the correct greedy option. For this reason we define the following assumption on the lower
bound of the marginal density.
\begin{assumption}[\textsc{Marginal Density Lower Bound}] \label{asp:densityLowerBound}
    We say that the density $\calD_x$ is $(\zeta, q)$-lower bounded, if for every set
  $Q \subset [d]$ with size $\abs{Q} = q$ then for every $\vecw \in \{0, 1\}^q$ it holds that
  $\Prob_{\vecx \sim \calD_x}\p{\vecx_Q = \vecw} \ge \zeta/2^q$.
\end{assumption}
Importantly, this assumption will be required only for a subset of our results
that provide faster rates and whenever required, it will be for the case of
$q=r$. Hence, it will for instance be satisfied with $\zeta=(1-\epsilon)^r$
(which is a constant, independent of the dimension $d$) if each coordinate of
$x$ is independent and takes each of the values $0, 1$, with probability at
least $(1-\epsilon)/2$.

    We next need to define the estimator that is defined by a tree that produces
a partition $\calP$ of the space $\{0,1\}^d$. Given a training set $\train_n$
and a cell $A\in \calP$, we define:
\begin{align} \label{eq:estimationGivenPartitionBreimanFiniteSampleCell}
  g_n(A) = \frac{1}{N_n(A)} \sum_{j \in [n]} y^{(j)} \cdot \chara\{\vecx^{(j)} \in A\} = \sum_{j \in [n]} W_n(\vecx^{(j)}; A) \cdot y^{(j)},
\end{align}
\noindent where in both the aforementioned definition $N_n(\cdot)$ and
$W_n(\cdot; \cdot)$ are defined as follows
\begin{align} \label{eq:estimationGivenPartitionBreimanFiniteSample:1}
  N_n(A) = \sum_{j \in [n]} \chara\{\vecx^{(j)} \in A\}, ~~ W_n(\vecx; A) = \frac{\chara\{\vecx \in A\}}{N_n(A)}
\end{align}
\noindent In words, the function $N_n(A)$ is the number of training point in the cell $A$
and $W_n(\vecx; A)$ is the coefficient of the training points that lie in the cell $A$,
when computing the local estimate at $A$. We also define the set $\calZ_n(A)$, as the
subset of the training set $\calZ_n(A) = \{j \mid \vecx^{(j)} \in A\}$. Based on this we
also define the partition $\calU_n(\calP)$ of the training set $\train_n$ as
$\calU_n(\calP) = \{\calZ_n(A) \mid A \in \calP\}$.
Given an input $\vecx$ we define the estimate $m(\vecx; \calP, \train_n)$ as follows (for simplicity, we use $m_n(\cdot; \cdot)$, $N_n(\cdot)$ and
$W_n(\cdot; \cdot)$ instead of $m(\cdot; \cdot, \train_n)$, $N(\cdot; \train_n)$ and
$W(\cdot; \cdot, \train_n)$):
\begin{align} \label{eq:estimationGivenPartitionBreimanFiniteSampleFull}
  m_n(\vecx; \calP) = g_n(\calP(x)),
\end{align}
\smallskip

  For the presentation of the algorithm we also need the definition of the
\textit{empirical mean squared error}, conditional on a cell $A$ and a potential split direction $i$, as
follows.
\begin{align} \label{eq:MSEFiniteSampleBinaryBreimanCellWithSplit}
  L^{\ell}_n(A, i) & \triangleq \sum_{z \in \{0, 1\}} \frac{N_n(A^i_{z})}{N_n(A)} \sum_{j \in \calZ_n(A^i_{z})} \frac{1}{N_n(A^i_z)} \p{y^{(j)} - m_n\p{\vecx^{(j)}; \calP\p{\vecx^{(j)}}}}^2 \\
  & = \frac{1}{N_n(A)} \sum_{j \in Z_n(A)} \p{y^{(j)}}^2 - \sum_{z \in \{0, 1\}} \frac{N_n(A^i_{z})}{N_n(A)} \p{g_n(A^i_z)}^2 \nonumber \\
  & \triangleq \frac{1}{N_n(A)} \sum_{j \in Z_n(A)} \p{y^{(j)}}^2 - V^{\ell}_n(A, i), \label{eq:MSEFiniteSampleBinaryBreimanCellWithSplit:2}
\end{align}

  We are now ready to present the Breiman's tree construction algorithm both with
and without honesty (we use the honesty flag $h$, where $h = 1$ means we use
honesty).

\smallskip

\begin{algorithm}[H]
\caption{Breiman's Tree Construction Algorithm} \label{alg:mainBreimanFiniteSample}
  \KwIn{maximum number of nodes $t$, a training data set $\train_n$, honesty flag $h$.}

  \KwOut{tree approximation of $m$.}

  $\calV \leftarrow D_{n, x}$

  \lIf{$h = 1$}{Split randomly $\train_n$ in half, $\train_{n/2}$, $\train'_{n/2}$, set $n \leftarrow n/2$, set $\calV \leftarrow D'_{n ,x}$}

  Set $\calP_0 = \{\{0, 1\}^d\}$ the partition associated with the root of the tree.

  For all $1 \leq \ell \leq t$, set $\calP_{\ell} = \varnothing$.

  $\level \leftarrow 0$, $n_{\nodes} \leftarrow 1$, $\queue \leftarrow \calP_0$.

  \While{$n_{\nodes}$ < t}{
    \uIf{$\queue = \emptyset$}{
      $\level \leftarrow \level + 1$, $\queue\leftarrow \calP_{\level}$
    }

      Pick $A$ the first element in $\queue$

      \uIf{$\abs{\calV \cap A} \le 1$}{
        $\queue \leftarrow \queue \setminus \{A\}$,
        $\calP_{\level + 1} \leftarrow \calP_{\level + 1} \cup \{A\}$
      }
      \Else{
        Select $i \in [d]$ that maximizes $V^{\ell}_{n}(A, i)$ (\textit{see \eqref{eq:MSEFiniteSampleBinaryBreimanCellWithSplit:2}}) with ties broken randomly

        Cut the cell $A$ to cells
        $A^{i}_k = \{\vecx \mid \vecx \in A \wedge x_i = k\}$, $k = 0, 1$

        $\queue \leftarrow \queue \setminus \{A\}$,
        $\calP_{\level + 1} \leftarrow \calP_{\level + 1} \cup \{A^i_0, A^i_1\}$
      }
  }

  $\calP_{\level+1} \leftarrow \calP_{\level+1} \cup \queue$

  \Return $(\calP_n, m_n) = \p{\calP_{\level+1}, \vecx \mapsto m_n(\vecx; \calP_{\level+1})}$ ~ [\emph{\textit{see \eqref{eq:estimationGivenPartitionBreimanFiniteSampleFull}}}]
\end{algorithm}

\smallskip

We can now state  our main result for the consistency of shallow trees with Breiman's splits as
described in Algorithm \ref{alg:mainBreimanFiniteSample}. The proof of this theorem can be found in
the Appendix \ref{sec:app:proofs:Breiman}. As we can see in Theorem
\ref{thm:finalRecoverySubmodularIndependentBalancedProperBreimanFiniteSample} the
rates are better as we make our assumptions stronger similar to the results
for the level-split algorithm. The main difference between the results in this
section and the results for the level-split algorithm is that for the analysis of
Breiman's algorithm we need to assume that the probability mass function of the
distribution $\calD_x$ is lower bounded by $\zeta/2^d$.

\begin{theorem} \label{thm:finalRecoverySubmodularIndependentBalancedProperBreimanFiniteSample}
    Let $\train_n$ be i.i.d. samples from the non-parametric regression model
  $y = m(\vecx) + \eps$, where $m(\vecx) \in [-1/2, 1/2]$, $\eps \sim \calE$,
  $\Exp_{\eps \sim \calE}[\eps] = 0$ and $\eps \in [-1/2, 1/2]$ with $m$ an $r$-sparse
  function. Let also $\calP_n$ be the partition that the Algorithm
  \ref{alg:mainBreimanFiniteSample} returns with input $h = 0$. Then the following statements
  hold.
  \begin{Enumerate}
    \item Let $q = \frac{C \cdot r}{C \cdot r + 3}\p{\log(n) - \log(\log(d/\delta))}$ and
    assume that the approximate diminishing returns Assumption
    \ref{asp:diminishingReturnsBreimanFull} holds. Moreover if we set the number of nodes $t$
    such that $\log(t) \ge q$, and if we have number of samples
    $n \ge \tilde{\Omega}\p{\log(d/\delta)}$ then it holds that
    \begin{align*}
      \Prob_{\train_n \sim \calD^n}\p{\Exp_{\vecx \sim \calD_x} \b{\p{m(\vecx) - m_n(\vecx; \calP_n)}^2} > \tilde{\Omega}\p{C \cdot r \cdot \sqrt[C \cdot r + 3]{\frac{\log(d/\delta)}{n}}}} \le \delta.
    \end{align*}
    \item Suppose that the distribution $\calD_x$ is a product distribution
    (see Assumption~\ref{asp:independentFeature}) and that Assumption~\ref{asp:diminishingReturnsBreimanFull} holds. Moreover if $\log(t) \ge r$, then it holds
    that
    \begin{align*}
      \Prob_{\train_n \sim \calD^n}\p{\Exp_{\vecx \sim \calD_x} \b{\p{m(\vecx) - m_n(\vecx, \calP_n)}^2} > \tilde{\Omega}\p{\sqrt[3]{C^2 \cdot \frac{2^r \cdot \log(d/\delta))}{n}}}} \le \delta.
    \end{align*}
    \item Suppose that the distribution $\calD_x$ is a product distribution
    (see Assumption~\ref{asp:independentFeature}), that is also $(\zeta, r)$-lower bounded
    (see Assumption~\ref{asp:densityLowerBound}) and that Assumption~\ref{asp:diminishingReturnsBreimanFull} holds. Moreover if $\log(t) \ge r$, then it holds
    that
    \begin{align*}
      \Prob_{\train_n \sim \calD^n}\p{\Exp_{\vecx \sim \calD_x} \b{\p{m(\vecx) - m_n(\vecx, \calP_n)}^2} > \tilde{\Omega}\p{C \cdot \sqrt{\frac{2^r \cdot \log(d/\delta))}{\zeta \cdot n}}}} \le \delta.
    \end{align*}
    \item Suppose that $m$ is $(\beta, r)$-strongly sparse (see Assumption
    \ref{asp:strongPartitionSparsity}) and that $\calD_x$ is $(\zeta, r)$-lower bounded
    (see Assumption \ref{asp:densityLowerBound}). If
    $n \ge \tilde{\Omega}\p{\frac{2^r(\log(d/\delta))}{\zeta \cdot \beta^2}}$, and
    $\log(t) \ge r$, then we have
    \begin{align*}
      \Prob_{\train_n \sim \calD^n}\p{\Exp_{\vecx \sim \calD_x} \b{\p{m(\vecx) - m_n(\vecx, \calP_n)}^2} > \tilde{\Omega}\p{\frac{2^r \log(d/\delta) \log(n)}{n}}} \le \delta.
    \end{align*}
  \end{Enumerate}
\end{theorem}

\subsection{Fully Grown Honest Forests with Breiman's Algorithm}
\label{sec:BreimanFullyGrown}

  In this section we consider the case of fully grown honest trees. As in the case
of level splits we are going to use the subsampling technique and honesty. That is,
for any subset $\train_s$ of size $s$ of the set of samples $\train_n$, we build one
tree estimator $m(\cdot; \train_s)$ according to Algorithm
\ref{alg:mainBreimanFiniteSample} with inputs, $\log(t)$ large enough so that every
leaf has two or three samples, training set $\train_s$ and $h = 1$. Then our final
estimator $m_{n, s}$ can be computed as follows
\begin{equation} \label{eq:BreimanFullyGrownSubsampled}
  m_{n, s}(\vecx) = \frac{1}{\binom{n}{s}} \sum_{\train_s \subseteq \train_n, \abs{\train_s} = s} \Exp_{\omega}[m(\vecx; \train_s)].
\end{equation}

 Where $\omega$ is the internal randomness of the tree building algorithm. For this estimator $m_{n, s}$ and under the strong partition sparsity Assumption
\ref{asp:strongPartitionSparsity} we have the following consistency and asymptotic normality theorems. The proof of the following theorems is presented in Appendices~\ref{sec:app:proofs:Breiman:fullyGrown} and \ref{app:normality}.

\begin{theorem}
  \label{thm:finalRecoverySubmodularIndependentBalancedProperBreimanFiniteSampleFullyGrown}
    Let $\train_n$ be i.i.d. samples from the non-parametric regression model
  $y = m(\vecx) + \eps$, where $m(\vecx) \in [-1/2, 1/2]$, $\eps \sim \calE$,
  $\Exp_{\eps \sim \calE}[\eps] = 0$ and $\eps \in [-1/2, 1/2]$. Suppose that $\calD_x$ is
  $(\zeta, r)$-lower bounded (see Assumption~\ref{asp:densityLowerBound}). Let $m_{n, s}$ be the
  forest estimator that is built with sub-sampling of size $s$ from the training set and where
  every tree $m(\vecx; \train_s)$ is built using the Algorithm \ref{alg:mainBreimanFiniteSample},
  with inputs: $\log(t)$ large enough so that every leaf has two or three samples, training set
  $\train_s$ and $h = 1$. Then using
  $s = \tilde{\Theta}\p{\frac{2^r(\log(d/\delta))}{\zeta \cdot \beta^2}}$ and under Assumption~\ref{asp:strongPartitionSparsity}:
  \begin{align*}
    \Prob_{\train_n \sim \calD^n}\p{\Exp_{\vecx \sim \calD_x}[(m(\vecx) - m_{n, s}(\vecx))^2] > \tilde{\Omega}\left( \frac{2^r \log(d/\delta)}{n \cdot \zeta \cdot \beta^2}\right)} \le \delta.
  \end{align*}
\end{theorem}

  Our next goal is to prove the asymptotic normality of the estimate $m_{n, s}$. As we have already
discussed for the level-splits algorithm, to prove the asymptotic normality we need that our
estimation algorithm treats samples, a priori symmetrically (i.e. the estimate is invariant to
permutations of the sample indices). Since for simplicity, we have presented  $m_{n, s}$ based on a
deterministic algorithm, this might be violated. For this reason, for the normality result, before
computing the $m_{n, s}$ we apply a random permutation $\tau \in S_n$ in the training set
$\train_n$. The permutation $\tau$ is part of the internal randomness $\omega$ of the algorithm.
Given the permutation $\tau$ we denote estimate that we compute by $m_{n, s, \tau}$. Ideally we
would like to compute the expected value of $m_{n, s, \tau}$ over a uniform choice of $\tau$ which
we denote by $\bar{m}_{n, s}$. However this is computationally very expensive since we need to
repeat the estimate for all the $n!$ permutations. Instead we compute a Monte Carlo approximation
of $\bar{m}_{n, s}$ by sampling $B$ permutations from $S_n$ and sub-samples of
size $s$ (where each sub-sample is drawn uniformly from the set of all possible
sub-samples of size $s$) and then taking the empirical average
of those. We denote this estimator as $m_{n, s, B}$.

\begin{theorem}\label{thm:normality-breiman}
    Under the same conditions of
  Theorem~\ref{thm:finalRecoverySubmodularIndependentBalancedProperBreimanFiniteSampleFullyGrown}
  and with the further assumption that:
  $\sigma^2(\vecx) = \Var(y^{(i)}\mid \vecx^{(i)} = \vecx)\geq \sigma^2 > 0$ and that for an a priori fixed
  $\vecx$ it holds $\Prob_{\vecz \sim \calD_x}\p{\vecz_R = \vecx_R} \ge \zeta/2^r$, if we set:
  $\tilde{\Theta}\p{\frac{2^r(\log(d\, n))}{\zeta \cdot \beta^2}} \leq s \leq o(\sqrt{n})$, then
  for $\sigma_n^2(\vecx) = O\left(\frac{s^2}{n}\right)$ it hold that
\begin{equation}
    \sigma_n^{-1}(\vecx)\, (m_{n, s, B}(\vecx) - m(\vecx)) \rightarrow_d N(0, 1).
\end{equation}
  where $B \ge n^2 \log(n)$.
\end{theorem}

\subsection{Population Algorithm of Breiman's Algorithm} \label{sec:app:population:Breiman}

  In this section we present the population versions of the Breiman's
algorithm together its convergence analysis. We believe that this section is a
good step before presenting the full proofs of the aforementioned theorems in
the Appendix.

  We first present the Breiman's algorithm in the population model.

\begin{algorithm}
\caption{Breiman's Tree Construction Algorithm -- Population Model} \label{alg:mainBreimanPopulation}
  \KwIn{maximum number of nodes $t$.}

  \KwOut{tree approximation of $m$.}

  Set $\calP_0 = \{\{0, 1\}^d\}$ the partition associated with the root of the tree.

  $\level \leftarrow 0$, $n_{\nodes} \leftarrow 1$, $\queue \leftarrow \calP_0$.

  \While{$n_{\nodes}$ < t}{
    \If{$\queue = \emptyset$}{
      $\level \leftarrow \level + 1$, $\queue \leftarrow \calP_{\level}$
     }
      Pick $A$ the first element in $\queue$

      Select $i \in [d]$ that maximizes $\bar{V}_{\ell}(A, i)$ (\textit{see \eqref{eq:cellMSEPopulationBinaryBreiman:2}}) with ties broken randomly

      Cut the cell $A$ to cells
      $A^{i}_k = \{\vecx \mid \vecx \in A \wedge x_i = k\}$, $k = 0, 1$

      $\queue \leftarrow \queue \setminus \{A\}$,
      $\calP_{\level + 1} \leftarrow \calP_{\level + 1} \cup \{A^{i}_0, A^{i}_1\}$}

  $\calP_{\level+1} \leftarrow \calP_{\level+1} \cup \queue$.

  \Return $(\bar{\calP}, \bar{m}) = \p{\calP_{\level+1}, \vecx \mapsto \Exp_{(\vecz, y) \sim \calD}\b{y | \vecz \in \calP_{\level+1}(\vecx)}}$
\end{algorithm}

  We now prove some important properties of the functions $\bar{V}$,
$\bar{V}_{\ell}$, $\bar{L}$ and $\bar{L}_{\ell}$ as presented in Equations~\eqref{eq:MSEPopulationBinaryLevelSplit}, \eqref{eq:MSEPopulationBinaryLevelSplit:2}, \eqref{eq:cellMSEPopulationBinaryBreiman} and \eqref{eq:cellMSEPopulationBinaryBreiman:2}.

\begin{lemma} \label{lem:varianceVsLeafVariance}
    For any partition $\calP$ and any cell $A \in \calP$ the following
  hold
  \begin{Enumerate}
    \item $\bar{V}(\calP) = \sum_{A \in \calP} \Prob_{\vecx}\p{\vecx \in A} \cdot \bar{V}_{\ell}(A)$,
    \item $\bar{L}(\calP) = \sum_{A \in \calP} \Prob_{\vecx}\p{\vecx \in A} \cdot \bar{L}_{\ell}(A)$,
    \item $\bar{V}(\calP, A, i) - \bar{V}(\calP) = \Prob_{\vecx}\p{\vecx \in A} \cdot \p{\bar{V}_{\ell}(A, i) - \bar{V}_{\ell}(A)}$,
    \item Under Assumption~\ref{asp:diminishingReturnsBreimanFull} for any two
    partitions $\calP' \sqsubseteq \calP$ and any cells $A$, $A'$, such that
    $A' \subseteq A$ and $A' \in \calP'$, $A \in \calP$, it holds that
    $\bar{V}(\calP', A', i) - \bar{V}(\calP') \le C \cdot (\bar{V}(\calP, A, i) - \bar{V}(\calP))$.
    \item Under Assumption~\ref{asp:diminishingReturnsBreimanFull} for any two
    partitions $\calP' \sqsubseteq \calP$ and any cells $A$, $A'$, such that
    $A' \subseteq A$ and $A' \in \calP'$, $A \in \calP$ and for any
    $T \subseteq [d]$, $i \in [d]$ it holds that
    $\bar{V}(\calP', A', T \cup \{i\}) - \bar{V}(\calP', A', T) \le C \cdot (\bar{V}(\calP, A, i) - \bar{V}(\calP))$.
  \end{Enumerate}
\end{lemma}

\begin{proof}
    The equations (1.), (2.) follow from the definitions of $\bar{V}$,
  $\bar{V}_{\ell}$, $\bar{L}$, $\bar{L}_{\ell}$. For equation 3. we
  have
  \begin{align*}
    \bar{V}(\calP, A, i) - \bar{V}(\calP) & =  \Prob_{\vecx}\p{\vecx \in A^i_0} \cdot \bar{V}_{\ell}(A^i_0) + \Prob_{\vecx}\p{\vecx \in A^i_1} \cdot \bar{V}_{\ell}(A^i_1) - \Prob_{\vecx}\p{\vecx \in A} \cdot \bar{V}_{\ell}(A) \\
    \bar{V}_{\ell}(A, i) - \bar{V}_{\ell}(A) & =  \Prob_{\vecx}\p{\vecx \in A^i_0 \mid \vecx \in A} \cdot \bar{V}_{\ell}(A^i_0) + \Prob_{\vecx}\p{\vecx \in A^i_1 \mid \vecx \in A} \cdot \bar{V}_{\ell}(A^i_1) - \bar{V}_{\ell}(A)
    \intertext{and therefore we have that}
    \bar{V}(\calP, A, i) - \bar{V}(\calP) & = \Prob_{\vecx}\p{\vecx \in A} \cdot \p{\bar{V}_{\ell}(A, i) - \bar{V}_{\ell}(A)}
  \end{align*}
  \noindent and equation (3.) follows. Now from Assumption
  \ref{asp:diminishingReturnsBreimanFull} we have that
  \[ \bar{V}_{\ell}(A', i) - \bar{V}_{\ell}(A') \le C \cdot \p{\bar{V}_{\ell}(A, i) - \bar{V}_{\ell}(A)} \implies \]
  \[ \Prob_{\vecx}\p{\vecx \in A} \cdot \p{\bar{V}_{\ell}(A', i) - \bar{V}_{\ell}(A')} \le \Prob_{\vecx}\p{\vecx \in A} \cdot C \cdot \p{\bar{V}_{\ell}(A, i) - \bar{V}_{\ell}(A)} \]
  \noindent but now since $A' \subseteq A$ this implies
  $\Prob_{\vecx}\p{\vecx \in A'} \le \Prob_{\vecx}\p{\vecx \in A}$ and hence
  \[ \Prob_{\vecx}\p{\vecx \in A'} \cdot \p{\bar{V}_{\ell}(A', i) - \bar{V}_{\ell}(A')} \le \Prob_{\vecx}\p{\vecx \in A} \cdot C \cdot \p{\bar{V}_{\ell}(A, i) - \bar{V}_{\ell}(A)} \]
  \noindent combining the last inequality with equation 3. we get a proof of
  equation (4.). The statement in (5.) can be proven in an identical manner to (4.).
\end{proof}

\begin{definition} \label{def:relevantVariablesForCellBreiman}
    Given a cell $A$, we define the set
  $\calR(A) = \{i \in [d] \mid \bar{V}_{\ell}(A, i) > \bar{V}_{\ell}(A)\}$.
  We also define the set $\calI(A) = \{i \in [d] \mid A^i_0 \subset A\}$
  and $\calO(A) = [d] \setminus \calI(A)$.
\end{definition}

\begin{lemma} \label{lem:optimalityWithoutSplitsPopulationBreiman}
    For every partition $\calP$, under the Assumption
  \ref{asp:diminishingReturnsBreimanFull}, if for every $A \in \calP$ it holds
  that $\calR(A) = \emptyset$, then for any $B \in \calP$, with $\Prob_{x}(x\in B)>0$,
  and $\vecx, \vecx' \in B$ it holds that $m(\vecx) = m(\vecx')$.
\end{lemma}

\begin{proof}
    We prove this by contradiction. Let $B \in \calP$, if there exist
  $\vecx, \vecx' \in B$ such that $m(\vecx) \neq m(\vecx')$ then obviously
  there exists a $\tilde{\vecx} \in B$ such that
  $ m(\tilde{\vecx}) \neq \Exp_{\vecx}\b{m(\vecx) | \vecx \in B}.$
  Therefore it holds that $\bar{L}(\calP) > \bar{L}(\calP, B, [d])$ and
  hence $\bar{V}(\calP, B, [d]) > \bar{V}(\calP)$. Now let's assume an
  arbitrary enumeration $\{i_1, \dots, i_k\}$ of the set
  $\calI(B)$. Because the function $\bar{V}$ is decreasing with
  respect to $\calP$ and $\bar{V}(\calP, B, [d]) > \bar{V}(\calP)$, there
  has to be a number $j \in [k]$ such that
  $\bar{V}(\calP, B, \{i_1, \dots, i_j\}) > \bar{V}(\calP, B, \{i_1, \dots, i_{j - 1}\})$.
  But because of Assumption \ref{asp:diminishingReturnsBreimanFull} of
  $\bar{V}_{\ell}$ and Lemma \ref{lem:varianceVsLeafVariance} it holds that
  \[ 0 < \bar{V}(\calP, B, \{i_1, \dots, i_j\}) - \bar{V}(\calP, B, \{i_1, \dots, i_{j - 1}\}) \le C \cdot \p{\bar{V}(\calP, B, i_j) - \bar{V}(\calP, B)}, \]
  by Lemma \ref{lem:varianceVsLeafVariance} we have that
  $\bar{V}_{\ell}(B, i_j) > \bar{V}_{\ell}(B)$ and together with
  $\Prob_{\vecx \sim \calD_x}(\vecx \in B) > 0$ these contradict the assumption
  $\calR(B) = \emptyset$.
\end{proof}

\begin{theorem} \label{thm:recoverySubmodularIndependentBalancedProperBreiman}
    Consider the non-parametric regression model $y = m(\vecx) + \eps$, where
  $m : \{0, 1\}^d \to \b{-\frac{1}{2}, \frac{1}{2}}$ is a $r$-sparse function and
  $\eps \sim \calE$, with $\eps \in \b{-\frac{1}{2}, \frac{1}{2}}$ and $\Exp\b{\eps} = 0$.
  Let $\bar{m}$ be the function that the Algorithm \ref{alg:mainBreimanPopulation} returns
  with input $t \ge \p{1/\eta}^{C \cdot r}$, then under the Assumption
  \ref{asp:diminishingReturnsBreimanFull} it holds that
  \[ \Exp_{\vecx}\b{\p{m(\vecx) - \bar{m}(\vecx)}^2} \le \eta. \]
  Also, under the Independence of Features Assumption \ref{asp:independentFeature} if
  $t \ge 2^r$ then $\Prob_{\vecx \sim \calD_x}\p{\bar{m}(\vecx) = m(\vecx)} = 1$.
\end{theorem}

\begin{proof}
    When the value of $\level$ changes, then the algorithm considers
  separately every cell $A$ in $\calP_{\level}$. For every such cell $A$
  it holds that $\bar{L}_{\ell}(A, R) = 0$ and hence
  $\bar{V}_{\ell}(A, R) \triangleq V^*(A)$ is maximized. Since
  $m(\vecx) \in [-1, 1]$ it holds that the maximum value of
  $\bar{V}_{\ell}$ is $1$. Now let $\{i_1, \dots, i_r\}$ be an
  arbitrary enumeration of $R$ and let $R_j = \{i_1, \dots, i_j\}$ then
  by adding and subtracting terms of the form $\bar{V}_{\ell}(A, R_j)$
  we have the following equality
  \[ \p{\bar{V}_{\ell}(A, R) - \bar{V}_{\ell}(A, R_{r - 1})} + \cdots + \p{\bar{V}_{\ell}(A, R_2) - \bar{V}_{\ell}(A, i_1)} + \bar{V}_{\ell}(A, i_1) = V^*(A). \]
  \noindent From Assumption~\ref{asp:diminishingReturnsBreimanFull} we have
  that
  \[ \p{\bar{V}_{\ell}(A, i_r) - \bar{V}_{\ell}(A)} + \cdots + \p{\bar{V}_{\ell}(A, i_2) - \bar{V}_{\ell}(A)} + \p{\bar{V}_{\ell}(A, i_1) - \bar{V}_{\ell}(A)} \ge \frac{V^*(A) - \bar{V}_{\ell}(A)}{C} \]
  \noindent which implies
  \[ \max_{j \in [r]} \p{\bar{V}_{\ell}(A, i_j) - \bar{V}_{\ell}(A)} \ge \frac{V^*(A) - \bar{V}_{\ell}(A)}{C \cdot r}. \]
  \noindent Let $i_A^{\level}$ be the coordinate that the algorithm
  chose to split cell $A$ at level $\level$. Now from the greedy
  criterion that we use to pick the next coordinate to split in Algorithm
  \ref{alg:mainBreimanPopulation} we get that for the coordinate
  $i^{\level}_A$ that we picked to split $A$ was at least as good as the
  best of the coordinates in $R$, hence it holds that
  \[ \bar{V}_{\ell}\p{A, i^{\level}_A} \ge \bar{V}_{\ell}(A) + \frac{V^*(A) - \bar{V}_{\ell}(A)}{C \cdot r} \]
  which in turn because
  $L^*(A) \triangleq \bar{L}_{\ell}(A, R) = 0$ implies that
  \begin{equation} \label{eq:potentialProgressPopulationCellSplit:Breiman}
    \bar{L}_{\ell}\p{A, i^{\level}_A} \le \bar{L}_{\ell}(A) \p{1 - \frac{1}{C \cdot r}}.
  \end{equation}
  \noindent Again we fix $\calQ_{\level}$ to be the partition
  $\calP_{\level}$ when $\level$ changed and $\calP_{\level}$ is a full
  partition of $\{0, 1\}^d$. Then because of
  \ref{eq:potentialProgressPopulationCellSplit:Breiman} and Lemma
  \ref{lem:varianceVsLeafVariance} it holds that
  \begin{align}
      \bar{L}\p{\calQ_{\level + 1}} =~& \sum_{A \in \calQ_{\level}} \Prob_{\vecx}\p{\vecx \in A} \bar{L}_{\ell}\p{A, i^{\level}_A} \le \sum_{A \in \calQ_{\level}} \bar{L}_{\ell}\p{A} \p{1 - \frac{1}{C \cdot r}}\\
      =~& \bar{L}\p{\calQ_{\level}} \p{1 - \frac{1}{C \cdot r}}.
  \end{align}
  \noindent Inductively and using the fact that $m(\vecx) \in [-1, 1]$
  implies that
  \begin{equation} \label{eq:potentialProgressPopulationLevel:Breiman}
    \bar{L}\p{\calQ_{\level}} \le \bar{L}\p{\calP_0} \p{1 - \frac{1}{C \cdot r}}^{\level} \le \p{1 - \frac{1}{C \cdot r}}^{\level}.
  \end{equation}
  \noindent Finally from the choice of $t$ we have that
  $\level \ge C \cdot r \cdot \ln\p{1/\eta}$ and hence
  $\bar{L}(\calQ_{\level}) \le \eta$ and hence the first part of the theorem follows.

    For the second part, we observe that for any coordinate $i \in [d] \setminus R$ and for
  any cell $A$ it holds that $\bar{V}_{\ell}(A, i) - \bar{V}_{\ell}(A) = 0$ and hence the
  Algorithm \ref{alg:mainBreimanPopulation} will pick a coordinate in $[d] \setminus R$
  only after it picks all the coordinates in $R$. Hence for $t \ge 2^r$ we have that
  $\calR(A) = \emptyset$ for all the cells $A$ in the output partition and from Lemma
  \ref{lem:optimalityWithoutSplitsPopulationBreiman} the second part of the
  theorem follows.
\end{proof}

\section{Necessity of Approximate Submodularity} \label{sec:app:lowerBound:submodularity}

  Let $m : \{0, 1\}^d \to [-1, 1]$ be a $2$-sparse function such that
$m(\vecx) = a \cdot x_1 + a \cdot x_2 - b \cdot x_1 x_2$ where and assume that
the feature vector $\vecx$ is sampled uniformly at random from $\{0, 1\}^d$ and
there is no noise, i.e. $\eps_i = 0$ for all $i \in [n]$. Then it is easy to see
the following
\begin{align*}
  \bar{V}(\emptyset) & = \p{b - \frac{a}{4}}^2 \\
  \bar{V}(\{1\})     & = \bar{V}(\{2\}) = \frac{5}{4} a^2 - \frac{3}{4} a \cdot b + \frac{1}{8} b^2 \\
  \bar{V}(\{1, 2\})  & = \frac{3}{2} a^2 - a \cdot b + \frac{1}{4} b^2.
\end{align*}
\noindent Additionally from the definition of $m$ we have that for any set
$S \subseteq [d] \setminus \{1, 2\}$ it holds that
$\bar{V}(S) = \bar{V}(\emptyset)$, $\bar{V}(S \cup \{i\}) = \bar{V}(i)$ with
$i \in \{1, 2\}$, and $\bar{V}(S \cup \{1, 2\}) = \bar{V}(1, 2)$. Putting these
together we can see that the smallest possible constant $C$ for which the
approximate submodularity holds is the following
\[ C = \frac{\bar{V}(\{1, 2\}) - \bar{V}(\{2\})}{\bar{V}(\{1\}) - \bar{V}(\emptyset)}, \frac{\bar{V}(\{1, 2\}) - \bar{V}(\{2\})}{\bar{V}(\{1\}) - \bar{V}(\emptyset)} = 2 \cdot \frac{2 a^2 - 2 a \cdot b + b^2}{\p{b - 2 a}^2}. \]
\noindent Hence the constant $C$ is finite for all values of $a$ and $b$ unless
the denominator is zero which corresponds to the case $b = 2 a$ which implies
that $m(\vecx) = a \cdot \p{x_1 \oplus x_2}$, where with $\oplus$ we denote the
xor operator. So for symmetric $2$-sparse functions the only case where our
theorems do not apply is the case where $m$ is a multiple of the xor function.
Nevertheless, we next argue that this is not a limitation of our analysis but a
limitation of the greedy algorithms that we analyze in this paper.

  We next make a formal argument about the failure of the greedy algorithms that
we analyze in this paper when $m(\vecx) = a \cdot \p{x_1 \oplus x_2}$ even in
the case where we have infinite number of samples available. In this case it
holds that $V_n = \bar{V}$ and hence we can assume we can run Algorithm
\ref{alg:mainLevelSplitPopulation}. In that the greedy choice of the coordinate
$i$ is based on the value of $\bar{V}$. But from the above discussion we have
that for any set $S \subseteq [d] \setminus \{1, 2\}$ it holds that
\begin{align*}
  \bar{V}(S \cup \{1\}) = \bar{V}(S \cup \{2\}) = \bar{V}(S \cup \{j\}) = \bar{V}(\emptyset), \quad \text{where $j \in [d] \setminus \{1, 2\}$}.
\end{align*}
\noindent Hence in every step, starting from the first one and until the
coordinates $1$ or $2$ are chosen, all the remaining coordinates achieve the
same $\bar{V}$ value and hence we have a tie among all of them. According to
Algorithm \ref{alg:mainLevelSplitPopulation} this implies that at every level
$i$ if $S_{i - 1}$ is the set of splits that we have already picked and if
$1, 2 \not\in S_{i - 1}$ then for every $j \in [d] \setminus S_{i - 1}$ it holds
that
\[ \Prob(\text{pick $j$ at level $i$}) = \frac{1}{d - (i - 1)}. \]
\noindent Therefore we continue to pick random coordinates until we pick one of
$1$ or $2$ and when this happens we pick the other important coordinate
at the next level and then we have an accurate estimation of the function and
a mean squared error $0$. Before picking coordinates $1$ or $2$ the mean squared
error of our estimation is a constant and hence our estimation is not
consistent. Our next goal is to show that with high probability we will not
pick neither $1$ or $2$ until the level $\sqrt{d}$. Formally we compute for
level $i = \floor{\sqrt{d}}$ the following
\begin{align*}
  \Prob\p{\text{we do not pick $1$ or $2$ until level $\floor{\sqrt{d}}$}} & = \prod_{\ell = 0}^i \p{1 - \frac{2}{d - \ell}} \ge \p{1 - \frac{2}{d - \sqrt{d}}}^i \\
  & \ge 1 - \frac{2 i}{d - \sqrt{d}} \ge 1 - \frac{4}{\sqrt{d}}.
\end{align*}
This means that we need depth at least $\Omega(\sqrt{d})$ to get small mean
square error for this function $m$. Since we need to get to level at least
$\Omega(\sqrt{d})$ this means that the greedy algorithm will take time at least
$\Omega(\sqrt{d})$ and hence it is not possible to achieve a running time of
order $O(\log(d))$ which is the running time that we would expect for a
$2$-sparse function. It is also the running time of our algorithm when the
target function $m$ is $2$-sparse and satisfies approximate submodularity. This
is proves an exponential separation between target functions that satisfy
approximate submodularity and functions that do not satisfy approximate
submodularity.

\bibliographystyle{alpha}
\bibliography{ref}

\clearpage

\appendix
 \newpage

\section{Bias-Variance Decomposition of Shallow Trees}

In this section, we prove a bias-variance decomposition of estimators defined via partitions of the function space; a special case of which are tree-based estimators. Moreover, we prove a bound on the variance via an adaptation of the localized Rademacher complexity analysis, to account for partition-based estimators (which are not necessarily global minimizers of the empirical risk).

\begin{definition} \label{def:piecewiseConstant}
    Given a partition $\calP = \{A_1, \dots, A_k\}$ of $\{0, 1\}^d$ we
  define the set $\calF(\calP)$ of \textit{piecewise constant} functions
  that have the value for every set in $\calP$, i.e.
  \[ \calF(\calP) = \{ m : \{0, 1\}^d \to [-1, 1] \mid \forall A \in P, ~ \forall \vecx, \vecx' \in A, m(\vecx) = m(\vecx') \}. \]
  \noindent If $\calZ = \{\calP_1, \dots, \calP_s\}$ is a family of partitions of
  $\{0, 1\}^d$, then we define $\calF(\calZ)$ to be the union of $\calF(\calP)$
  for all $\calP \in \calZ$.
\end{definition}

For any function class $\calG$, we define the critical radius as any solution to the inequality:
\begin{equation}
    \calR(\delta; \calG) \leq \delta^2
\end{equation}
where $\calR(\delta; \calG)$ is the localized Rademacher complexity, defined as:
\begin{equation}
    \calR(\delta; G) = \Exp_{\train_n \sim \calD^n,~ \vec{\epsilon} \sim \mathrm{Rad}^n}\left[ \sup_{g\in G: \|g\|_2\leq \delta} \left|\frac{1}{s}\sum_i \epsilon_i g(\vec{x}_i, y_i)\right|\right]
\end{equation}
where $\epsilon_i$ are independent Rademacher random variables taking values equi-probably in $\{-1, 1\}$. Moreover, we define the star-hull of the function class as: $\sh(G)=\{\kappa \, g: g\in G, \kappa\in [0,1]\}$.

\begin{lemma}[Bias-Variance Decomposition]\label{lem:bias-variance}
    Consider a random mapping $\calA(\train_n)$, that maps a set of training
  samples into a partition of the space $\{0, 1\}^d$ and let $\calP_n\triangleq \calA(\train_n)$, denote the random partition output by $\calA$ (where randomness is over $\calD_n$ and the randomness of $\calA$). Let $\bar{\calZ}$ be the image of this
  mapping, i.e. the union of the supports of the random variable $\calA(\train_n)$ for all possible $\train_n$. Suppose that an estimator
  $\hat{m}$, minimizes the empirical mean squared error among all piece-wise constant functions
  $f \in \calF(\calP_n)$, i.e.:
    \begin{equation}
        \hat{m} = \argmin_{f\in \calF(\calP_n)} \sum_{i \in [n]} \p{y^{(i)} - f(\vecx)}^2.
    \end{equation}
    Let $\bar{\calF}=\calF(\bar{\calZ})$ and let $\delta_n^2 \geq \Theta\p{\frac{\log(\log(n))}{n}}$
    be an upper bound on the critical radius of $\sh\p{\bar{\calF}-m}$. Moreover, let
    $\tilde{m}_n(\cdot) = \Exp_{z\sim D_x}\b{m(z)\mid z\in \calP_n(\cdot)}$. Then for a universal
    constant $C$, w.p. $1-\zeta$:
    \begin{equation} \label{eq:mainBeforeFinal}
      \Exp_{\vecx\sim \calD_x}\b{\p{\hat{m}(x) - m(x)}^2} \leq C\p{\p{\delta_n + \sqrt{\frac{\log(1/\zeta)}{n}}}^2 + \Exp_{\vecx\sim \calD_x}\b{\p{\tilde{m}_n(x) - m(x)}^2}}
    \end{equation}
\end{lemma}

\subsection{Proof of Lemma~\ref{lem:bias-variance}}\label{sec:app:proofs:biasVariance}

\paragr{Notation.} To simplify the exposition we introduce here some
notation that we need for our local Radermader complexity analysis. We
define $c(\vecx, y; m)$ to be represent the error of the sample $(\vecx, y)$
according to the function $m$. In our setting we have that
$c(\vecx, y; m) = (y - m(\vecx))^2$ and we may drop the argument $m$ from
$c$ when $m$ is clear from the context. We define $\calD_{\calX}$ to be the
marginal with respect to $\vecx$ of the distribution of $\calD$. Also we use
the notation
$\norm{m}_2 = \sqrt{\Exp_{\vecx \sim \calD_{\calX}}\b{m(\vecx)}}$. In this
section we sometimes use $\hat{m}$ is place for $m_n$ but they have the same
meaning. We next give a formal definition of the set of piece-wise constant
functions. Finally, let $\Exp_{n}$ denote the empirical expectation with respect
to $\train_n$.

For simplicity of notation, we define
$\calF_n \triangleq \calF(\calP_n)$ and
$\bar{\calF} \triangleq \calF(\bar{\calZ})$.
From the definition of
$\hat{m}$ we have:
\begin{align} \label{eq:localRademacherPiecewiseLinearApproximation}
  \sum_{i \in [n]} \p{y^{(i)} - \hat{m}(\vec{x}^{(i)})}^2  \le \inf_{f \in \calF_n} \sum_{i \in [n]} \p{y^{(i)} - f(\vecx)}^2.
\end{align}

\noindent Now for any function $g : \{0, 1\}^d \to \R$ we have that
\begin{align}
  \Exp_{(\vecx, y) \sim \calD}\b{\p{y - g(\vecx)}^2 - \p{y - m(\vecx)}^2} & = \Exp_{(\vecx, y) \sim \calD} \b{g^2(\vecx) - m^2(\vecx) - 2 y (g(\vecx) + m(\vecx))} \nonumber \\
  & = \Exp_{\vecx \sim \calD_{\calX}} \b{g^2(\vecx) - m^2(\vecx) - 2 \Exp\b{y \mid \vecx} (g(\vecx) + m(\vecx))} \nonumber \\
  & = \Exp_{\vecx \sim \calD_{\calX}} \b{g^2(\vecx) + m^2(\vecx) - 2 m(\vecx) g(\vecx)} \nonumber\\
  & = \norm{g - m}_2^2. \label{eq:localRademacherDistancetoUnderlineFunction}
\end{align}

\noindent If we plug in $g = \hat{m}$ in
\eqref{eq:localRademacherDistancetoUnderlineFunction} then we get that
\begin{align} \label{eq:localRademacherDistanceEmpiricalFunctiontoUnderlinefunction}
  \norm{\hat{m} - m}_2^2 & = \Exp_{(\vecx, y) \sim \calD} \b{c(\vec{x}, y; \hat{m}) - c(\vec{x}, y; m)}
\end{align}

\noindent We define also the following function
\begin{equation}
    \tilde{m}_n = \argmin_{f \in \calF_n} \Exp_{\vec{x} \sim \calD_X}[(f(\vec{x}) - m(\vec{x}))^2]. \end{equation}
Observe that the solution to this optimization takes the form:
\begin{equation}
    \tilde{m}_n(\cdot) = \Exp_{\vecz\sim \calD_x}\b{m(\vecz) \mid \vecz \in \calP_n(\cdot)}
\end{equation}

Conditional on the training set $\train_n$, we have:
\begin{align*}
    \norm{\hat{m} - m}_2^2 = & \Exp_{(\vecx, y) \sim \calD} \b{c(\vec{x}, y; \hat{m}) - c(\vec{x}, y; m)} \\
    = & \Exp_{(\vecx, y) \sim \calD} \b{c(\vec{x}, y; \hat{m}) - c(\vec{x}, y; \tilde{m}_n)}  + \Exp_{(\vecx, y) \sim \calD}\left[c(\vecx, y; \tilde{m}_n)- c(\vecx, y; m)\right] \\
    = & \Exp_{(\vec{x}, y) \sim \calD}\left[c(\vec{x}, y; \hat{m}) - c(\vec{x}, y; \tilde{m}_n)\right]  + \norm{\tilde{m}_n - m}_2^2 \tag{by \eqref{eq:localRademacherDistancetoUnderlineFunction}}
\end{align*}

  Now we can relate the population generalization error with the empirical.
\begin{align*}
    \Exp_{(\vec{x}, y) \sim \calD} \left[c(\vec{x}, y; \hat{m}) - c(\vec{x}, y; \tilde{m}_n)\right] = & \Exp_n \left[c(\vec{x}, y; \hat{m}) - c(\vec{x}, y; \tilde{m}_n)\right] \\
    & + \Exp_{(\vec{x}, y) \sim \calD} \left[c(\vec{x}, y; \hat{m}) - c(\vec{x}, y; \tilde{m}_n)\right] - \Exp_n \left[c(\vec{x}, y; \hat{m}) - c(\vec{x}, y; \tilde{m}_n)\right]
\end{align*}
Since by definition $\hat{m}$ minimizes the empirical loss over $\calF_n$ and since
$\tilde{m}_n \in \calF_n$ the first term is non-positive and hence
\begin{align*}
    \Exp_{(\vec{x}, y) \sim \calD} \left[c(\vec{x}, y; \hat{m}) - c(\vec{x}, y; \tilde{m}_n)\right] \le & \Exp_{(\vec{x}, y) \sim \calD} \left[c(\vec{x}, y; \hat{m}) - c(\vec{x}, y; \tilde{m}_n)\right] - \Exp_n \left[c(\vec{x}, y; \hat{m}) - c(\vec{x}, y; \tilde{m}_n)\right] \\
    = & \Exp_{(\vec{x}, y) \sim \calD}[c(\vec{x}, y; \hat{m}) - c(\vec{x}, y; m)] - \Exp_n[c(\vec{x}, y; \hat{m}) - c(\vec{x}, y; m)] + \\
    & + \Exp_{(\vec{x}, y) \sim \calD}[c(\vec{x}, y; m) - c(\vec{x}, y; \tilde{m}_n) ] - \Exp_n[c(\vec{x}, y; m) - c(\vec{x}, y; \tilde{m}_n)]
\end{align*}

Observe that the space of functions $\calF_n$ is a subset of the space of functions
$\bar{\calF}$, which is independent of $\train_n$. Thus it suffices to prove a
uniform convergence tail bound for all functions in the latter space.

  By Lemma 7 of \cite{foster2019orthogonal}, we have that if
$\delta_n^2\geq \Theta\left(\frac{\log(\log(n))}{n}\right)$ is any solution to the
inequality:
\begin{equation}
    \calR(\delta; \sh(\bar{\calF} - m)) \leq \delta^2
\end{equation}
then for some universal constant $C$, we have that with probability $1 - \delta$ for all
$f \in \bar{\calF}$ it holds that
\begin{align*}
    & \left|\Exp_{(\vec{x}, y) \sim \calD}[c(\vec{x}, y; f) - c(\vec{x}, y; m)] - \frac{1}{n}\sum_{i = 1}^n c(\vec{x}_i, y_i; f) - c(\vec{x}_i, y_i; m)\right| \\
    & ~~~~~~~~~~~~~~~~~~~~~~~~~~~~~~~~~~~~~~~~~~~~~~~~~~~~~~~~~~~~~~~~~~~~~~~~~~~~~~~~~~~~~~~ \leq C\,(\delta_n + \zeta) \norm{f - m}_2 + C\,(\delta_n + \zeta)^2
\end{align*}
for $\zeta = \sqrt{\frac{\log(1/\delta)}{n}}$. Applying the same
lemma to loss function $-c$, we get
\begin{align*}
    & \left|\Exp_{(\vec{x}, y) \sim \calD}[c(\vec{x}, y; m) - c(\vec{x}, y; f)] - \frac{1}{n}\sum_{i = 1}^n c(\vec{x}_i, y_i; m) - c(\vec{x}_i, y_i; f)\right| \\
    & ~~~~~~~~~~~~~~~~~~~~~~~~~~~~~~~~~~~~~~~~~~~~~~~~~~~~~~~~~~~~~~~~~~~~~~~~~~~~~~~~~~~~~~~ \leq C\, (\delta_n + \zeta) \norm{f - m}_2 + C\, (\delta_n + \zeta)^2.
\end{align*}

Applying the first inequality for $f = \hat{m}$ and the second for
$f = \tilde{m}_n$ and taking a union bound over both events, we have that for
$\zeta = \sqrt{\frac{\log(2/\delta)}{n}}$, w.p. $1-\delta$:
\begin{align*}
    & \Exp_{(\vec{x}, y) \sim \calD}[c(\vec{x}, y; \hat{m}) - c(\vec{x}, y; m)] - \Exp_n[c(\vec{x}, y; \hat{m}) - c(\vec{x}, y; m)] \\
    & ~~~~~~~~~~~~~~~~~~~~~~~~~~~~~~~~~~~~~~~~~~~~~~~~~~~~~~~~~~~~~~~~~~~~~~~~~~~~~~~~~~~~~~~~~~ \le C\, (\delta_n + \zeta) \norm{\hat{m} - m}_2 + C\, (\delta_n + \zeta)^2 \\
    & \Exp_{(\vec{x}, y) \sim \calD}[c(\vec{x}, y; m) - c(\vec{x}, y, \tilde{m}_n)] - \Exp_n[c(\vec{x}, y; m) - c(\vec{x}, y, \tilde{m}_n)] \\
    & ~~~~~~~~~~~~~~~~~~~~~~~~~~~~~~~~~~~~~~~~~~~~~~~~~~~~~~~~~~~~~~~~~~~~~~~~~~~~~~~~~~~~~~~~~~ \le C\, (\delta_n + \zeta) \norm{\tilde{m}_n - m}_2 + C\, (\delta_n + \zeta)^2
\end{align*}

Combining all these we have, w.p. $1-\delta$ over the training set:
\begin{align*}
    \norm{\hat{m} - m}_2^2 \le & C\, (\delta_n + \zeta) (\norm{\hat{m} - m}_2 + \norm{\tilde{m}_n - m}_2) + 2C\, (\delta_n + \zeta)^2 + \norm{\tilde{m}_n - m}_2^2 \\
    \le & C~\left(C\, (\delta_n+\zeta)^2 + \frac{1}{4C} (\norm{\hat{m} - m}_2 + \norm{\tilde{m}_n - m}_2)^2\right) + 2C\, (\delta_n+\zeta)^2 + \norm{\tilde{m}_n - m}_2^2\\
    \le & C\left(2 + C\right) (\delta_n+\zeta)^2 + \frac{1}{2} \left(\norm{\hat{m} - m}_2^2 + \norm{\tilde{m}_n - m_0}_2^2\right) + \norm{\tilde{m}_n - m}_2^2
\end{align*}
Re-arranging the last inequality, yields:
\begin{equation}
     \norm{\hat{m} - m}_2^2 \leq 2 C\left(1 + C\right) (\delta_n + \zeta)^2 + 3 \norm{\tilde{m}_n - m}_2^2
\end{equation}

\subsection{Critical Radius of Shallow Trees}

\paragr{VC dimension of $\bar{\calF}$.} We now show that when the partition
$\calP_n\triangleq \calA(\train_n)$ is defined by a tree with $t$ leafs, then the
function class $\bar{\calF}$ is a VC-subgraph class. Let $\bar{\calF}(\zeta)$
denote the subgraph of $\bar{\calF}$ at any level $\zeta$ (i.e. the space of
binary functions
$\bar{\calF}(\zeta)\triangleq \{ x \rightarrow \chara \{f(x) >\zeta\}: f \in \bar{\calF}\}$.
To show that $\bar{\calF}$ is VC-subgraph with VC dimension $v$, we need to show
that $\bar{\calF}(\zeta)$ has VC dimension at most $v$.

Observe that the number of all possible observationally equivalent functions that the function class $\bar{\calF}(\zeta)$ can output on $n$ samples is at most $(n d)^t\, 2^t$. This follows by the following argument: the number of possible functions is equal to the number of possible partitions of the $n$ samples that can be induced by a tree with $t$ leafs, multiplied by the number of possible binary value assignments at the leafs. The latter is $2^t$. The former is at most $(n d)^t$.\footnote{This can be shown by induction; Let $S_{s,t}$ be the number of possible partitions induced by a tree with $t$ leafs. Then $S_{s, 1}=1$ and in order to create a tree with $t$ leafs, we need to take a tree with
$t-1$ leafs and expand the leaf that one of the $s$ samples belongs to along the dimension of one of the $d$ features. Thus we have $d\, s$ total choices, leading to $S_{s, t}=S_{s, t-1}\, d\, s = (d\, s)^{t}$}

On the other hand, the set of all binary functions on $n$ points is $2^n$. Thus for the function class $\bar{F}(\zeta)$ to be able to shatter a set of $n$ points, it must be that $2^n \leq (2\,n\, d)^t$. Equivalently:
\begin{equation}
    n \leq t \log(2\,d) + t\,\log(n) \Rightarrow m \leq 4\, t\log(t) + 2\, t \log(2\,d)=O\left(t\log(d\, t)\right)
\end{equation}
Thus we get that the function class $\bar{\calF}(\zeta)$ has VC dimension at most $v = O\left(t\log(d\, t)\right)$. Thus $\bar{\calF}$ is a VC-subgraph class of VC dimension $v = O\left(t\log(d\, t)\right)$.
\medskip

\paragr{Bounding the critical radius.} We will use the fact that the critical radius of $\sh\p{\bar{\calF}-m}$ is $O(\delta_n)$, where $\delta_n$ is any solution to the inequality (see e.g. \cite{wainwright_2019}):
\begin{equation}
\int_{\delta^2/8}^{\delta} \sqrt{\frac{\calH_2(\epsilon, \sh(\bar{\calF} - m)_{\delta,n}, z_{1:s})}{n}} \leq \delta^2
\end{equation}
where $G_{\delta,n} = \{g\in G: \|g\|_n=\sqrt{\frac{1}{n}\sum_i g(z_i)^2}\leq \delta\}$ and $H_2(\epsilon, G, z_{1:n})$ is the logarithm of the size of the smallest $\epsilon$-cover of $G$, with respect to the empirical $\ell_2$ norm $\|g\|_n$ on the samples $z_{1:n}$.

First observe that the star hull can only add at most a logarithmic extra factor to the metric entropy, by a simple discretization argument on the parameter $\delta$, i.e.:
\begin{align*}
    \calH_2(\epsilon, \sh(\bar{\calF} - m)_{\delta,n}, z_{1:n}) \leq~& \calH_2(\epsilon/2, (\bar{\calF} - m)_{\delta,n}, z_{1:n}) + \log(2\,\sup_{f\in (\bar{\calF} - m)_{\delta,n}} \|f\|_{2,n}/\epsilon)\\
    \leq~& \calH_2(\epsilon, (\bar{\calF} - m)_{\delta,n}, z_{1:n}) + \log(2\,\delta /\epsilon)
\end{align*}

Moreover, observer that the metric entropy of $(\bar{\calF} - m)_{\delta,n}$ is at most the metric entropy of $\bar{\calF} - m$, which is at most the metric entropy of $\bar{\calF}$ (since $m$ is a fixed function). Thus it suffices to bound the metric entropy of $\bar{\calF}$.

Theorem 2.6.7 of \cite{VanDerVaartWe96} shows that for any VC-subgraph class $G$ of VC dimension $v$ and bounded in $[-1,1]$ we have:
\begin{equation}
    \calH_2(\epsilon, G, z_{1:n}) = O( v(1+\log(1/\epsilon)))
\end{equation}
This implies that the critical radius of $\sh(\bar{\calF} - m)$ is of the order of any solution to the inequality:
\begin{align*}
\int_{\delta^2/8}^{\delta} \sqrt{\frac{ v(1+\log(2/\epsilon) + \log(2\delta/\epsilon)}{n}} \leq \delta^2
\end{align*}
The left hand side is of order $\delta \sqrt{\frac{v(1+\log(1/\delta))}{n}}$. Thus the critical radius needs to satisfy for some constant $D$, that:
\begin{equation}
    \delta \geq D \sqrt{\frac{v(1+\log(1/\delta))}{n}}
\end{equation}
This is satisfied for:
\begin{equation}
    \delta = \Theta\left(  \sqrt{\frac{v(1+\log(n))}{n}} \right)
\end{equation}
Thus the critical radius of $\sh(\bar{\calF} - m)$ is $\Theta\left(  \sqrt{\frac{t\,\log(d\,t) \, (1+\log(n))}{n}} \right)$.

\begin{corollary}[Critical Radius of Shallow Trees]\label{cor:criticalradiusShallow}
Let $\calA$, be a function that maps any set of training samples $\train_n$ into a partition $\calP_n$ of $\{0,1\}^d$, defined in the form of a binary tree with $t$ leafs. Then the critical radius of $\sh(\bar{\calF} - m)$, as defined in Lemma~\ref{lem:bias-variance} is $\Theta\left(  \sqrt{\frac{t\,\log(d\,t) \, (1+\log(n))}{n}} \right)$.
\end{corollary}

\clearpage

\section{Bias-Variance Decomposition of Deep Honest Forests}

In this section we offer a type of bias-variance decomposition for the case of deep honest forests. We will couple this result with bounds on the bias part in subsequent sections to get our final mean squared error bounds for deep honest forests.

To define our bias-variance decomposition, we first need the notion of diameter of a cell $A$ with respect to the value of
$m(\vecx)$.
\begin{definition}[Value-Diameter of a Cell] \label{def:valueDiameter}
    Given set $B \subseteq \{0, 1\}^d$ we define the subset $\bar{B} \subseteq B$
  such that $\vecx \in \bar{B}$ if and only if
  $\Prob_{\vecz \sim \calD_x}\p{\vecz \in B} > 0$. The value-diameter $\Delta(B)$
  of $B$ to be equal to
  $\Delta_m(B) = \max_{\vecx, \vecy \in \bar{B}} \p{m(\vecx) - m(\vecy)}^2$.
  For any partition $\calP$ of $\{0, 1\}^d$ we define the value-diameter of the
  partition $\calP$ to be
  $\Delta_m(\calP) = \max_{A \in \calP} \Prob_{\vecx \sim \calD_x}(\vecx \in A) \cdot \Delta_m(A)$.
\end{definition}

We can now show that the mean squared error of deep honest forests is upper bounded by the expected diameter of the returned partition by any single tree, plus $O(s/n)$. The expected diameter can be thought of as the bias of the forest and the latter term as the variance.

\begin{lemma} \label{lem:biasVarianceDeepHonestForests}
Consider any forest with $B$ trees, where each tree is built with honesty and on a random sub-sample of size $s$. Let $\epsilon(s) = \Exp_{\vecx \sim \calD_x, \train_{s/2} \sim \calD^{s/2}}\left[\Delta_m(\calP_{s/2}(\vecx))\right]$. Then
\begin{equation*}
\Prob_{\train_n \sim \calD^n}\p{\Exp_{\vecx \sim \calD_x}[(m_{n, s}(\vecx) - m(\vecx))^2] \leq O\left(\frac{s\, \log(n/\delta)}{n}\right) + \epsilon(s)} \ge 1 - \delta.
\end{equation*}
\end{lemma}
\begin{proof}
  We start with defining the following function
\begin{align} \label{eq:definitionOfExpectedTree}
  \bar{m}_{s}(\vecx) = \Exp_{\train_n \sim \calD^n} \b{m_{n, s}(\vecx)} = \Exp_{\train_s \sim \calD^s} \b{m_{s}(\vecx)}.
\end{align}

\noindent For mean squared error of $m_{n, s}$ we have:
\begin{align*}
    \Exp_{\vecx \sim \calD_x, \train_n \sim \calD^n}[(m_{n, s}(\vecx) - m(\vecx))^2] = & \Exp_{\vecx \sim \calD_x, \train_n \sim \calD^n}[(m_{n, s}(\vecx) - \bar{m}_s(\vecx))^2] + \Exp_{\vecx \sim \calD_x}[(\bar{m}_s(\vecx) - m(\vecx))^2]
\end{align*}

\noindent The first part we know that it is bounded for every $\vecx$ and with
exponential tails due to concentration of U-statistics
\cite{Hoeffding94, PeelASR10}., i.e. for any fixed $\vecx$ with probability
$1 - \delta$ it holds that
\begin{equation}
    (m_{n, s}(\vecx) - \bar{m}_s(\vecx))^2 \leq O\left(\frac{s\, \log(1/\delta)}{n}\right)
\end{equation}
Thus integrating over $\vecx \sim \calD_x$, we have:
\begin{equation} \label{eq:fullyGrownLevelSplitsProof:1}
    \Prob_{\vecx \sim \calD_x, \train_n \sim \calD^n}\p{(m_{n, s}(\vecx) - \bar{m}_s(\vecx))^2 \geq O\left(\frac{s\, \log(1/\delta)}{n}\right)} \leq \delta.
\end{equation}
Let $T(\vecx) = (m_{n, s}(\vecx) - \bar{m}_s(\vecx))^2$ and
$\epsilon = \Theta\p{\frac{s\log(1/\delta)}{n}}$. Suppose that with probability more than
$n\, \delta$ over the training set $\train_n \sim \calD^n$, we had that
$\Prob_{\vecx \sim \calD_x}(T(\vecx) \geq \epsilon \mid \train_n) \geq 1/n$. Then we have
that
$\Prob_{\vecx \sim \calD_x, \train_n \sim \calD^n}\p{T(\vecx) \geq \epsilon}\geq \delta$,
which contradict \ref{eq:fullyGrownLevelSplitsProof:1}. Thus we know that with probability
$1 - n\, \delta$ over $\train_n \sim \calD^n$ it holds that
$\Prob_{\vecx \sim \calD_x}\p{T(\vecx) \geq \epsilon \mid \train_n} \leq 1/n$. Hence with
probability $1 - n\, \delta$ over the training set $\train_n \sim \calD^n$ it holds that
\begin{equation*}
  \Exp_{\vecx \sim \calD_x}[(m_{n, s}(\vecx) - \bar{m}_s(\vecx))^2] \leq \epsilon + \Prob_{\vecx\sim \calD_x}\p{T(\vecx) \geq \epsilon} \leq \epsilon + \frac{1}{n} = O\left(\frac{s\, \log(1/\delta)}{n}\right) + \frac{1}{n}
\end{equation*}
Setting $\delta' = n\, \delta$, we have the following
\begin{equation} \label{eq:fullyGrownLevelSplitsProof:2}
  \Prob_{\train_n \sim \calD^n}\p{\Exp_{\vecx \sim \calD_x}[(m_{n, s}(\vecx) - \bar{m}_s(\vecx))^2] \leq O\left(\frac{s\, \log(n/\delta')}{n}\right)} \ge 1 - \delta'.
\end{equation}

For the bias term we define for simplicity
$w^{(j)}(\vecx) = \frac{\chara\{\vecx \in \calP_n(\vecx^{(j)})\}}{N_n(\calP_n(\vecx^{(j)}))}$
and hence $m_{s}(\vecx) = \sum_{i = 1}^s w^{(j)}(\vecx) y^{(j)}$ and we have:
\begin{align*}
  \Exp_{\vecx \sim \calD_x}[(\bar{m}_s(\vecx) & - m(\vecx))^2] =  \Exp_{\vecx \sim \calD_x}\left[\left(\Exp_{\train_n \sim \calD^n}[m_{n,s}(\vecx)] - m(\vecx)\right)^2\right] \\
    =~& \Exp_{\vecx \sim \calD_x} \left[\left(\Exp_{\train_s \sim \calD^s}\left[\sum_{j = 1}^s w^{(j)}(\vecx)\, (y^{(j)} - m(\vecx^{(j)}))\right] +
    \Exp_{\train_n \sim \calD^n}\left[\sum_{j = 1}^s w^{(j)}(\vecx)\, (m(\vecx^{(j)}) - m(\vecx))\right]\right)^2\right]
\end{align*}
Due to honesty $w^{(j)}(\vecx)$ is independent of $y^{(j)}$ and we have
that the first term is equal to $0$ by a tower law. Thus we have:
\begin{align*}
     \Exp_{\vecx \sim \calD_x}[(\bar{m}_s(\vecx) - m(\vecx))^2] =~& \Exp_{\vecx \sim \calD_x}\left[\Exp_{\train_n \sim \calD^n}\left[\sum_{j = 1}^s w^{(j)}(\vecx)\, (m(\vecx^{(j)}) - m(\vecx))\right]^2\right]\\
     \leq~& \Exp_{\vecx \sim \calD_x, \train_n \sim \calD^n}\left[\left(\sum_{j = 1}^s w^{(j)}(\vecx) (m(\vecx^{(j)}) - m(\vecx))\right)^2\right]\\
     \leq~& \Exp_{\vecx \sim \calD_x, \train_{s/2} \sim \calD^{s/2}}\left[\Delta_m(\calP_{s/2}(\vecx))\right]
\end{align*}

\end{proof} \newpage

\section{Proofs for Level Splits Algorithms} \label{sec:app:proofs:levelSplitsShallow}

  In this section we present the proofs of
Theorem~\ref{thm:finalRecoverySubmodularIndependentBalancedProperLevelSplitFiniteSample}
and Theorem~\ref{thm:finalRecoverySubmodularIndependentBalancedProperLevelSplitFiniteSampleFullyGrown}.
We start with a proof about the bias of the trees that are produced by Algorithm
\ref{alg:mainLevelSplitFiniteSample} and then we show how we can bound the variance term.
First, we define the set $\calK(S; \train_n)$, or for simplicity $\calK_n(S)$, as the
partition of the samples induced by the set of splits $S$, i.e.
$\calK_n(S) = \set{\calK_n(S, \vecz) \mid \vecz \in \{0, 1\}^d}$ where we define
$\calK_n(S, z)$ as the following set
$\calK_n(S, z) = \set{j \mid \vecx^{(j)}_{\calT_n(S, \vecz)} = \vecz_{\calT_n(S, \vecz)}, ~ j \in [n]}$.
Observe that $\calK_n$ is the same as the partition of the samples implied by the partition
$\calP_n$ of the space $\{0,1\}^d$, returned by Algorithm~\ref{alg:mainLevelSplitFiniteSample}.

\subsection{Bounding The Bias} \label{sec:app:proofs:levelSplitsShallow:Bias}

  We first prove a technical lemma for the concentration of the function $V_n$
around the function $\bar{V}$. Observe that $\bar{V}$ is not the expected value
of $V_n$ and hence this concentration bound is not trivial.

\begin{lemma} \label{lem:concentrationOfVarianceFunctionLevelSplit}
    Assuming that $d > 1$, $q \in [d]$ and $k > 1$, we have that
  \begin{align*}
    \Prob_{\train_n \sim \calD^n} \p{\sup_{S \subseteq [d], \abs{S} \le q} \abs{V_n(S) - \bar{V}(S)} \ge 10 \sqrt{\frac{2^q \cdot (q \log\p{d \cdot q} + t)}{n}}} \le \exp\p{- t}.
  \end{align*}
\end{lemma}

\begin{proof}
    For the purpose of the proof we will define the following
 function that interpolates between then sample based function $V_n$
 and the population based function $\bar{V}$.
  \begin{align}
            J_n(S) & \triangleq \sum_{K \in \calK_n(S)} \frac{\abs{K}}{n} \p{\Exp_{(\vecx, y) \sim \calD}\b{y \mid \vecx_S = \vecx_S^{(K)}}}^2 \label{eq:definition:empiricalMeanExactMeanEmpiricalFunction}
  \end{align}
First we bound the difference $\abs{V_n(S) - J_n(S)}$ in the following
  claim.

  \begin{claim} \label{clm:empiricalMean_EmpiricalvsExactMean_EmpiricalFunction}
      Assuming that $d > 1$, $r \in [d]$ and $t > 1$, we have that
    \begin{align*}
        \Prob_{\train_n \sim \calD^n} \p{\sup_{S \subseteq [d], \abs{S} \le r} \abs{V_n(S) - J_n(S)} \ge 5 \sqrt{\frac{2^r \cdot (r \log\p{d \cdot r} + t)}{n}}} \le \exp\p{- t}.
    \end{align*}
  \end{claim}

  \begin{proof}
      For the first part of the proof, we fix a particular set of splits $S$. Using the
    fact that both $y^{(j)}$, $m(\cdot)$ take values in $[0, 1]$ we get that
      \begin{align*}
        \abs{V_n(S) - J_n(S)} & = \abs{\sum_{K \in \calK_n(S)} \frac{\abs{K}}{n} \p{\p{\sum_{j \in K} \frac{1}{\abs{K}} y^{(j)}}^2 - \Exp_{(\vecx, y)}\b{y \mid \vecx_S = \vecx_S^{(K)}}^2}} \\
        & \le 2 \sum_{K \in \calK_n(S)} \frac{\abs{K}}{n} \abs{\p{\sum_{j \in K} \frac{1}{\abs{K}} y^{(j)}} - \Exp_{(\vecx, y)}\b{y \mid \vecx_S = \vecx_S^{(K)}}}.
      \end{align*}

      Now let $\calY_S(\vecx_S)$ be the distribution of the random variable $y$
    conditional that the random variable $\vecx$ takes value $\vecx_S$ at the subset $S$
    of the coordinates. Observe that conditional on $\vecx_S^{(j)}$, the variables
    $y^{(j)}$ for $j \in K \in \calK_n(S)$ are i.i.d. samples from the distribution
    $\calY_S(\vecx^{(K)}_S)$. Hence, using the Hoeffding's inequality we have that for any
    $K \in \calK_n(S)$ it holds that
    \begin{align*}
        \Prob_{y^{(j)} \sim \calY_S(\vecx^{(K)}_S)} \p{\abs{\sum_{j \in K} \frac{1}{\abs{K}} y^{(j)} - \Exp_{(\vecx, y)}\b{y \mid \vecx_S = \vecx_S^{(K)}}} \ge \sqrt{\frac{2 t}{\abs{K}}}} \le \exp\p{- t},
    \end{align*}
    \noindent which, by a union bound over $\calK_n(S)$, implies that
    \begin{align*}
        \Prob_{y^{(j)} \sim \calY_S(\vecx^{(j)}_S)} \p{\bigvee_{K \in \calK_n(S)} \p{\abs{\sum_{j \in K} \frac{1}{\abs{K}} y^{(j)} - \Exp_{(\vecx, y)}\b{y \mid \vecx_S = \vecx_S^{(K)}}} \ge \sqrt{\frac{2 \p{\abs{S} + t}}{\abs{K}}}}} \le \exp\p{- t},
    \end{align*}
    \noindent where we have used the fact that after splitting on $\abs{S}$ coordinates we
    can create at most $2^{\abs{S}}$ leaf nodes, i.e. $\abs{\calK_n(S)} \le 2^{\abs{S}}$.
    Hence we have that
    \begin{align} \label{eq:proof:claim:yConcentration:simple}
        \Prob_{y^{(j)} \sim \calY_S(\vecx^{(j)}_S)} \p{\abs{V_n(S) - J_n(S)} \ge \sqrt{8(\abs{S} + t)}  \frac{\sum_{K \in \calK_n(S)} \sqrt{\abs{K}}}{n}} \le \exp\p{- t}.
    \end{align}
    \noindent But we know that $\sum_{K \in \calK_n(S)} \abs{K} = n$, and also we have
    that the for any vector $\vecw \in \R^k$ it holds that
    $\norm{\vecw}_1 \le \sqrt{k} \norm{\vecw}_2$. Therefore if we define the vector
    $\vecw = (\sqrt{\abs{K}})_{K \in \calK_n(S)}$ we have that
    \begin{align*}
        \frac{\sum_{K \in \calK_n(S)} \sqrt{\abs{K}}}{n} = \frac{\norm{\vecw}_1}{\norm{\vecw}^2_2} \le \frac{\sqrt{\abs{\calK_S}}}{\norm{\vecw}_2} \le \sqrt{\frac{2^{\abs{S}}}{n}}.
    \end{align*}
    \noindent Now using this in the inequality
    \eqref{eq:proof:claim:yConcentration:simple} and taking the
    expectation over $\vecx^{(j)}_S$ for all $j$ we get the following
    inequality for any $S \subseteq [d]$.
    \begin{align} \label{eq:proof:claim:yConcentration:medium}
        \Prob_{\train_n \sim \calD^n} \p{\abs{V_n(S) - J_n(S)} \ge \sqrt{\frac{8(\abs{S} + t) 2^{\abs{S}}}{n}}} \le \exp\p{- t}.
    \end{align}
    \noindent To finalize the proof, using a union bound over all $S \subseteq [d]$ with
    $\abs{S} \le r$ we get that
    \begin{align} \label{eq:proof:claim:yConcentration:almostFinal}
        \Prob_{\train_n \sim \calD^n} \p{\sup_{S \subseteq [d], \abs{S} = r} \abs{V_n(S) - J_n(S)} \ge \sqrt{\frac{8(r + t) 2^{r}}{n}}} \le \p{\sum_{i = 0}^r \binom{d}{i}} \exp\p{- t}.
    \end{align}
    \noindent Finally using the fact that
    $\log\p{\sum_{i = 0}^r \binom{d}{i}} \le (r + 1) \log(d \cdot r)$ and
    assuming that $d > 1$, we have that $r + (r + 1) \log(d\, r) \le 3\,r\, \log(dr)$ and the claim follows.
  \end{proof}

    Next we bound the difference $\abs{J_n(S) - \bar{V}(S)}$.

  \begin{claim} \label{clm:empiricalvsExactMean_ExactMeanExactFunction}
      If we assume that $d > 1$, $r \in [d]$, $t > 1$ then we have that
    \begin{align*}
        \Prob_{\train_n \sim \calD^n} \p{\sup_{S \subseteq [d], \abs{S} \le r}\abs{J_n(S) - \bar{V}(S)} \ge \sqrt{2 \frac{r \cdot \log(d \cdot r) + t}{n}}} \le \exp\p{- t}.
    \end{align*}
  \end{claim}

  \begin{proof}
      For the first part of the proof, we fix a particular set of splits $S$. We
    then have that
    \[ \p{\Exp_{(\vecx, y) \sim \calD}\b{y \mid \vecx_S = \vecz_S}}^2 = \p{\Exp_{\vecx \sim \calD_x}\b{m(\vecx) \mid \vecx_S = \vecz_S}}^2 \triangleq M_S(\vecz_S), \]
    \noindent and hence
    \begin{align*}
      J_n(S) - \bar{V}(S) & = \sum_{K \in \calK_n(S)} \frac{\abs{K}}{n} \p{\Exp\b{m(\vecx) \mid \vecx_S = \vecx^{(K)}_S}}^2 - \Exp_{\vecx_S}\b{ \p{\Exp_{\vecx}\b{m(\vecx) \mid \vecx_S}}^2} \\
      & = \frac{1}{n} \sum_{j \in [n]} M_S(\vecx^{(j)}_S) - \Exp_{\vecx_S}\b{M_S(\vecx_S)}.
    \end{align*}
    \noindent Now since $m(\cdot) \in \b{-\frac{1}{2}, \frac{1}{2}}$, we have that for any
    $\vecx \in \{0, 1\}^d$ it holds that $\abs{M_S(\vecx_S)} \le 1$ and hence from
    Hoeffding's inequality we get that
    \begin{align*}
        \Prob_{\train_n \sim \calD^n} \p{\abs{\frac{1}{n} \sum_{j \in [n]} M_S(\vecx^{(j)}_S) - \Exp_{\vecx_S}\b{M_S(\vecx_S)}} \ge \sqrt{\frac{t}{2 n}}} \le \exp\p{- t}.
    \end{align*}
    \noindent Finally if we apply the union bound over all sets
    $S \subseteq [d]$, with $\abs{S} = r$, the claim follows.
  \end{proof}

  \noindent If we combine Claim
  \ref{clm:empiricalMean_EmpiricalvsExactMean_EmpiricalFunction} and
  \ref{clm:empiricalvsExactMean_ExactMeanExactFunction}, the Lemma
  \ref{lem:concentrationOfVarianceFunctionLevelSplit} follows.
\end{proof}

  Towards bounding the bias term we provide a relaxed version of the Definition
\ref{def:relevantVariablesForCellLevelSplit}.

\begin{definition} \label{def:relevantVariablesForCellLevelSplitFiniteSample}
    Given a set $S$, a positive number $\eta$ and a training set $\train_n$, we
  define the sets
  $\calR^\eta_n(S) = \{i \in [d] \mid V_n(S \cup \{i\}) - V_n(S) > \eta \}$ and
  $\calR^\eta(S) = \{i \in [d] \mid \bar{V}(S \cup \{i\}) - \bar{V}(S) > \eta \}$.
  For simplicity for $\eta = 0$ we use the simpler notation $\calR(S)$ and
  $\calR_n(S)$.
\end{definition}

  Since $V_n$ is a monotone increasing function we have that
$V_n(S \cup i) \ge V_n(S)$. Hence given $S$ the Algorithm
\ref{alg:mainLevelSplitFiniteSample} chooses the direction $i$ that
maximizes the positive quantity $V_n(S \cup i) - V_n(S)$. So the bad
event is that for all $j \in [d]$,
$V_n(S \cup i) - V_n(S) > V_n(S \cup j) - V_n(S)$ but
$\bar{V}(S \cup i) - \bar{V}(S) = 0$ and there exists a $k \in [d]$ such
that $\bar{V}(S \cup k) - \bar{V}(S) > 0$. A relaxed version of this bad
event can be described using the Definition
\ref{def:relevantVariablesForCellLevelSplitFiniteSample}. In this
language the bad event is that the index $i \in [d]$ that the Algorithm
\ref{alg:mainLevelSplitFiniteSample} chooses to split does not belong to
$\calR^{\eta}(S)$ although $\calR^{\eta}(S) \neq \emptyset$. We bound the
probability of this event in the next lemma.

\begin{lemma} \label{lem:splitRelevantFiniteSampleLevelSplit}
    Let $\eta = 10 \sqrt{\frac{2^r \cdot (r \log\p{d \cdot r} + t)}{n}}$
  and assume that $d > 1$, $r \in [d]$ and $t > 1$, then it holds that
  \begin{align*}
    \Prob_{\train_n \sim \calD^n} \p{\bigvee_{S \subseteq [d], \abs{S} \le r} \p{\p{\argmax_{i \in [d]} V_n(S \cup i)} \not\in \calR(S) \mid \calR^{2 \eta}(S) \neq \emptyset}} \le 2 \exp(- t)
  \end{align*}
\end{lemma}

\begin{proof}
    Directly applying Lemma
  \ref{lem:concentrationOfVarianceFunctionLevelSplit} to Definition
  \ref{def:relevantVariablesForCellLevelSplitFiniteSample} we have that
  \begin{align} \label{eq:proof:chooseBadCoordinatesBound:1}
    \Prob_{\train_n \sim \calD^n} \p{\bigvee_{i, S} \p{i \in \calR^{\eta}_n(S) \mid i \not\in \calR(S)}} \le \exp(- t),
  \end{align}
  \noindent where
  $\eta = 10 \sqrt{\frac{2^r \cdot (r \log\p{d \cdot r} + t)}{n}}$. Similarly we have
  that
  \begin{align} \label{eq:proof:chooseBadCoordinatesBound:2}
    \Prob_{\train_n \sim \calD^n} \p{\bigvee_{i, S} \p{i \not\in \calR^{\eta}_n(S) \mid i \in \calR^{2 \eta}(S)}} \le \exp(- t).
  \end{align}
  \noindent If we combine the above inequalities we get that there is a
  very small probability that there exists an index
  $i \in \calR^{2 \eta}(S)$ but an index $j \not\in \calR(S)$ is chosen
  instead. This is summarized in the following inequality
  \begin{align} \label{eq:proof:chooseBadCoordinatesBound:combined}
    \Prob_{\train_n \sim \calD^n} \p{\bigvee_{S} \p{\p{\argmax_{i \in [d]} V_n(S \cup i)} \not\in \calR(S) \mid \calR^{2 \eta}(S) \neq \emptyset}} \le 2 \exp(- t)
  \end{align}
  \noindent and the lemma follows.
\end{proof}

\begin{lemma} \label{lem:optimalityWithoutSplitsFiniteSampleLevelSplit}
    For every set $S \subseteq [d]$, under Assumption~\ref{asp:submodularityLevelSplit}, if $\calR^{\eta}(S) = \emptyset$,
  then
  \[ \Exp_{(\vecx, y) \sim \calD} \b{\p{m(\vecx) - \Exp_{(\vecx, y) \sim \calD}\b{m(\vecx) \mid \vecx_S}}^2} \le C \cdot \eta \cdot \abs{\calR(S)}. \]
\end{lemma}

\begin{proof}
    We know that $\bar{L}([d]) = 0$ and under the Assumption
  \ref{asp:submodularityLevelSplit}, the function $\bar{L}(\cdot)$ is
  approximate supermodular. We first prove that $\bar{L}(S \cup \calR(S)) = 0$.
  If this is not the case then there exists an $i \not\in \calR(S)$ such
  that $\bar{L}(S \cup \calR(S) \cup i) - \bar{L}(S \cup \calR(S)) < 0$.
  But because of the approximate supermodularity of $\bar{L}$ we have that
  \[ \bar{L}(S \cup i) - \bar{L}(S) < C \cdot (\bar{L}(S \cup \calR(S) \cup i) - \bar{L}(S \cup \calR(S))) < 0 \]
  \noindent which contradicts with the assumption that
  $i \not\in \calR(S)$.

    Now assume that $\calR^{\eta}(S) = \emptyset$ and for the sake of
  contradiction also assume that
  $\bar{L}(S) > C \cdot \eta \cdot \abs{\calR(S)}$. Let $\{r_1, \dots, r_k\}$ be
  an arbitrary enumeration of the set $\calR(S)$. From the argument
  before we have that $\bar{L}(S \cup \calR(S)) = 0$ hence there exists
  an element $r_j$ of $\calR(S)$ such that
  \[ \bar{L}(S \cup \{r_1, \dots, r_{j - 1}\}) - \bar{L}(S \cup \{r_1, \dots, r_j\}) > C \cdot \eta, \]
  \noindent otherwise we would immediately have
  $\bar{L}(S) \le C \cdot \eta \cdot \abs{\calR(S)}$. But because of the
  approximate supermodularity of $\bar{L}(\cdot)$ we have that
  \[ C \cdot (\bar{L}(S) - \bar{L}(S \cup r_j)) \ge \bar{L}(S \cup \{r_1, \dots, r_{j - 1}\}) - \bar{L}(S \cup \{r_1, \dots, r_j\}) > C \cdot \eta, \]
  but this last inequality implies $r_j \in \calR^{\eta}(S)$ which
  contradicts with our assumption that $\calR^{\eta}(S) = \emptyset$.
  Hence $\bar{L}(S) \le C \cdot \eta \cdot \abs{\calR(S)}$ and the lemma
  follows.
\end{proof}

  Finally we need one more Lemma to handle the case where Assumption
\ref{asp:strongSparsityLevelSplit} holds.

\begin{lemma} \label{lem:splitRelevantFiniteSampleLevelSplitStrongSparsity}
    Let $m$ be a target function that is $(\beta, r)$-strongly sparse,
  with set of relevant features $R$, and suppose
  $n \ge 256 \cdot \frac{2^r \cdot (r \log (d \cdot r) + t)}{\beta^2}$,
  then it holds that
  \begin{align*}
    \Prob_{\train_n \sim \calD^n} \p{\bigvee_{S \subseteq [d], \abs{S} \le r} \p{\p{\argmax_{i \in [d]} V_n(S \cup i)} \not\in R \mid (R \setminus S) \neq \emptyset}} \le 2 \exp(- t)
  \end{align*}
\end{lemma}

\begin{proof}
    Directly applying Lemma
  \ref{lem:concentrationOfVarianceFunctionLevelSplit} to Definition
  \ref{def:relevantVariablesForCellLevelSplitFiniteSample} we have that
  \begin{align} \label{eq:proof:chooseBadCoordinatesBoundStrongSparsity:1}
    \Prob_{\train_n \sim \calD^n} \p{\bigvee_{i, S} \p{i \in \calR^{\eta}_n(S) \mid i \not\in \calR(S)}} \le \exp(- t),
  \end{align}
  \noindent where
  $\eta = 10 \sqrt{\frac{2^r \cdot (r \log\p{d \cdot r} + t)}{n}}$.
  Similarly we have that
  \begin{align} \label{eq:proof:chooseBadCoordinatesBoundStrongSparsity:2}
    \Prob_{\train_n \sim \calD^n} \p{\bigvee_{i, S} \p{i \not\in \calR^{\eta}_n(S) \mid i \in \calR^{2 \eta}(S)}} \le \exp(- t).
  \end{align}
  \noindent If we combine the above inequalities with the Assumption
  \ref{asp:strongSparsityLevelSplit} the lemma follows.
\end{proof}

 We are now ready to upper bound the bias of the Algorithm
\ref{alg:mainLevelSplitFiniteSample} under the Assumption
\ref{asp:submodularityLevelSplit}.

\begin{theorem}[Bias under Approximate Submodularity]
  \label{thm:recoverySubmodularIndependentBalancedProperLevelSplitFiniteSample}
    Let $\train_n$ be i.i.d. samples from the non-parametric regression model
  $y = m(\vecx) + \eps$, where $m(\vecx) \in [-1/2, 1/2]$, $\eps \sim \calE$,
  $\Exp_{\eps \sim \calE}[\eps] = 0$ and $\eps \in [-1/2, 1/2]$. Let also
  $\calP_n$ be the partition that Algorithm~\ref{alg:mainLevelSplitFiniteSample}
  returns. Then under the submodularity Assumption
  \ref{asp:submodularityLevelSplit} the following statements hold.
  \begin{Enumerate}
    \item If $m$ is $r$-sparse and we set
    $\log(t) \ge \frac{C \cdot r}{C \cdot r + 2}\p{\log(n) - \log(\log(d/\delta))}$,
    then it holds that
    \begin{align*}
      \Prob_{\train_n \sim \calD^n}\p{\Exp_{\vecx \sim \calD_x} \b{\p{m(\vecx) - \Exp_{\vecz \sim \calD_x}\b{m(\vecz) \mid \vecz \in \calP_n(\vecx)}}^2} > \tilde{\Omega}\p{C \cdot r \cdot \sqrt[C \cdot r + 2]{\frac{\log(d/\delta)}{n}}}} \le \delta.
    \end{align*}
    \item Under the independence of features Assumption
    \ref{asp:independentFeature} and assuming that $m$ is $r$-sparse and if
    $\log(t) \ge r$, then it holds that
    \begin{align*}
      \Prob_{\train_n \sim \calD^n}\p{\Exp_{\vecx \sim \calD_x} \b{\p{m(\vecx) - \Exp_{\vecz \sim \calD_x}\b{m(\vecz) \mid \vecz \in \calP_n(\vecx)}}^2} > \tilde{\Omega}\p{C \cdot \sqrt{\frac{2^r \cdot \log(d/\delta))}{n}}}} \le \delta.
    \end{align*}
  \end{Enumerate}
\end{theorem}

\begin{proof}
    We fix $S$ to be the set of splits that Algorithm
  \ref{alg:mainLevelSplitFiniteSample} chooses. Our goal in this lemma is to show that with
  high probability the following quantity is small
  \begin{align} \label{eq:proof:levelSplitsBias:definition:partitionMSE}
    \bar{L}(\calP_n) = \Exp_{\vecx \sim \calD_x}\b{\p{m(\vecx) - \Exp_{\vecz \sim \calD_x}\b{m(\vecz) \mid \vecz \in \calP_n(\vecx)}}^2}.
  \end{align}
  We prove this in two steps. First we show that after the completion of the level
  $\level$ of the algorithm the quantity $\bar{L}(S_{\level})$ is small. Then we bound the
  difference $\abs{\bar{L}(\calP_{\level}) - \bar{L}(S_{\level})}$ in Claim
  \ref{clm:proof:levelSplitsBias:partitionMSEsetMSE}. Finally we use the monotonicity of
  $\bar{L}$ to argue about the upper bound on $\calP_n$.

    Let $R \subseteq [d]$ be the set of size $\abs{R} = r$ of the relevant
  features of the target function $m$. Observe that it holds that $\bar{L}(S \cup R) = 0$
  and hence $\bar{V}(S \cup R) \triangleq V^*$ is maximized. Since
  $m(\vecx) \in [-1/2, 1/2]$, the maximum value of $\bar{V}$ is $1$.

    For the first part of the theorem, let $\{i_1, \dots, i_r\}$ and be an
  arbitrary enumeration of $R$ and let $R_j = \{i_1, \dots, i_j\}$ then by
  adding and subtracting terms of the form $\bar{V}(S \cup R_j)$ we have the
  following equality
  \[ \p{\bar{V}(S \cup R) - \bar{V}(S \cup R_{r - 1})} + \cdots + \p{\bar{V}(S \cup R_2) - \bar{V}(S \cup \{i_1\})} + \bar{V}(S \cup \{i_1\}) = V^*. \]
  \noindent From the approximate submodularity of $\bar{V}$ we hence have that
  \[ \p{\bar{V}(S \cup \{i_r\}) - \bar{V}(S)} + \cdots + \p{\bar{V}(S \cup \{i_2\}) - \bar{V}(S)} + \p{\bar{V}(S \cup \{i_1\}) - \bar{V}(S)} \ge \frac{V^* - \bar{V}(S)}{C} \]
  \noindent which implies
  \[ \max_{j \in [r]} \p{\bar{V}(S \cup \{i_j\}) - \bar{V}(S)} \ge \frac{V^* - \bar{V}(S)}{C \cdot r}. \]
  \noindent Let $i^{\level}$ be the coordinate that the algorithm chose to
  split at level $\level$. Now from the greedy criterion of Algorithm
  \ref{alg:mainLevelSplitFiniteSample} we get that the coordinate $i^{\level}$
  that we picked to split was at least as good as the best of the coordinates in
  $R$, hence using Lemma
  \ref{lem:concentrationOfVarianceFunctionLevelSplit} and if we define
  $\eta = 8 \sqrt{\frac{2^q \cdot (q \log\p{d \cdot q} + \log(1/\delta))}{n}}$,
  where $q$ is the maximum depth of the tree for which we are applying Lemma
  \ref{lem:concentrationOfVarianceFunctionLevelSplit}, we have that with
  probability at least $1 - \delta$ it holds that
  \[ \bar{V}\p{S \cup \{i^{\level}\}} \ge \bar{V}(S) + \frac{V^* - \bar{V}(S)}{C \cdot r} - 2 \eta \]
  which in turn because
  $L^* \triangleq \bar{L}(S \cup R) = 0$ implies that
  \begin{equation} \label{eq:potentialProgressPopulationLevelSplitFiniteSample}
    \bar{L}\p{S \cup \{i^{\level}\}} \le \bar{L}(S) \p{1 - \frac{1}{C \cdot r}} + 2 \eta.
  \end{equation}

  \noindent Let $S_{\level}$ to be the set of splits after the step
  $\level$ of Algorithm \ref{alg:mainLevelSplitFiniteSample}. Then it holds that
  \[ \bar{L}\p{S_{\level + 1}} \le \bar{L}\p{S_{\level}} \p{1 - \frac{1}{C \cdot r}} + 2 \eta. \]
  \noindent Inductively and using the fact that $m(\vecx) \in [-1/2, 1/2]$, the latter
  implies that
  \begin{equation} \label{eq:potentialProgressFiniteSampleLevel}
    \bar{L}\p{S_{\level}} \le \bar{L}\p{\emptyset} \p{1 - \frac{1}{C \cdot r}}^{\level} + 2 \level \cdot \eta \le \p{1 - \frac{1}{C \cdot r}}^{\level} + 2 \level \cdot \eta.
  \end{equation}
  \noindent From the choice of $t$ we have that for
  $\level = C \cdot r \ln\p{1/\eta}$ it holds $\bar{L}(S_{\level}) \le 3\cdot C \cdot r \ln(1/\eta) \eta$. For this analysis to
  be consistent we have to make sure that the maximum depth $q$ for which we are
  applying Lemma \ref{lem:concentrationOfVarianceFunctionLevelSplit} is at least the value required for $\level$. Thus we need to
  find values for $q$, $\eta$ such that $q \ge C \cdot r \ln\p{1/\eta}$ at
  the same time when
  $\eta \ge 8 \sqrt{\frac{2^q \cdot (q \log\p{d \cdot q} + \log(1/\delta))}{n}}$.
  It is easy to see that the smallest possible value for $\eta$ is hence achieved
  for $q = \frac{C \cdot r}{C \cdot r + 2}\p{\log(n) - \log(\log(d/\delta))}$
  and $\eta = \tilde{\Theta}\p{\sqrt[C \cdot r + 2]{\frac{\log(d/\delta)}{n}}}$.
  Hence the inequality
  $\bar{L}(S_{\level}) \le 3 \cdot C \cdot r \ln(1/\eta) \eta$ which implies
    \begin{align} \label{eq:proof:levelSplitsBias:setPartitionBound:1}
    \bar{L}(S_{\level}) \le \tilde{O}\p{C \cdot r \cdot \sqrt[C \cdot r + 2]{\frac{\log(d/\delta)}{n}}}.
  \end{align}

    For the second part of the theorem we use Lemma
  \ref{lem:splitRelevantFiniteSampleLevelSplit} and we have that at every step
  either the algorithm chooses to split with respect to a direction
  $i \in \calR(S)$ or $\calR^{\eta}(S) = \emptyset$. Because of our assumption
  that $m$ is $r$-sparse and because we assume that the features are
  distributed independently we have that at any step $\abs{\calR(S)} \le r$.
  Hence, it has to be that for some $\level \le r$ the set $S_{\level}$
  during the execution of the Algorithm \ref{alg:mainLevelSplitFiniteSample}
  satisfies $\calR^{\eta}(S) = \emptyset$. Then using Lemma
  \ref{lem:optimalityWithoutSplitsFiniteSampleLevelSplit} we have that
  $\bar{L}(S_{\level}) \le C \cdot \eta \cdot r$ from which we get that
  \begin{align} \label{eq:proof:levelSplitsBias:setPartitionBound:2}
    \bar{L}(S_{\level}) \le O\p{C \cdot \sqrt{\frac{2^r \cdot \log(d/\delta))}{n}}}.
  \end{align}

  \smallskip

  Next we need to compare $\bar{L}(S_{\level})$ with $\bar{L}(\calP_{\level})$.

  \begin{claim}[Dealing with Empty Cells] \label{clm:proof:levelSplitsBias:partitionMSEsetMSE}
    It holds that with probability $1-\delta$
    \[ \abs{\bar{L}(\calP_{\level}) - \bar{L}(S_{\level})} \le 8 \cdot \frac{2^{\level} \level \ln(2 d\, \level) + \ln(1/\delta)}{n}. \]
  \end{claim}

  \begin{proof}
      Fix any possible cell $A$ after doing a full partition on the first
    $q \triangleq \level$ splits of the algorithm. For simplicity of the exposition of this
    proof we define for every subset $B$ of $\{0, 1\}^n$ the probability
    $P_B \triangleq \Prob_{\vecx \sim \calD_x}\b{\vecx \in B}$ and the empirical
    probability
    $\hat{P}_B \triangleq \frac{1}{n} \sum_{i = 1}^n \chara\{\vecx^{(j)} \in B\}$. Using
    the Chernoff-Hoeffding bound we get that
    \[ \Prob_{\train_n \sim \calD^n}\p{\hat{P}_A \ge P_A - \sqrt{\frac{2 \ln(1/\delta) P_A}{n}}} \ge 1 - \delta. \]
    If the empirical probability $\hat{P}_A$ is zero then we get the following
    \[ \Prob_{\train_n \sim \calD^n}\p{P_A \le \frac{2 \ln(1/\delta)}{n} + \chara\{\hat{P}_A \neq 0\}} \ge 1 - \delta. \]
    The number of possible cells from a tree of depth $q$ is at most $2^q d^q q^q$.
    Therefore, by union bound over all the possible cells $A$, we have that
    \begin{align} \label{eq:cellProbability}
      \Prob_{\train_n \sim \calD^n}\p{\bigvee\limits_{A} \p{P_A \le \frac{2 q \ln(2 d q) + 2 \ln(1/\delta)}{n} + \chara\{\hat{P}_A \neq 0\}}} \ge 1 - \delta.
    \end{align}
    Next, we consider the difference $\bar{L}(\calP_{\level}) - \bar{L}(S_{\level})$. Let
    $\calP_S^{\level}$ be the partition of space if we split along all the coordinates in
    $S_{\level}$. It is easy to see that $\calP_S^{\level}$ is a refinement of the
    partition $\calP_{\level}$. Hence in the difference
    $\bar{L}(\calP_{\level}) - \bar{L}(S_{\level})$ we only have terms of the form
    \[ \Prob_{\vecx \sim \calD_x}\p{\vecx \in B} \Exp_{\vecx \sim \calD_x}\b{\p{m(\vecx) - \Exp_{\vecz \sim \calD_x}\b{m(\vecz) \mid \vecz \in B}}^2 \mid \vecx \in B} - ~~~~~~~~~~~~~~~~~~~~~~~ \]
    \[ ~~~~~~~~~~~~~ - \sum_{j = 1}^{\ell} \Prob_{\vecx \sim \calD_x}\p{\vecx \in A_j} \Exp_{\vecx \sim \calD_x}\b{\p{m(\vecx) - \Exp_{\vecz \sim \calD_x}\b{m(\vecz) \mid \vecz \in A_j}}^2 \mid \vecx \in A_j}. \]
    Where $B \in \calP_{\level}$, $A_j \in \calP^{\level}_S$ and $B$ is the union of the
    cells $A_1$, $\dots$, $A_{\ell}$. In order for $B$ to remain unsplit in
    $\calP_{\level}$ it has to be that for all but one of $A_j$'s it holds that
    $\hat{P}_{A_j} = 0$. We denote with $\calE(B)$ the above difference and we observe
    that it is equal to the following
    \[ \sum_{j = 1}^{\ell} P_{A_j} \Exp_{\vecx \sim \calD_x}\b{\p{m(\vecx) - \Exp_{\vecz \sim \calD_x}\b{m(\vecz) \mid \vecz \in B}}^2 - \p{m(\vecx) - \Exp_{\vecz \sim \calD_x}\b{m(\vecz) \mid \vecz \in A_j}}^2 \mid \vecx \in A_j}. \]
    Without loss of generality we will assume that $A_1$ is the only subcell of $B$
    that is not empty. We define $\calQ(A_1)$ the following quantity
    \[ P_{A_1} \Exp_{\vecx \sim \calD_x}\b{\p{m(\vecx) - \Exp_{\vecz \sim \calD_x}\b{m(\vecz) \mid \vecz \in B}}^2 - \p{m(\vecx) - \Exp_{\vecz \sim \calD_x}\b{m(\vecz) \mid \vecz \in A_1}}^2 \mid \vecx \in A_1}. \]
    Since $m(\vecx) \in [-1/2, 1/2]$, we get that
    \begin{align} \label{eq:proof:fromQA1toEB}
      \calE(B) & \le \calQ(A_1) + 2 \sum_{j=2}^\ell P_{A_j}.
    \end{align}
    Next we also bound $\calQ(A_1)$ by the measure of cells in $B\setminus A_1$. The intuition why $\calQ(A_1)$ is small is that, since the cells $B\setminus A_1$ have small measure, then the conditional expectation of $m(z)$ conditional on $z\in B$ is very close to the conditional expectation of $m(z)$ conditional on $z\in A_1$.
    More formally, since $x^2$ is $2$-Lipschitz for $x\in [-1,1]$, $m(z)\in [-1/2, 1/2]$ and $A_1\subseteq B$:
    \begin{align*}
        \calQ(A_1) & \leq 2 P_{A_1} \abs{\Exp_{\vecz \sim \calD_x}\b{m(\vecz) \mid \vecz \in B} - \Exp_{\vecz \sim \calD_x}\b{m(\vecz) \mid \vecz \in A_1}}\\
        & = 2 P_{A_1} \abs{ \Exp_{\vecz \sim \calD_x}\b{m(\vecz) \mid \vecz \in B\setminus A_1} - \Exp_{\vecz \sim \calD_x}\b{m(\vecz) \mid \vecz \in A_1}}\, \Prob(\vecz\in B\setminus A_1\mid \vecz \in B)\\
        & \leq 2 \frac{P_{A_1}}{P_B} \cdot \p{P_B - P_{A_1}} \leq 2  \p{P_B - P_{A_1}} = 2 \p{\sum_{j = 2}^{\ell} P_{A_j}}
    \end{align*}
    Since all the cells $A_j$, with $j \ge 2$, have $\hat{P}_{A_j} = 0$, this means that
    $P_{A_j} \le \frac{2 q \ln(2 d q) + 2 \ln(1/\delta)}{n}$ due to
    \eqref{eq:cellProbability}.  Putting this together with \eqref{eq:proof:fromQA1toEB} we get that
    \[ \calE(B) \le 4 \sum_{j=2}^\ell P_{A_j} \leq 8 (\ell - 1) \frac{q \ln(2 d q) + \ln(1/\delta)}{n}. \]
    Let $\ell_B$ the number of subcells of $B \in \calP_{\level}$ that are inside
    $\calP^{\level}_S$. If we sum over all $B \in \calP_{\level}$ we get that
    \[ \abs{\bar{L}(\calP_{\level}) - \bar{L}(S_{\level})} \le 8 \p{\sum_{B \in \calP_{\level}} \ell_B} \frac{q \ln(2 d q) + \ln(1/\delta)}{n} \]
    but the sum $\sum_{B \in \calP_{\level}} \ell_B$ is less than the size of
    $\calP^{\level}_S$ which is $2^q$ and hence we get that
    \[\abs{\bar{L}(\calP_{\level}) - \bar{L}(S_{\level})} \le 8 \cdot \frac{2^q q \ln(2 d q) + \ln(1/\delta)}{n}.\]
  \end{proof}

    Using Claim \ref{clm:proof:levelSplitsBias:partitionMSEsetMSE} and equations
  \eqref{eq:proof:levelSplitsBias:setPartitionBound:1} and
  \eqref{eq:proof:levelSplitsBias:setPartitionBound:2} we get the first two parts of the
  theorem by observing that the error term in Claim
  \ref{clm:proof:levelSplitsBias:partitionMSEsetMSE} is less that the error terms in
  \eqref{eq:proof:levelSplitsBias:setPartitionBound:1} and
  \eqref{eq:proof:levelSplitsBias:setPartitionBound:2}.
\end{proof}

We conclude the section by bounding the bias under the strong sparsity Assumption \ref{asp:strongSparsityLevelSplit}. For this, we use the notion of the value-diameter of a cell given in Definition~\ref{def:valueDiameter}.

\begin{theorem} \label{thm:biasStronglySparseLevelSplitFiniteSample}
    Let $\train_n$ be i.i.d. samples from the non-parametric regression model
  $y = m(\vecx) + \eps$, where $m(\vecx) \in [-1/2, 1/2]$, $\eps \sim \calE$,
  $\Exp_{\eps \sim \calE}[\eps] = 0$ and $\eps \in [-1/2, 1/2]$. Let also
  $\calP_n$ be the partition that Algorithm~\ref{alg:mainLevelSplitFiniteSample}
  returns. If $m$ is $(\beta, r)$-strongly sparse as per Assumption
  \ref{asp:strongSparsityLevelSplit} then the following statements hold for the bias
  of the output of Algorithm \ref{alg:mainLevelSplitFiniteSample}.
  \begin{Enumerate}
    \item If $n \ge \tilde{\Omega}\p{\frac{2^r(\log(d/\delta))}{\beta^2}}$ and we set
    $\log(t) \ge r$, then it holds that
    \begin{align*}
      \Prob_{\train_n \sim \calD^n}\p{\Exp_{\vecx \sim \calD_x} \b{\p{m(\vecx) - \Exp_{\vecz \sim \calD_x}\b{m(\vecz) \mid \vecz \in \calP_n(\vecx)}}^2} > \tilde{\Omega}\p{\frac{2^r \cdot \log(d/\delta))}{n}}} \le \delta.
    \end{align*}
    \item If $R$ is the set of relevant features and and for every $\vecw \in \{0, 1\}^r$
    it holds for the marginal probability that
    $\Prob_{\vecz \sim \calD_x}\p{\vecz_R = \vecw} \not\in (0, \zeta/2^r)$ and if
    $n \ge \tilde{\Omega}\p{\frac{2^r(\log(d/\delta))}{\beta^2} + \frac{2^r \log(1/\delta)}{\zeta}}$ and we set $\log(t) \ge r$, then it holds that
    \begin{align*}
      \Prob_{\train_n \sim \calD^n}\p{\Delta_m(\calP_n) = 0} \ge 1 - \delta.
    \end{align*}
    \item Let $R$ be the set of relevant features, $\vecx \in \{0, 1\}^d$ such that
    $\Prob_{\vecz \sim \calD_x}\p{\vecz_R = \vecx_R} \ge \zeta/2^r$, and assume that we run
    Algorithm \ref{alg:mainLevelSplitFiniteSample} with input $h = 1$ and $\log(t) \ge r$. If
    $n \ge \tilde{\Omega}\p{\frac{2^r(\log(d/\delta))}{\beta^2} + \frac{2^r \log(1/\delta)}{\zeta}}$, then it holds that
    \begin{align*}
      \p{\Exp_{\train_n \sim \calD^n}\b{m_{n}(\vecx)} - m(\vecx)}^2 \le \delta.
    \end{align*}
  \end{Enumerate}
\end{theorem}
\begin{proof}
    For the first part of the theorem we observe that Lemma
  \ref{lem:splitRelevantFiniteSampleLevelSplitStrongSparsity} implies that
  $\bar{L}(S_{\level}) = 0$. Then using Claim
  \ref{clm:proof:levelSplitsBias:partitionMSEsetMSE} the first part of the theorem
  follows.\\

    For the second part of the theorem we fix any possible cell $A$ after doing a full
  partition on the first $r$ splits of Algorithm
  \ref{alg:mainLevelSplitFiniteSample}. For simplicity of the exposition of this
  proof we define for every subset $B$ of $\{0, 1\}^n$ the probability
  $P_B \triangleq \Prob_{\vecx \sim \calD_x}\b{\vecx \in B}$ and the empirical
  probability
  $\hat{P}_B \triangleq \frac{1}{n} \sum_{i = 1}^n \chara\{\vecx^{(j)} \in B\}$. Using
  the multiplicative form of the Chernoff bound we get that
  \[ \Prob_{\train_n \sim \calD^n}\p{n \hat{P}_A \ge \p{1 - \sqrt{\frac{2 \log(1/\delta)}{n P_A}}} n P_A} \ge 1 - \delta. \]
  Hence for $n \ge \frac{8 \log(1/\delta)}{P_A}$ we have that
  \[ \Prob_{\train_n \sim \calD^n}\p{\sum_{i = 1}^n \chara\{\vecx^{(i)} \in A\} \ge 1} \ge 1 - \delta. \]
  Next we can apply a union bound over all possible cells $A$ that are defined via splits on the
  coordinates in $R$ and using our assumption that $P_A \ge \frac{\zeta}{2^r}$ we get
  that for $n \ge \frac{24 2^r (r + \log(1/\delta)}{\zeta}$ it holds that
  \begin{align} \label{eq:proof:strongSparsityBias:leafprobability}
    \Prob_{\train_n \sim \calD^n}\p{\bigvee\limits_A \p{\sum_{i = 1}^n \chara\{\vecx^{(i)} \in A\} \ge 1}} \ge 1 - \delta.
  \end{align}
  Now let $S$ be the set of splits after $r$ iterations of the Algorithm
  \ref{alg:mainLevelSplitFiniteSample}. Then Lemma
  \ref{lem:splitRelevantFiniteSampleLevelSplitStrongSparsity} implies that $S = R$ and
  $\bar{L}(S) = 0$. Finally from \eqref{eq:proof:strongSparsityBias:leafprobability} we
  also have that the partition $\calP_r$ after $r$ iterations of Algorithm
  \ref{alg:mainLevelSplitFiniteSample} is the full partition to all the cells
  of $R$ and hence
  \[ \Exp_{\vecx \sim \calD_x} \b{\p{m(\vecx) - \Exp_{\vecz \sim \calD_x}\b{m(\vecz) \mid \vecz \in \calP_n(\vecx)}}^2} \le \Exp_{\vecx \sim \calD_x} \b{\p{m(\vecx) - \Exp_{\vecz \sim \calD_x}\b{m(\vecz) \mid \vecz_S = \vecx_S}}^2} \]
  where the later is $0$ with high probability because of Lemma
  \ref{lem:splitRelevantFiniteSampleLevelSplitStrongSparsity}. This means that with
  probability at least $1 - \delta$ it holds that
  \[ \sum_{A \in \calP_n} \Prob_{\vecx \sim \calD_x} \p{\vecx \in A} \cdot \Exp_{\vecx \sim \calD_x}\b{\p{m(\vecx) - \Exp_{\vecz \sim \calD_x}\b{m(\vecz) \mid \vecz \in A}}^2 \mid \vecx \in A} = 0. \]
  Since all the summands in the above expression are positive, it has to be that for
  every cell $A \in \calP_n$ it holds that either
  \[ \Prob_{\vecx \sim \calD_x}\p{\vecx \in A} = 0 ~~~~\text{or}~~~~ \Exp_{\vecx \sim \calD_x}\b{\p{m(\vecx) - \Exp_{\vecz \sim \calD_x}\b{m(\vecz) \mid \vecz \in A}}^2 \mid \vecx \in A} = 0 \]
  which from the definition of value-diameter implies that
  \[ \Prob_{\vecx \sim \calD_x}\p{\vecx \in A} = 0 ~~~~\text{or}~~~~ \Delta_m(A) = 0 \]
  and in turn this implies $\Delta_m(\calP_n) = 0$ with probability at least
  $1 - \delta$. \\

    The third part of the theorem follows similar lines as the argument we used in the Bias-Variance decomposition Lemma~\ref{lem:biasVarianceDeepHonestForests}, albeit we repeat it here, since we are only interested at evaluating bias at a fixed $x$ and hence we can prove a stronger statement that only requires a lower bound on the density of $x$ and not a uniform lower bound on the density of all cells defined via splits on the relevant variables.

    We define for simplicity
  $w^{(j)}(\vecx) = \frac{\chara\{\vecx_{\calT_n(S, \vecx)}^{(j)}=\vecx_{\calT_n(S, \vecx)}\}}{N_n(\vecx; \calT_n(S, \vecx))}$ and hence
  $m_{n}(\vecx) = \sum_{i = 1}^n w^{(j)}(\vecx) y^{(j)}$ and we have:
  \begin{align*}
    & \left(\Exp_{\train_n \sim \calD^n}[m_{n}(\vecx)] - m(\vecx)\right)^2
    = \\
    & ~~~~~~~~~~ = \left(\Exp_{\train_n \sim \calD^n}\left[\sum_{j = 1}^n w^{(j)}(\vecx)\, (y^{(j)} - m(\vecx^{(j)}))\right] +
    \Exp_{\train_n \sim \calD^n}\left[\sum_{j = 1}^n w^{(j)}(\vecx)\, (m(\vecx^{(j)}) - m(\vecx))\right]\right)^2
  \end{align*}
  Due to honesty $w^{(j)}(\vecx)$ is independent of $y^{(j)}$ and we have that the first term is
  equal to $0$ by a tower law. Thus we have:
  \begin{align*}
    \p{\Exp_{\train_n \sim \calD^n}\b{m_{n}(\vecx)} - m(\vecx)}^2 =~& \p{\Exp_{\train_n \sim \calD^n}\b{\sum_{j = 1}^s w^{(j)}(\vecx)\, (m(\vecx^{(j)}) - m(\vecx))}}^2\\
    \leq~& \Exp_{\train_n \sim \calD^n}\b{\p{\sum_{j = 1}^n w^{(j)}(\vecx) (m(\vecx^{(j)}) - m(\vecx))}^2}\\
    \intertext{Let also $A = \{\vecz \mid \vecz_R = \vecx_R\}$, then using the
    multiplicative form of the Chernoff Bound from the proof of the second part of the theorem
    above we get
    $\Prob_{\train_n \sim \calD^n}\p{\sum_{i = 1}^n \chara\{\vecx^{(i)} \in A\} \ge 1} \ge 1 - \delta$.
    Therefore with probability $1 - \delta$ the path of the tree that leads to $\vecx$ has split
    all the relevant coordinates $R$ and hence for all $j$ such that $w^{(j)}(\vecx) > 0$ it
    holds that $\vecx^{(j)}_R = \vecx_R$ which in turn implies that
    $m(\vecx^{(j)}) = m(\vecx)$. With the rest $\delta$ probability the square inside the
    expectation is at most $1$ since $m(\cdot) \in \b{-\frac{1}{2}, \frac{1}{2}}$, hence we get}
    \p{\Exp_{\train_n \sim \calD^n}\b{m_{n}(\vecx)} - m(\vecx)}^2 \le & ~~~\delta.
  \end{align*}
\end{proof}

\subsection{Proof of Theorem \ref{thm:finalRecoverySubmodularIndependentBalancedProperLevelSplitFiniteSample}}
\label{sec:app:levelSplit:Shallow:mainTheoremProof}

Observe that the output estimate $m_n(\cdot; S)$ and partition $\calP_n$ of
Algorithm~\ref{alg:mainLevelSplitFiniteSample}, satisfies the conditions of
Lemma~\ref{lem:bias-variance}. Moreover, by Corollary~\ref{cor:criticalradiusShallow}, we
have that the critical radius quantity $\delta_n$ is of order
$\Theta\p{\sqrt{\frac{t\log(d\, t)(1+\log(n))}{n}}}$, if we grow the tree at depth $\log(t)$.
Thus applying the bound presented in \eqref{eq:mainBeforeFinal} with the bound on
$\delta_n$ we have the following cases:
\begin{Enumerate}
  \item from case 1 of Theorem
  \ref{thm:recoverySubmodularIndependentBalancedProperLevelSplitFiniteSample} we get case
  1 of Theorem
  \ref{thm:finalRecoverySubmodularIndependentBalancedProperLevelSplitFiniteSample},
  \item from case 2 of Theorem
  \ref{thm:recoverySubmodularIndependentBalancedProperLevelSplitFiniteSample} we get case
  2 of Theorem
  \ref{thm:finalRecoverySubmodularIndependentBalancedProperLevelSplitFiniteSample} and
  \item from case 1 of Theorem \ref{thm:biasStronglySparseLevelSplitFiniteSample} we get
  case 3 of Theorem
  \ref{thm:finalRecoverySubmodularIndependentBalancedProperLevelSplitFiniteSample}.
\end{Enumerate}

\subsection{Proof of Theorem \ref{thm:finalRecoverySubmodularIndependentBalancedProperLevelSplitFiniteSampleFullyGrown}} \label{sec:app:proofs:levelSplits:variance:fullyGrown}

From case 2 of Theorem \ref{thm:biasStronglySparseLevelSplitFiniteSample}
and since the maximum possible value diameter is $1$, we have that if
$s \ge \tilde{\Omega}\p{\frac{2^r(\log(d/\delta))}{\beta^2} + \frac{2^r \log(1/\delta)}{\zeta}}$
then $\Exp_{\train_n \sim \calD^n}\b{\Delta_m(\calP_s)} \le \delta$ which implies
$\Exp_{\vecx \sim \calD_x}[(\bar{m}_s(\vecx) - m(\vecx))^2] \le \delta$. Putting this
together with the Lemma~\ref{lem:biasVarianceDeepHonestForests} we get that if
$s = \tilde{\Omega}\p{\frac{2^r(\log(d/\delta))}{\beta^2} + \frac{2^r \log(1/\delta)}{\zeta}}$ then
\begin{equation*}
  \Prob_{\train_n \sim \calD^n}\p{\Exp_{\vecx \sim \calD_x}[(m_{n, s}(\vecx) - m(\vecx))^2] \ge \tilde{\Omega}\p{\frac{2^r(\log(d/(\delta \cdot \delta')))}{n \cdot \beta^2} + \frac{2^r \log(1/(\delta \cdot \delta'))}{n \cdot \zeta}} + \delta} \le \delta'.
\end{equation*}
From the above we get Theorem
\ref{thm:finalRecoverySubmodularIndependentBalancedProperLevelSplitFiniteSampleFullyGrown} by setting
$\delta = \tilde{\Omega}\p{\frac{2^r(\log(d/\delta'))}{n \cdot \beta^2} + \frac{2^r \log(1/\delta')}{n \cdot \zeta}}$.
 \newpage

\section{Proofs for Breiman's Algorithm} \label{sec:app:proofs:Breiman}

  In this Section we present the proof of the Theorem
\ref{thm:finalRecoverySubmodularIndependentBalancedProperBreimanFiniteSample}
and the Theorem \ref{thm:finalRecoverySubmodularIndependentBalancedProperBreimanFiniteSampleFullyGrown}.
We start with a proof about the bias of the trees that are produced by the Algorithm
\ref{alg:mainBreimanFiniteSample} and then we show how we can combine this with a bias-variance decomposition
and bounds on the variance part.

\subsection{Bounding The Bias} \label{sec:app:proofs:Breiman:Bias}

  We start with definitions of the empirical mean squared error for a given
partition $\calP$ and the empirical mean squared error of a leaf for a particular
leaf $A$. For the derivations below, we remind the following definitions from
Section~\ref{sec:Breiman}:
for a cell $A$, we define the set $\calZ_n(A)$, as the subset of the training set
$\calZ_n(A) = \{j \mid \vecx^{(j)} \in A\}$ and we define the partition $\calU_n(\calP)$ of the
training set $\train_n$ as $\calU_n(\calP) = \{\calZ_n(A) \mid A \in \calP\}$.

\begin{align} \label{eq:MSEFiniteSampleBinaryBreimanFull}
  L_n(\calP) & \triangleq \frac{1}{n} \sum_{j \in [n]} \p{y^{(j)} - m_n(\vecx^{(j)}; \calP)}^2 \\
  & = \frac{1}{n} \sum_{j \in [n]} \p{y^{(j)}}^2 + \frac{1}{n} \sum_{j \in [n]} m^2_n(\vecx^{(j)}; \calP) - 2 \sum_{j \in [n]} \frac{1}{n} y^{(j)} m_n(\vecx^{(j)}; \calP) \nonumber \\
  & = \frac{1}{n} \sum_{j \in [n]} \p{y^{(j)}}^2 + \sum_{Z \in \calZ_n(\calP)} \frac{\abs{Z}}{n} m^2_n(\vecx^{(Z)}; \calP) - 2 \sum_{Z \in \calZ_n(\calP)} \frac{\abs{Z}}{n} \p{\sum_{j \in Z} \frac{1}{\abs{Z}} y^{(j)}} m_n(\vecx^{(Z)}; \calP) \nonumber \\
  & = \frac{1}{n} \sum_{j \in [n]} \p{y^{(j)}}^2 - \frac{1}{n} \sum_{j \in [n]} m^2_n(\vecx^{(j)}; \calP) \nonumber \\
  & \triangleq \frac{1}{n} \sum_{j \in [n]} \p{y^{(j)}}^2 - V_n(\calP). \label{eq:MSEFiniteSampleBinaryBreimanFull:2}
\end{align}

\begin{align} \label{eq:MSEFiniteSampleBinaryBreimanCell}
  L^{\ell}_n(A) & \triangleq \frac{1}{N_n(A)} \sum_{j \in \calZ_n(A)} \p{y^{(j)} - m_n(\vecx^{(j)}; A)}^2 \\
  & = \frac{1}{N_n(A)} \sum_{j \in \calZ_n(A)} \p{y^{(j)}}^2 + \frac{1}{N_n(A)} \sum_{j \in \calZ_n(A)} m^2_n(\vecx^{(j)}; A) - 2 \sum_{j \in \calZ_n(A)} \frac{1}{N_n(A)} y^{(j)} \cdot m_n(\vecx^{(j)}; A) \nonumber \\
  & = \frac{1}{N_n(A)} \sum_{j \in \calZ_n(A)} \p{y^{(j)}}^2 + m^2_n(\vecx^{(\calZ_n(A))}; A) - 2 \p{\sum_{j \in \calZ_n(A)} \frac{1}{N_n(A)} y^{(j)}} m_n(\vecx^{(Z_n(A))}; A) \nonumber \\
  & = \frac{1}{N_n(A)} \sum_{j \in \calZ_n(A)} \p{y^{(j)}}^2 - m^2_n(\vecx^{(j)}; A) \nonumber \\
  & \triangleq \frac{1}{N_n(A)} \sum_{j \in Z_n(A)} \p{y^{(j)}}^2 - V^{\ell}_n(A), \label{eq:MSEFiniteSampleBinaryBreimanCell:2}
\end{align}

  We first prove a technical lemma for the concentration of the function $V^{\ell}_n$
around the function $\bar{V}_{\ell}$ uniformly over the set of sufficiently large cells $A$. Observe that $\bar{V}_{\ell}$ is not the expected
value of $V^{\ell}_n$ and hence this concentration bound is not a trivial one. We first define formally the notion of a large cell and then provide the uniform concentration bound.

\begin{definition}[\textbf{Large Cells}]
    Let $\calA(q, \zeta)$ be the set of $(q, \zeta)$-large cells defined as:
  $A \in \calA(q, \zeta)$ if and only if $A\subseteq \{0,1\}^d$, $\abs{A} \ge 2^{d - q}$ and
  $\Prob_{\vecx \sim \calD_x}\p{\vecx \in A} \ge \zeta/2^q$.
\end{definition}

\begin{lemma} \label{lem:concentrationOfVarianceFunctionBreiman}
    If $d > 1$, $r \in [d]$, $t > 1$ and
  $n \ge \frac{2^{3 + q}}{\zeta} \p{q \log(d) + t}$ then we have that
  \begin{align*}
    \Prob_{\train_n \sim \calD^n} \p{\sup_{A \in \calA(q, \zeta)} \abs{V^{\ell}_n(A, i) - \bar{V}_{\ell}(A, i)} \ge \sqrt{\frac{2^{18 + q}(q \log(d) + t)}{\zeta \cdot n}}} \le \exp\p{- t}.
  \end{align*}
\end{lemma}

\begin{proof}
    For the purpose of the proof we will define the following
  function that interpolates between the sample based function
  $V^{\ell}_n$ and the population based function $\bar{V}_{\ell}$.
  \begin{align}
            J_n(A, i) & \triangleq \sum_{z \in \{0, 1\}} \frac{N_n(A^i_z)}{N_n(A)} \p{\Exp_{(\vecx, y) \sim \calD}\b{y \mid \vecx \in A^i_z}}^2 \label{eq:definition:empiricalMeanExactMeanEmpiricalFunction:Breiman}
  \end{align}

\noindent First we bound the difference
  $\abs{V^{\ell}_n(A, i) - J_n(A, i)}$ in the following claim.

  \begin{claim} \label{clm:empiricalMean_EmpiricalvsExactMean_EmpiricalFunction:Breiman}
      If $d > 1$, $q \in [d]$, $t > 1$, and
    $n \ge \frac{2^q}{\zeta} \p{4 q \log d + t}$ then we have that
    \begin{align*}
      \Prob_{\train_n \sim \calD^n} \p{\sup_{A \in \calA(q, \zeta)} \abs{V^{\ell}_n(A, i) - J_n(A, i)} \ge  \sqrt{\frac{2^{8 + q}(q \log(d) + t)}{\zeta \cdot n}}} \le \exp\p{- t}.
    \end{align*}
  \end{claim}

  \begin{proof}
    We start by fixing a specific cell $A \in \calA(q, \zeta)$. This cell $A$ is fixed before we observe the
    training samples $\train_n$. We have that
      \begin{align*}
        \abs{V^{\ell}_n(A, i) - J_n(A, i)} & = \abs{\sum_{z \in \{0, 1\}} \frac{N_n(A^i_z)}{N_n(A)} \p{\p{\sum_{j \in \calZ_n(A^i_z)} \frac{1}{N_n(A^i_z)} y^{(j)}}^2 - \Exp_{(\vecx, y) \sim \calD}\b{y \mid \vecx \in A^i_z}^2}} \\
        \intertext{Since $m(\vecx) \in \b{-\frac{1}{2}, \frac{1}{2}}$ and $\eps \in \b{-\frac{1}{2}, \frac{1}{2}}$, we have that $y \in [-1, 1]$ and hence}
        & \le 4 \sum_{z \in \{0, 1\}} \frac{N_n(A^i_z)}{N_n(A)} \abs{\p{\sum_{j \in \calZ_n(A^i_z)} \frac{1}{N_n(A^i_z)} y^{(j)} - \Exp_{(\vecx, y) \sim \calD}\b{y \mid \vecx \in A^i_z}}}.
      \end{align*}

        Now let $\calY(A)$ be the distribution of the random variable $y$ conditional on the fact
    that the random variable $\vecx$ lies in the cell $A$. Observe that since $A$ is a cell
    fixed before observing the training set $\train_n$, the samples $y^{(j)}$ for
    $j \in \calZ_n(A)$ are i.i.d. samples from the distribution $\calY(A)$, conditional on
    the event that $\vecx^{(j)}$ is in $A$. We define $\calQ(A, K)$ to be the event that
    $\calZ_n(A) = K$ where $K \subseteq [n]$ and we have the following using Hoeffding's inequality.
    \begin{align*}
        \Prob_{\train_n \sim \calD^n} \p{\abs{\sum_{j \in \calZ_n(A^i_z)} \frac{1}{N_n(A^i_z)} y^{(j)} - \Exp_{(\vecx, y) \sim \calD}\b{y \mid \vecx \in A^i_z}} \ge \sqrt{\frac{t + \ln(2)}{2 N_n(A^i_z)}} \mid \calQ(A^i_z, K^z)} \le e^{-t},
    \end{align*}
    \noindent where $z \in \set{0, 1}$ and $K^0$, $K^1$ are two disjoint subsets of $[n]$.
    Observe that conditional on $\calQ(A, K^i)$ the number $N_n(A^i_z)$ is equal to $\abs{K^i}$ and
    hence is not a random variable any more.
    Then from union bound we have that if we condition on $\calQ(A^i_0, K^0) \cap \calQ(A^i_1, K^1)$ then we have that
    \begin{align*}
        \Prob_{\train_n \sim \calD^n} \p{\bigvee_{z \in \{0,1\}} \abs{\sum_{j \in \calZ_n(A^i_z)} \frac{1}{N_n(A^i_z)} y^{(j)} - \Exp_{(\vecx, y) \sim \calD}\b{y \mid \vecx \in A^i_z}} \ge \sqrt{\frac{(2 \ln(2) + t)}{2 N_n(A^i_z)}}}
        \le e^{-t},
    \end{align*}
    \noindent where we dropped the condition on $\calQ(A^i_0, K^0) \cap \calQ(A^i_1, K^1)$ from the
    above notation for simplicity of exposition. Hence we have the following which holds again if we
    condition on the event $\calQ(A^i_0, K^0) \cap \calQ(A^i_1, K^1)$.
    \begin{align} \label{eq:proof:claim:yConcentration:simple:Breiman}
        \Prob_{\train_n \sim \calD^n} \p{\abs{V^{\ell}_n(A, i) - J_n(A, i)} \ge \frac{\sqrt{2 (2 \ln(2) + t)} \sum_{z \in \{0, 1\}} \sqrt{N_n(A^i_z)}}{N_n(A)}} \le \exp\p{- t}.
    \end{align}
    \noindent But we know that $\sum_{z \in \{0, 1\}} \abs{N_n(A^i_z)} = N_n(A)$, and also we have
    that the for any vector $\vecw \in \R^k$ it holds that
    $\norm{\vecw}_1 \le \sqrt{k} \norm{\vecw}_2$. Therefore we have that
    \begin{align*}
        \frac{\sum_{z \in \{0, 1\}} \sqrt{N_n(A^i_z)}}{N_n(A)} \le  \sqrt{\frac{2}{N_n(A)}}.
    \end{align*}
    \noindent  Now using this in inequality~\eqref{eq:proof:claim:yConcentration:simple:Breiman} and
    taking the expectation over $\calD^n$ conditional on the event $\calR(A, k)$ that is equal to the event that $N_n(A) = k$ and by the law of total expectation
    we get the following
    inequality for any a priori fixed cell $A$.
    \begin{align} \label{eq:proof:claim:yConcentration:medium:Breiman}
        \Prob_{\train_n \sim \calD^n} \p{\abs{V^{\ell}_n(A, i) - J_n(A, i)} \ge \sqrt{\frac{4(2 \ln(2) + t)}{k}} \mid \calR(A, k)} \le \exp\p{- t}.
    \end{align}
    \noindent Since $A$ is a cell of size at least $2^{d - q}$ we have that there exists
    a set $Q_A \subseteq [d]$ with $q_A = \abs{Q_A} \le q$ and a vector
    $\vecw_A \in \{0, 1\}^{q_A}$ such that $\vecx \in A \Leftrightarrow \vecx_{Q_A} = \vecw_A$.
    Therefore by the assumption that $A \in \calA(q, \zeta)$,
    we have that
    $\Prob_{\vecx}\p{\vecx \in A} \ge \frac{\zeta}{2^q}$.
    Hence from classical Chernoff bound for binary random variables we have that
    \begin{align} \label{eq:proof:claim:yConcentration:LeafSize:Breiman}
      \Prob_{\train_n \sim \calD^n} \p{N_n(A) \le \frac{n \cdot \zeta}{8 \cdot 2^q}} \le \exp\p{-\frac{n \cdot \zeta}{2^q}}
    \end{align}
    \noindent Then by combining \eqref{eq:proof:claim:yConcentration:medium:Breiman} and
    \eqref{eq:proof:claim:yConcentration:LeafSize:Breiman} and invoking the Bayes rule we get that
    \begin{align} \label{eq:proof:claim:yConcentration:medium:prime:Breiman}
        \Prob_{\train_n \sim \calD^n} \p{\abs{V^{\ell}_n(A, i) - J_n(A, i)} \ge \sqrt{\frac{32 (2 \ln(2) + t) \cdot 2^q}{\zeta \cdot n}}} \le \exp\p{- t} + \exp\p{-\frac{n \cdot \zeta}{2^q}}.
    \end{align}
    \noindent It is also easy to see that the number of possible cells
    $A$ with size at least $2^{d - q}$ is at most
    $2^q \cdot \p{\sum_{i = 0}^q \binom{d}{i}}$ and hence
    $\abs{\calA(q, \zeta)} \le 2^q \cdot \p{\sum_{i = 0}^q \binom{d}{i}}$. Now using a union
    bound of \eqref{eq:proof:claim:yConcentration:medium:prime:Breiman}
    over all possible cells $A \in \calA(q, \zeta)$ we get
    \begin{align} \label{eq:proof:claim:yConcentration:almostFinal:Breiman}
        & \Prob_{\train_n \sim \calD^n} \p{\sup_{A \in \calA(q, \zeta)} \abs{V^{\ell}_n(A, i) - J_n(A, i)} \ge  \sqrt{\frac{32(2 \ln(2) + t) \cdot 2^q}{\zeta \cdot n}}} \le \nonumber \\
        & ~~~~~~~~~~~~~~~~~~~~~~~~~~~~~~~~~ \le \p{2^q \sum_{i = 0}^q \binom{d}{i}} \p{\exp\p{- t} + \exp\p{-\frac{n \cdot \zeta}{2^q}}}.
    \end{align}
    \noindent Finally using
    $\log\p{2^q \cdot \sum_{i = 0}^q \binom{d}{i}} \le (q + 1) \log(2 \cdot d \cdot q)$ and since $d > 1$, $q \in [d]$, the claim follows.
  \end{proof}

    Next we bound the difference
    $\abs{J_n(A, i) - \bar{V}_{\ell}(A, i)}$.

  \begin{claim} \label{clm:empiricalvsExactMean_ExactMeanExactFunction:Breiman}
      If $d > 1$, $q \in [d]$, $t > 1$, and
    $n \ge \frac{2^q}{\zeta} \p{4 q \log d + t}$ then we have that
    \begin{align*}
        \Prob_{\train_n \sim \calD^n} \p{\sup_{A \in \calA(q, \zeta)} \abs{J_n(A, i) - \bar{V}_{\ell}(A, i)} \ge \sqrt{\frac{2^{10 + q} \cdot \p{q \log(d) + t}}{n \cdot \zeta}}} \le \exp\p{- t}.
    \end{align*}
  \end{claim}

  \begin{proof}
      Since the error distribution has zero mean, i.e. $\Exp[\eps] = 0$,
    we have that
    \[ \p{\Exp_{(\vecx, y) \sim \calD}\b{y \mid \vecx \in A}}^2 = \p{\Exp_{\vecx}\b{m(\vecx) \mid \vecx \in A}}^2 \triangleq M(A) \]
    \noindent and hence
    \begin{align*}
      \abs{J_n(A, i) - \bar{V}_{\ell}(A, i)} & \le \sum_{z \in \{0, 1\}} \abs{ \frac{N_n(A^i_z)}{N_n(A)} - \Prob_{\vecx}\p{\vecx \in A^i_z \mid \vecx \in A}} \p{\Exp_{(\vecx, y) \sim \calD}\b{y \mid \vecx \in A^i_z}}^2 \\
      & \le 4 \sum_{z \in \{0, 1\}} \abs{ \frac{N_n(A^i_z)}{N_n(A)} - \Prob_{\vecx}\p{\vecx \in A^i_z \mid \vecx \in A}}.
    \end{align*}
    \noindent Now from Hoeffding bound we have that
    \begin{align*}
      \Prob_{\train_n \sim \calD^n}\p{\left.\abs{\frac{N_n(A^i_z)}{k} - \Prob_{\vecx}\p{\vecx \in A^i_z \mid \vecx \in A}} \ge t ~\right| N_n(A) = k} \le \exp\p{ - 2 \cdot k \cdot t^2}
    \end{align*}
    \noindent excluding the case
    $N_n(A) \le \frac{n \cdot \zeta}{16 \cdot 2^q}$ and taking
    expectation of both sides we get that
    \begin{align*}
      \Prob_{\train_n \sim \calD^n}\p{\abs{\frac{N_n(A^i_z)}{N_n(A)} - \Prob_{\vecx}\p{\vecx \in A^i_z \mid \vecx \in A}} \ge t} \le &  \Exp_{\train_n \sim \calD^n}\b{\exp\p{ - 2 t^2 \cdot N_n(A)} \mid N_n(A) > \frac{n \cdot \zeta}{16 \cdot 2^q}} \\
      & ~~~~ + \Prob_{\train_n \sim \calD^n}\p{N_n(A) \le \frac{n \cdot \zeta}{16 \cdot 2^q}}
    \end{align*}
    \noindent now we can use
    \eqref{eq:proof:claim:yConcentration:LeafSize:Breiman} to get that
    \begin{align*}
      \Prob_{\train_n \sim \calD^n}\p{\abs{\frac{N_n(A^i_z)}{N_n(A)} - \Prob_{\vecx}\p{\vecx \in A^i_z \mid \vecx \in A}} \ge \sqrt{\frac{8 \cdot 2^q \cdot t}{n \cdot \zeta}}} \le \exp\p{-t} + \exp\p{-\frac{n \cdot \zeta}{2^q}}
    \end{align*}
    \noindent Finally if we apply the union bound over all possible cells
    $A \in \calA(q, \zeta)$ together with
    the assumption that $n \ge \frac{2^q}{\zeta} \p{4 q \log d + t}$, the claim follows.
  \end{proof}

  \noindent If we combine Claim
  \ref{clm:empiricalMean_EmpiricalvsExactMean_EmpiricalFunction:Breiman}
  and \ref{clm:empiricalvsExactMean_ExactMeanExactFunction:Breiman}, the
  lemma follows.

\end{proof}

  We are now ready to prove that the bias of every tree that is constructed by Algorithm
\ref{alg:mainBreimanFiniteSample} is small under the Assumption
\ref{asp:diminishingReturnsBreimanFull}. We start by proving the finite sample analogues of
Lemma \ref{lem:optimalityWithoutSplitsPopulationBreiman}. First we provide a relaxed
version of the Definition \ref{def:relevantVariablesForCellBreiman} and a version with
finite samples.

\begin{definition} \label{def:relevantVariablesForCellBreimanFiniteSample}
    Given a partition $\calP$ and a cell $A$ of $\calP$, a positive
  number $\eta$ and a training set $\train_n$, we define the sets
  $\calL^\eta_n(A) = \{i \in [d] \mid V^{\ell}_n(A, i) - V^{\ell}_n(A) > \eta \}$
  and $\calL^\eta(A) = \{i \in [d] \mid \bar{V}_{\ell}(A, i) - \bar{V}_{\ell}(A) > \eta\}$.
  Similarly, when $\eta = 0$ we use the simpler notation $\calL_n(A)$ and $\calL(A)$.
\end{definition}

  Since $V_n$ is monotone decreasing with respect to $\calP$,
we have that $V_n(\calP, A, i) \ge V_n(\calP)$. Hence given $\calP$, $A$
the Algorithm \ref{alg:mainBreimanFiniteSample} chooses the direction $i$
that maximizes the positive quantity $V_n(\calP, A, i) - V_n(\calP)$. So the bad
event is that for all $j \in [d]$,
$V_n(\calP, A, i) - V_n(\calP) \ge V_n(\calP, A, j) - V_n(\calP)$ but
$\bar{V}(\calP, A, i) - \bar{V}(\calP) = 0$ and there exists a $k \in [d]$ such that
$\bar{V}(\calP, A, k) - \bar{V}(\calP) > 0$. A relaxed version of this bad event can be described
using the Definition \ref{def:relevantVariablesForCellBreimanFiniteSample}. In this language the
bad event is that the index $i \in [d]$ that the Algorithm \ref{alg:mainBreimanFiniteSample}
chooses to split does not belong to $\calL^{\eta}(\calP, A)$ although
$\calL^{\eta}(\calP, A) \neq \emptyset$. We bound the probability of this event in the next lemma.

\begin{lemma} \label{lem:splitRelevantFiniteSampleBreiman}
    If $d > 1$, $q \in [d]$, $t > 1$, and
  $n \ge \frac{2^{3 + q}}{\zeta} \p{q \log(d) + t}$ and let
  $\eta = \sqrt{\frac{2^{18 + q}(q \log(d) + t)}{\zeta \cdot n}}$ then we have that
  \begin{align*}
    \Prob_{\train_n \sim \calD^n} \p{\bigvee_{A \in \calA(q, \zeta)} \p{\p{\argmax_{i \in [d]} V^{\ell}_n(A, i)} \not\in \calL(A) \mid \calL^{2 \eta}(A) \neq \emptyset}} \le 2 \exp(- t)
  \end{align*}
\end{lemma}

\begin{proof}
    Directly applying Lemma
  \ref{lem:concentrationOfVarianceFunctionBreiman} to Definition
  \ref{def:relevantVariablesForCellBreimanFiniteSample} we have that
  \begin{align} \label{eq:proof:chooseBadCoordinatesBound:1:Breiman}
    \Prob_{\train_n \sim \calD^n} \p{\bigvee_{i, A \in \calA(q, \zeta)} \p{i \in \calL^{\eta}_n(A) \mid i \not\in \calL(A)}} \le \exp(- t),
  \end{align}
  \noindent where $\eta = \sqrt{\frac{2^{18 + q}(q \log(d) + t)}{\zeta \cdot n}}$ and
  $n \ge \frac{2^{3 + q}}{\zeta} \p{q \log(d) + t}$. Similarly we have that
  \begin{align} \label{eq:proof:chooseBadCoordinatesBound:2:Breiman}
    \Prob_{\train_n \sim \calD^n} \p{\bigvee_{i, A \in \calA(q, \zeta)} \p{i \not\in \calL^{\eta}_n(A) \mid i \in \calL^{2 \eta}(A)}} \le \exp(- t).
  \end{align}
  \noindent If we combine the above inequalities we get that there is a
  very small probability that there exists an index
  $i \in \calL^{2 \eta}(A)$ but an index $j \not\in \calL(A)$ is chosen
  instead. This is summarized in the following inequality
  \begin{align} \label{eq:proof:chooseBadCoordinatesBound:combined:Breiman}
    \Prob_{\train_n \sim \calD^n} \p{\bigvee_{A \in \calA(q, \zeta)} \p{\p{\argmax_{i \in [d]} V^{\ell}_n(A, i)} \not\in \calL(A) \mid \calL^{2 \eta}(A) \neq \emptyset}} \le 2 \exp(- t)
  \end{align}
  \noindent and the lemma follows.
\end{proof}

\begin{lemma} \label{lem:optimalityWithoutSplitsFiniteSampleBreiman}
    For every partition $\calP$, under the Assumption
  \ref{asp:diminishingReturnsBreimanFull}, if for some $A \in \calP$ it holds
  that $\calL^{\eta}(A) = \emptyset$, then
  \[ \Exp_{\vecx \sim \calD_x}\b{\p{m(\vecx) - \Exp_{\vecz \sim \calD_x}\b{m(\vecz) \mid \vecz \in \calP(\vecx)}}^2 ~ \mid ~ \vecx \in A} \le C \cdot \eta \cdot \abs{\calL(A)}. \]
\end{lemma}

\begin{proof}
    We define $\bar{\calP}$ to be the partition where all the cells contain only one element of the
  space, that is $\bar{\calP} = \set{\set{\vecx} \mid \vecx \in \{0, 1\}^d}$. We then know that
  $\bar{L}(\bar{\calP}) = 0$.
  Let $\calP'$ be the refinement of $\calP$ where every cell $A$ of $\calP$ has been split with respect to
  all the coordinates $\calR(A)$. We first prove that $\bar{L}(\calP') = 0$. If this is not the case
  then there exists a cell $A$ and a direction $i \not\in \calR(A)$ such that
  $\bar{L}(\calP', A, i) - \bar{L}(\calP') < 0$. But because of the Assumption
  \ref{asp:diminishingReturnsBreimanFull} and item (4.) of Lemma
  \ref{lem:varianceVsLeafVariance} we have
  \[ C \cdot (\bar{L}(\calP, A, i) - \bar{L}(\calP)) \le \bar{L}(\calP', A, i) - \bar{L}(\calP') < 0 \]
  \noindent which contradicts the assumption that $i \not\in \calR(A)$.

  Now assume that
  $\calL^{\eta}(A) = \emptyset$ and for the sake of
  contradiction also assume that $\bar{L}_{\ell}(A) > C \cdot \eta \cdot \abs{\calL(A)}$.
  Let $\{r_1, \dots, r_k\}$ be an arbitrary enumeration of the set $\calL(A)$. From the
  argument before we have that $\bar{L}_{\ell}(A, \calL(A)) = 0$ hence there exists
  an element $r_j$ of $\calL(A)$ such that
  \[ \bar{L}_{\ell}(A, \{r_1, \dots, r_{j - 1}\}) - \bar{L}_{\ell}(A, \{r_1, \dots, r_j\})  > C \cdot \eta, \]
  \noindent otherwise we would immediately have
  $\bar{L}_{\ell}(A) \le C \cdot \eta \cdot \abs{\calL(A)}$. But because of the diminishing
  returns property of $\bar{L}(\cdot)$ we have that
  \[ C \cdot \p{\bar{L}_{\ell}(A) - \bar{L}_{\ell}(A, r_j)} \ge \bar{L}_{\ell}(A, \{r_1, \dots, r_{j - 1}\}) - \bar{L}_{\ell}(A, \{r_1, \dots, r_j\}) > C \cdot \eta, \]
  but this last inequality implies $r_j \in \calL^{\eta}(A)$ which contradicts
  with our assumption that $\calL^{\eta}(A) = \emptyset$. Hence
  $\bar{L}_{\ell}(A) \le C \cdot \eta \cdot \abs{\calL(A)}$ and the lemma
  follows.
\end{proof}

  Finally we need one more Lemma to handle the case where Assumption
\ref{asp:strongPartitionSparsity} holds.

\begin{lemma} \label{lem:splitRelevantFiniteSampleBreimanStrongPartitionSparsity}
    If $d > 1$, $r \in [d]$, $t > 1$, assume that
  $m$ is $(\beta, r)$-strongly partition sparse with relevant features $R$ as per Assumption
  \ref{asp:strongPartitionSparsity} and let
  $n \ge \frac{2^{3 + r}}{\zeta} \p{r \log(d) + t}$ and $n \ge \frac{2^{18 + r}(r \log(d) + t)}{\zeta \cdot \beta^2}$ then we have that
  \begin{align*}
    \Prob_{\train_n \sim \calD^n} \p{\bigvee_{A \in \calA(r, \zeta)} \p{\p{\argmax_{i \in [d]} V^{\ell}_n(A, i)} \not\in R \mid \calL(A) \neq \emptyset}} \le 2 \exp(- t)
  \end{align*}
\end{lemma}
\begin{proof}
    Directly applying Lemma
  \ref{lem:concentrationOfVarianceFunctionBreiman} to Definition
  \ref{def:relevantVariablesForCellBreimanFiniteSample} we have that
  \begin{align} \label{eq:proof:chooseBadCoordinatesBound:1:Breiman:prime}
    \Prob_{\train_n \sim \calD^n} \p{\bigvee_{i, A \in \calA(r, \zeta)} \p{i \in \calL^{\eta}_n(A) \mid i \not\in \calL(A)}} \le \exp(- t),
  \end{align}
  \noindent where $\eta = \sqrt{\frac{2^{26 + r}(r \log(d) + t)}{\zeta \cdot n}}$ and
  $n \ge \frac{2^{3 + r}}{\zeta} \p{r \log(d) + t}$. Similarly we have that
  \begin{align} \label{eq:proof:chooseBadCoordinatesBound:2:Breiman:prime}
    \Prob_{\train_n \sim \calD^n} \p{\bigvee_{i, A \in \calA(r, \zeta)} \p{i \not\in \calL^{\eta}_n(A) \mid i \in \calL^{2 \eta}(A)}} \le \exp(- t).
  \end{align}
  \noindent If we combine the above inequalities with Assumption
  \ref{asp:strongPartitionSparsity} the lemma follows.
\end{proof}

\begin{theorem}
  \label{thm:recoverySubmodularIndependentBalancedProperBreimanFiniteSample}
    Let $\train_n$ be i.i.d. samples from the non-parametric regression model
  $y = m(\vecx) + \eps$, where $m(\vecx) \in [-1/2, 1/2]$, $\eps \sim \calE$,
  $\Exp_{\eps \sim \calE}[\eps] = 0$ and $\eps \in [-1/2, 1/2]$ with $m$ an $r$-sparse
  function. Let also $\calP_n$ be the partition that Algorithm~\ref{alg:mainBreimanFiniteSample}
  returns. Then the following statements hold:
  \begin{Enumerate}
    \item Let $q = \frac{C \cdot r}{C \cdot r + 3}\p{\log(n) - \log(\log(d/\delta))}$ and
    assume that the approximate diminishing returns Assumption
    \ref{asp:diminishingReturnsBreimanFull} holds. Moreover if we set the number of nodes $t$
    such that $\log(t) \ge q$, and if we have number of samples
    $n \ge \tilde{\Omega}\p{\log(d/\delta)}$ then it holds that
    \begin{align*}
      \Prob_{\train_n \sim \calD^n}\p{\Exp_{\vecx \sim \calD_x} \b{\p{m(\vecx) - \Exp_{\vecz \sim \calD_x}\b{m(\vecz) \mid \vecz \in \calP_n(\vecx)}}^2} > \tilde{\Omega}\p{C \cdot r \cdot \sqrt[C \cdot r + 3]{\frac{\log(d/\delta)}{n}}}} \le \delta.
    \end{align*}
    \item Suppose that the distribution $\calD_x$ is a product distribution
    (see Assumption \ref{asp:independentFeature}) and that the Assumption
    \ref{asp:diminishingReturnsBreimanFull} holds. If $\log(t) \ge r$, then it holds
    that
    \begin{align*}
      \Prob_{\train_n \sim \calD^n}\p{\Exp_{\vecx \sim \calD_x} \b{\p{m(\vecx) - \Exp_{\vecz \sim \calD_x}\b{m(\vecz) \mid \vecz \in \calP_n(\vecx)}}^2} > \tilde{\Omega}\p{\sqrt[3]{C^2 \cdot \frac{2^r \cdot \log(d/\delta))}{n}}}} \le \delta.
    \end{align*}
    \item Suppose that the distribution $\calD_x$ is a product distribution
    (see Assumption \ref{asp:independentFeature}), that is also $(\zeta, r)$-lower bounded
    (see Assumption \ref{asp:densityLowerBound}) and that the Assumption
    \ref{asp:diminishingReturnsBreimanFull} holds. If $\log(t) \ge r$, then it holds
    that
    \begin{align*}
      \Prob_{\train_n \sim \calD^n}\p{\Exp_{\vecx \sim \calD_x} \b{\p{m(\vecx) - \Exp_{\vecz \sim \calD_x}\b{m(\vecz) \mid \vecz \in \calP_n(\vecx)}}^2} > \tilde{\Omega}\p{C \cdot \sqrt{\frac{2^r \cdot \log(d/\delta))}{\zeta \cdot n}}}} \le \delta.
    \end{align*}
    \item Suppose that $m$ is $(\beta, r)$-strongly partition sparse (see Assumption
    \ref{asp:strongPartitionSparsity}) and that $\calD_x$ is $(\zeta, r)$-lower bounded
    (see Assumption \ref{asp:densityLowerBound}). If
    $n \ge \tilde{\Omega}\p{\frac{2^r(\log(d/\delta))}{\zeta \cdot \beta^2}}$, and
    $\log(t) \ge r$, then we have
    \begin{align*}
      \Prob_{\train_n \sim \calD^n}\p{\Exp_{\vecx \sim \calD_x} \b{\p{m(\vecx) - \Exp_{\vecz \sim \calD_x}\b{m(\vecz) \mid \vecz \in \calP_n(\vecx)}}^2} > 0} \le \delta.
    \end{align*}
  \end{Enumerate}
\end{theorem}

\begin{proof}
    When the value of $\level$ changes, then the algorithm considers
  separately every cell $A$ in $\calP_{\level}$. For every such cell $A$
  it holds that $\bar{L}_{\ell}(A, R) = 0$ and hence
  $\bar{V}_{\ell}(A, R) \triangleq V^*(A)$ is maximized. Since $m(\vecx) \in [-1, 1]$ it
  holds that the maximum value of $\bar{V}_{\ell}$ is $1$. Now let $\{i_1, \dots, i_r\}$
  be an arbitrary enumeration of $R$ and let $R_j = \{i_1, \dots, i_j\}$. Then by adding
  and subtracting terms of the form $\bar{V}_{\ell}(A, R_j)$ we have the following equality
  \[ \p{\bar{V}_{\ell}(A, R) - \bar{V}_{\ell}(A, R_{r - 1})} + \cdots + \p{\bar{V}_{\ell}(A, R_2) - \bar{V}_{\ell}(A, i_1)} + \bar{V}_{\ell}(A, i_1) = V^*(A). \]
  \noindent From Assumption~\ref{asp:diminishingReturnsBreimanFull} we have that
  \[ \p{\bar{V}_{\ell}(A, i_r) - \bar{V}_{\ell}(A)} + \cdots + \p{\bar{V}_{\ell}(A, i_2) - \bar{V}_{\ell}(A)} + \p{\bar{V}_{\ell}(A, i_1) - \bar{V}_{\ell}(A)} \ge \frac{V^*(A) - \bar{V}_{\ell}(A)}{C} \]
  \noindent which implies
  \[ \max_{j \in [r]} \p{\bar{V}_{\ell}(A, i_j) - \bar{V}_{\ell}(A)} \ge \frac{V^*(A) - \bar{V}_{\ell}(A)}{C \cdot r}. \]
  \noindent Let $i_A^{\level}$ be the coordinate that the algorithm
  chose to split at cell $A$ at level $\level$. Now from the greedy
  criterion that Algorithm~\ref{alg:mainBreimanFiniteSample} uses to pick the next coordinate to split on, we have that $i^{\level}_A$ was at least as good as the  best of the coordinates
  in $R$ with respect to $V^{\ell}_n$. If we set
  $\zeta = \tilde{\Theta}\p{\sqrt[C \cdot r + 3]{\frac{\log(d/\delta)}{n}}}$,
  $q = \frac{C \cdot r}{C \cdot r + 3}\p{\log(n) - \log(\log(d/\delta))}$ and
  $\xi = \tilde{\Theta}\p{\sqrt[C \cdot r + 2]{\frac{\log(d/\delta)}{\zeta \cdot n}}}$
  and use Lemma \ref{lem:concentrationOfVarianceFunctionBreiman} we get that if
  $n \ge \tilde{\Omega}\p{\log(d/\delta)}$ \footnote{This condition is necessary to
  guarantee that we have enough number of samples $n$ to make the splits that are necessary for the $q$
  splits of every path of the tree.} and if
  $\Prob_{\vecx \sim \calD_x}\p{\vecx \in A} \ge \zeta/2^q$ then it holds with probability at least
  $1 - \delta$ that
  \[ \bar{V}_{\ell}\p{A, i^{\level}_A} \ge \bar{V}_{\ell}(A) + \frac{V^*(A) - \bar{V}_{\ell}(A)}{C \cdot r} - 3 \xi \]
  which in turn because
  $L^*(A) \triangleq \bar{L}_{\ell}(A, R) = 0$ implies that
  \begin{equation} \label{eq:potentialProgressFiniteSampleCellSplit:Breiman}
    \bar{L}_{\ell}\p{A, i^{\level}_A} \le \bar{L}_{\ell}(A) \p{1 - \frac{1}{C \cdot r}} + 3 \xi.
  \end{equation}
  \noindent We fix $\calQ_{\level}$ to be the partition $\calP_{\level}$ of $\{0, 1\}^d$ when $\level$
  changed. Also we define $\calU$ to be the set of cells $A$ in $\calQ_{\level}$ such that
  $\Prob_{\vecx \sim \calD_x}\p{\vecx \in A} \ge \zeta/2^q$ and $\calV$ the rest of the cells $A$ in
  $\calQ_{\level}$. Then because of \ref{eq:potentialProgressFiniteSampleCellSplit:Breiman} and Lemma
  \ref{lem:varianceVsLeafVariance} it holds that
  \begin{align*}
    & \bar{L}\p{\calQ_{\level + 1}} = \sum_{A \in \calQ_{\level}} \Prob_{\vecx \sim \calD_x}\p{\vecx \in A} \bar{L}_{\ell}\p{A, i^{\level}_A} + 3 \xi \\
    & ~~~~~~~~~~~~~~~~~~~~~~~~ = \sum_{A \in \calU} \Prob_{\vecx \sim \calD_x}(\vecx \in A) \bar{L}_{\ell}\p{A} \p{1 - \frac{1}{C \cdot r}} + \sum_{A \in \calV} \Prob_{\vecx \sim \calD_x}(\vecx \in A) \bar{L}_{\ell}\p{A, i^{\level}_A} + 3 \xi \\
    & ~~~~~~~~~~~~~~~~~~~~~~~~ \le \bar{L}\p{\calQ_{\level}} \p{1 - \frac{1}{C \cdot r}} + \frac{\zeta}{2^q} \sum_{A \in \calV} \bar{L}_{\ell}\p{A, i^{\level}_A} + 3 \xi \\
    & ~~~~~~~~~~~~~~~~~~~~~~~~ \le \bar{L}\p{\calQ_{\level}} \p{1 - \frac{1}{C \cdot r}} + 3 \xi + \zeta.
  \end{align*}
  \noindent Inductively and using the fact that $m(\vecx) \in [-1, 1]$,
  we get that
  \begin{equation} \label{eq:potentialProgressFiniteSampleLevel:Breiman}
    \bar{L}\p{\calQ_{\level}} \le \bar{L}\p{\calP_0} \p{1 - \frac{1}{C \cdot r}}^{\level} + (\xi + \zeta) \cdot \level \le \p{1 - \frac{1}{C \cdot r}}^{\level} + (\xi + \zeta) \cdot \level.
  \end{equation}
  \noindent Finally if we set $\eta = \xi + \zeta = \Theta(\xi)$ from the choice of $t$ we have
  that $\level \ge C \cdot r \ln\p{1/\eta}$ and hence when $\level$ is exactly
  equal to $C \cdot r \ln\p{1/\eta}$, it holds that
  $\bar{L}(\calQ_{\level}) \le 3 \cdot C \cdot r \ln\p{1/\eta} \eta$.
  Now from the
  monotonicity of the $\bar{L}$ function with respect to $\calQ_{\level}$ we have that for
  any $\level \ge C\, r \ln\p{1/\eta}$ it holds that
  $\bar{L}(\calQ_{\level}) \le 3 \cdot C \cdot r \ln\p{1/\eta} \eta$ and the first part of
  the theorem follows.

    For the second part of the theorem we use Lemma
  \ref{lem:splitRelevantFiniteSampleBreiman} and we have that at every step,
  if $A \in \calA(r, \zeta)$ then either the algorithm chose to split with
  respect to a direction $i \in \calL(A)$ or $\calL^{\eta}(A) = \emptyset$ for
  $\eta = \sqrt{\frac{2^{18 + r}(r \log(d) + t)}{\zeta \cdot n}}$. Because of
  our assumption that $m$ is $r$-sparse and because we assume that the features
  are distributed independently we have that at any step $\abs{\calL(A)} \le r$.
  Hence after $r$ levels it has to be that for every cell
  $A \in \calA(r, \zeta) \cap \calP_n$, there is a super-cell $B \supseteq A$
  such that $\calL^{\eta}(B) = \emptyset$. These super-cells create a partition
  $\tilde{\calP}_n$ that has the following two properties: (1) every cell $B$ of
  $\tilde{\calP}_n$ with $B \in \calA(r, \zeta)$ has
  $\calL^{\eta}(B) = \emptyset$, and (2) $\calP_n \sqsubseteq \tilde{\calP}_n$,
  i.e., $\calP_n$ is a refinement of $\tilde{\calP_n}$.
  Let now $\calU$ be the cells $B \in \tilde{\calP}_n$ such that
  $B \in \calA(r, \zeta)$, and let $\calV$ be the rest of the cells of
  $\tilde{\calP}_n$. Then using Lemma
  \ref{lem:optimalityWithoutSplitsFiniteSampleBreiman} we have that
  \begin{align*}
    \bar{L}(\tilde{\calP}_n) & = \sum_{B \in \calU} \Prob_{\vecx \sim \calD_x}(\vecx \in B) \cdot \bar{L}_{\ell}\p{B} + \sum_{B \in \calV} \Prob_{\vecx \sim \calD_x}(\vecx \in B) \cdot \bar{L}_{\ell}\p{B} \\
    & \le C \cdot \eta \cdot r + \zeta
    \intertext{now setting
    $\zeta = (C \cdot r)^{2/3} \sqrt[3]{\frac{2^{18 + r}(r \log(d) + t)}{n}}$ we get}
    \bar{L}(\tilde{\calP}_n) & \le 2 \zeta
  \end{align*}
  and since $\bar{L}(\calP)$ is a monotone function of $\calP$, we also get that
  $\bar{L}(\calP_n) \le 2 \zeta$ and the second part of the theorem follows.

      For the third part of the theorem we use Lemma
  \ref{lem:splitRelevantFiniteSampleBreiman} and we have that at every step
  either the algorithm chose to split with respect to a direction
  $i \in \calL(A)$ or $\calL^{\eta}(A) = \emptyset$ for
  $\eta = \sqrt{\frac{2^{18 + r}(r \log(d) + t)}{\zeta \cdot n}}$. Because of
  our assumption that $m$ is $r$-sparse and because we assume that the features
  are distributed independently we have that at any step $\abs{\calL(A)} \le r$.
  Hence after $r$ levels it has to be that for every cell
  $A \in \calA(r, \zeta) \cap \calP_n$, there is a super-cell $B \supseteq A$
  such that $\calL^{\eta}(B) = \emptyset$. These super-cells create a partition
  $\tilde{\calP}_n$ that has the following two properties: (1) every cell $B$ of
  $\tilde{\calP}_n$ has $\calL^{\eta}(B) = \emptyset$, and (2)
  $\calP_n \sqsubseteq \tilde{\calP}_n$, i.e., $\calP_n$ is a refinement of
  $\tilde{\calP}_n$. Then taking expectations over $\vecx \sim \calD_x$ in both sides of Lemma
  \ref{lem:optimalityWithoutSplitsFiniteSampleBreiman} we have that
  $\bar{L}(\tilde{\calP}_n) \le C \cdot \eta \cdot r$ and since $\bar{L}(\calP)$
  is a monotone function, the third part of the theorem follows.

    The last part of the theorem follows easily from Lemma
  \ref{lem:splitRelevantFiniteSampleBreimanStrongPartitionSparsity}.
\end{proof}

    Recall the definition of the value-diameter from Definition \ref{def:valueDiameter}. We
prove the following, diameter-based bound on the bias of regression trees.

\begin{theorem} \label{thm:biasStronglyPartitionSparseBreimanFiniteSample}
    Let $\train_n$ be i.i.d. samples from the non-parametric regression model
  $y = m(\vecx) + \eps$, where $m(\vecx) \in [-1/2, 1/2]$, $\eps \sim \calE$,
  $\Exp_{\eps \sim \calE}[\eps] = 0$ and $\eps \in [-1/2, 1/2]$. If $m$ is $(\beta, r)$-strongly
  partition sparse (see Assumption \ref{asp:strongPartitionSparsity}) and $\calD_x$ is
  $(\zeta, r)$-lower bounded (see Assumption \ref{asp:densityLowerBound}) then the following
  statements hold for the bias of the output of Algorithm \ref{alg:mainLevelSplitFiniteSample}.
  \begin{Enumerate}
    \item If $n \ge \tilde{\Omega}\p{\frac{2^r(\log(d/\delta))}{\zeta \cdot \beta^2}}$, and
    $\log(t) \ge r$, then it holds that
    \begin{align*}
      \Prob_{\train_n \sim \calD^n}\p{\Delta_m(\calP_n) = 0} \ge 1 - \delta.
    \end{align*}
    \item Let $R$ be the set of relevant features, $\vecx \in \{0, 1\}^d$ and assume that we run
    Algorithm \ref{alg:mainBreimanFiniteSample} with input $h = 1$ and $\log(t) \ge r$. If
    $n \ge \tilde{\Omega}\p{\frac{2^r(\log(d/\delta))}{\zeta \cdot \beta^2}}$ then it holds that
    \begin{align*}
      \p{\Exp_{\train_n \sim \calD^n}\b{m_{n}(\vecx; \calP_n)} - m(\vecx)}^2 \le \delta.
    \end{align*}
  \end{Enumerate}
\end{theorem}
\begin{proof}
    We begin with the first part of the theorem. Observe that our assumption that
  $\calD_x$ is $(\zeta, r)$-lower bounded implies that the set $\calA(\zeta, r)$
  contains all the cells that can be produced using at most $r$ splits. Therefore if
  we apply Lemma \ref{lem:splitRelevantFiniteSampleBreimanStrongPartitionSparsity} we
  conclude that all the splits that take place in the first $r$ levels of the
  algorithm choose a direction $i$ that belongs to the set of relevant features $R$.
  The latter implies that the set of cells after $r$ iterations is the precisely
  following
  \[ \set{\set{\vecx \mid \vecx_R = \vecz_R} \mid \vecz_R \in \set{0, 1}^r}. \]
  \noindent and therefore due to the monotonicity of $\bar{L}$ we have that
  \[ \Exp_{\vecx \sim \calD_x} \b{\p{m(\vecx) - \Exp_{\vecz \sim \calD_x}\b{m(\vecz) \mid \vecz \in \calP_n(\vecx)}}^2} \le \Exp_{\vecx \sim \calD_x} \b{\p{m(\vecx) - \Exp_{\vecz \sim \calD_x}\b{m(\vecz) \mid \vecz_R = \vecx_R}}^2} \]
  \noindent where the right hand side is equal to $0$ due to the sparsity assumption
  of $m$. Hence we have that $\bar{L}(\calP_n)$ is equal to $0$. Therefore from the
  Definition \ref{def:valueDiameter} the first part of the theorem follows.

    For the second part of the theorem we define for simplicity
  $w^{(j)}(\vecx) = \frac{\chara\{\vecx \in \calP_n(\vecx^{(j)})\}}{N_n(\calP_n(\vecx^{(j)})}$ and
  hence $m_{n}(\vecx) = \sum_{i = 1}^n w^{(j)}(\vecx) y^{(j)}$ and we have:
  \begin{align*}
    & \left(\Exp_{\train_n \sim \calD^n}[m_{n}(\vecx)] - m(\vecx)\right)^2
    = \\
    & ~~~~~~~~~~ = \left(\Exp_{\train_n \sim \calD^n}\left[\sum_{j = 1}^n w^{(j)}(\vecx)\, (y^{(j)} - m(\vecx^{(j)}))\right] +
    \Exp_{\train_n \sim \calD^n}\left[\sum_{j = 1}^n w^{(j)}(\vecx)\, (m(\vecx^{(j)}) - m(\vecx))\right]\right)^2
  \end{align*}
  Due to honesty, which is implied by $h = 1$ in the input of Algorithm
  \ref{alg:mainBreimanFiniteSample}, $w^{(j)}(\vecx)$ is independent of $y^{(j)}$ and we have that
  the first term is equal to $0$ by a tower law. Thus we have:
  \begin{align*}
    \p{\Exp_{\train_n \sim \calD^n}\b{m_{n}(\vecx)} - m(\vecx)}^2 =~& \p{\Exp_{\train_n \sim \calD^n}\b{\sum_{j = 1}^s w^{(j)}(\vecx)\, (m(\vecx^{(j)}) - m(\vecx))}}^2\\
    \leq~& \Exp_{\train_n \sim \calD^n}\b{\p{\sum_{j = 1}^n w^{(j)}(\vecx) (m(\vecx^{(j)}) - m(\vecx))}^2}\\
    \intertext{Let also $A = \{\vecz \mid \vecz_R = \vecx_R\}$, then using the
    multiplicative form of the Chernoff Bound from the proof of the first part of the theorem we
    get
    $\Prob_{\train_n \sim \calD^n}\p{\sum_{i = 1}^n \chara\{\vecx^{(i)} \in A\} \ge 1} \ge 1 - \delta$.
    Therefore with probability $1 - \delta$ the path of the tree that leads to $\vecx$ has split
    all the relevant coordinates $R$ and hence for all $j$ such that $w^{(j)}(\vecx) > 0$ it
    holds that $\vecx^{(j)}_R = \vecx_R$ which in turn implies that
    $m(\vecx^{(j)}) = m(\vecx)$. With the rest $\delta$ probability the square inside the
    expectation is at most $1$ since $m(\cdot) \in \b{-\frac{1}{2}, \frac{1}{2}}$, hence we get}
    \p{\Exp_{\train_n \sim \calD^n}\b{m_{n}(\vecx)} - m(\vecx)}^2 \le & ~~~\delta.
  \end{align*}
\end{proof}

\subsection{Proof of Theorem \ref{thm:finalRecoverySubmodularIndependentBalancedProperBreimanFiniteSample}} \label{sec:app:proofs:Breiman:shallow}

  Observe that the output estimate $m_n(\cdot; \calP_n)$ and partition $\calP_n$ of
Algorithm~\ref{alg:mainBreimanFiniteSample}, satisfies the conditions of
Lemma~\ref{lem:bias-variance}. Moreover, since the number of vertices in a binary tree upper
bounds the number of leafs we can apply Corollary~\ref{cor:criticalradiusShallow} and we
have that the critical radius quantity $\delta_n$ is of order
$\Theta\p{\sqrt{\frac{t\log(d\, t)(1+\log(n))}{n}}}$, if the total number of nodes is at most $t$.
Thus applying the bound presented in \eqref{eq:mainBeforeFinal} with the bound on $\delta_n$ we
can conclude each of the cases of Theorem
\ref{thm:finalRecoverySubmodularIndependentBalancedProperBreimanFiniteSample} by using the
corresponding case of Theorem
\ref{thm:recoverySubmodularIndependentBalancedProperBreimanFiniteSample}.

\subsection{Proof of Theorem \ref{thm:finalRecoverySubmodularIndependentBalancedProperBreimanFiniteSampleFullyGrown}} \label{sec:app:proofs:Breiman:fullyGrown}

From case 1. of Theorem \ref{thm:biasStronglyPartitionSparseBreimanFiniteSample} and since the
maximum possible value diameter is $1$, we have that if
$s \ge \tilde{\Omega}\p{\frac{2^r(\log(d/\delta))}{\zeta \cdot \beta^2}}$ then
$\Exp_{\train_n \sim \calD^n}\b{\Delta_m(\calP_n)} \le \delta$ which implies
$\Exp_{\vecx \sim \calD_x}[(\bar{m}_s(\vecx) - m(\vecx))^2] \le \delta$. Putting this together with
Lemma \ref{lem:biasVarianceDeepHonestForests} we get that
if $s = \tilde{\Omega}\p{\frac{2^r(\log(d/\delta))}{\zeta \cdot \beta^2}}$ then
\begin{equation*}
  \Prob_{\train_n \sim \calD^n}\p{\Exp_{\vecx \sim \calD_x}[(m_{n, s}(\vecx) - m(\vecx))^2] \ge \tilde{\Omega}\p{\frac{2^r(\log(d/(\delta \cdot \delta')))}{n \cdot \zeta \cdot \beta^2}} + \delta} \le \delta'.
\end{equation*}
From the above we get Theorem
\ref{thm:finalRecoverySubmodularIndependentBalancedProperBreimanFiniteSampleFullyGrown} by setting
$\delta = \tilde{\Omega}\p{\frac{2^r(\log(d/\delta'))}{n \cdot \zeta \cdot \beta^2}}$.
 \newpage

\section{Proofs of Asymptotic Normality}\label{app:normality}

\subsection{Proof of  Theorem~\ref{thm:normality-level}}

  We define $m_{s, \pi}$ to be the output of the Algorithm \ref{alg:mainLevelSplitFiniteSample}
when the samples have been permuted by the permutation $\pi \in S_s$. We denote by $\calS_n$ the
uniform distribution over the symmetric group $S_n$. We also define
\begin{align} \label{eq:asymptoticNormality:definitionOfExpectedTree}
  \bar{m}_{s}(\vecx) = \Exp_{\train_n \sim \calD^n, \tau \sim \calS} \b{m_{n, s, \tau}(\vecx)} = \Exp_{\train_s \sim \calD^s, \pi \sim \calS_s} \b{m_{s, \pi}(\vecx)}.
\end{align}
where the last inequality follows due to symmetry of the distribution $\calD^n$ and the definition
of $m_{n, s, \tau}$. We also remind that $\bar{m}_{n, s}$ is equal to
$\Exp_{\tau \sim S_n}\b{m_{n, s, \tau}}$ and that $m_{n, s, B}$ is the Monte Carlo approximation of
$\bar{m}_{n, s}$ with $B$ terms. We now have the following.
\begin{align}
    \sigma_n^{-1}(\vecx) (m_{n, s, \tau}(\vecx) - m(\vecx)) & = \sigma_n^{-1}(\vecx) (m_{n, s, B}(\vecx) - \bar{m}_{n, s}(\vecx)) + \sigma_n^{-1}(\vecx) (\bar{m}_{n, s}(\vecx) - \bar{m}_s(\vecx)) \nonumber \\
    & + \sigma_n^{-1}(\vecx) (\bar{m}_{s}(\vecx) - m(\vecx)). \label{eq:proof:asymptoticNormality:decomposition}
\end{align}
\noindent We define for simplicity
$w^{(j)}(\vecx) = \frac{\chara\{\vecx_{\calT_n(S, \vecx)}^{(j)}=\vecx_{\calT_n(S, \vecx)}\}}{N_n(\vecx; \calT_n(S, \vecx))}$. By Theorem~2 of \cite{Fan2018} we have that:
\begin{equation} \label{eq:proof:asymptoticNormality:normalitySecondTerm}
    \sigma_n^{-1}(\vecx) (\bar{m}_{n, s}(\vecx) - \bar{m}_s(\vecx)) \rightarrow N(0, 1)
\end{equation}
where:
\begin{align*}
  \sigma_n^2(\vecx) =~& \frac{s^2}{n} \Var_{\vecx^{(1)} \sim \calD_x}\left(\Exp_{\train_s \sim \calD^s, \tau \sim \calS_n}\left[\sum_{j = 1}^s w^{(j)}(\vecx)\, y^{(j)} \mid \vecx^{(1)}, y^{(1)}\right]\right)\\
  \geq~& \frac{s^2}{n} \Exp_{\vecx^{(1)} \sim \calD_x}\left[\Exp_{\train_s \sim \calD^s, \tau \sim \calS_n}\left[w^{(1)}(\vecx) \mid \vecx^{(1)}\right]^2\, \sigma^2(\vecx^{(1)})\right]\\
  \geq~& \frac{s^2 \sigma^2}{n} \Exp_{\vecx^{(1)}}\left[\Exp_{\train_s \sim \calD^s, \tau \sim \calS_n}\left[w^{(1)}(\vecx) \mid \vecx^{(1)}\right]^2\right]\\
  \geq~& \frac{s^2 \sigma^2}{n} \Exp_{\train_s \sim \calD^s, \tau \sim \calS_n}\left[w^{(1)}(\vecx)\right]^2\\
  =~& \frac{s^2 \sigma^2}{n} \frac{1}{4\,s^2} = \frac{\sigma^2}{4\, n}
\end{align*}
Where the last inequality follows by our assumption that $\sigma^2(\vecx) \geq \sigma^2$,
uniformly for all $\vecx$ and the fact that due to the expectation over the random
permutation $\tau$ in the beginning of the algorithm we have symmetry between the samples and
hence: $\Exp_{\train_n \sim \calD^n, \tau \sim \calS_n}[w^{(1)}(\vecx)] = 1/(2\,s)$. Also, since
$s \ge \tilde{\Omega}\p{\frac{2^r(\log(d/\delta))}{\beta^2} + \frac{2^r \log(n)}{\zeta}}$, from
part 3 of Theorem \ref{thm:biasStronglySparseLevelSplitFiniteSample} with $\delta = 1/n^2$ we have
that:
\begin{align} \label{eq:proof:asymptoticNormality:normalityThirdTerm}
  \sigma_n^{-1}(\vecx) \p{\bar{m}_{s}(\vecx) - m(\vecx)} = o_p(1)
\end{align}
where we have used the fact that the part 3 of Theorem
\ref{thm:biasStronglySparseLevelSplitFiniteSample} holds pointwise for every permutation
$\tau \in S_n$.
The last step is to bound the error from the Monte Carlo approximation $m_{n, s, B}$ of
$\bar{m}_{n, s}$. Since we have fixed the $\vecx \in \{0, 1\}$ before the execution of the
algorithm and since $m_{n, s, \tau}(\vecx) \in \b{-1/2, 1/2}$, we can use the Hoeffding bound with
$B = n^2 \log(n)$ and
get the following
\begin{align*}
  \Prob\p{\abs{m_{n, s, B}(\vecx) - \bar{m}_{n, s}(\vecx)} \ge \frac{1}{n}} \le \frac{2}{n^2}
\end{align*}
where the probability is over the randomness that is used to sample the $B$ permutations uniformly
from $S_n$ and the random sub-samples of size $s$ used for each of the $B$ trees to compute the empirical expectation $m_{n, s, B}$. Hence we have that
\begin{align} \label{eq:proof:asymptoticNormality:normalityFirstTerm}
  \sigma_n^{-1}(\vecx) \p{m_{n, s, B}(\vecx) - \bar{m}_{n, s}(\vecx)} = o_p(1).
\end{align}
Finally, putting together \eqref{eq:proof:asymptoticNormality:decomposition},
\eqref{eq:proof:asymptoticNormality:normalitySecondTerm},
\eqref{eq:proof:asymptoticNormality:normalityThirdTerm},
\eqref{eq:proof:asymptoticNormality:normalityFirstTerm} and invoking Slutzky's theorem we get that:
\begin{align*}
    \sigma_n^{-1}(\vecx) (m_{n, s, B}(\vecx) - m(\vecx)) \rightarrow_d N(0, 1).
\end{align*}

\subsection{Proof of Theorem~\ref{thm:normality-breiman}}

The proof is almost identical to the proof of Theorem~\ref{thm:normality-level}
presented in the previous section. The only difference is the derivation of
\eqref{eq:proof:asymptoticNormality:normalityThirdTerm}. For this instead of
using part 3 of Theorem \ref{thm:biasStronglySparseLevelSplitFiniteSample} we
use part 2 of Theorem \ref{thm:biasStronglyPartitionSparseBreimanFiniteSample}
again with $\delta = 1/n$. The rest of the proof remains the same and
Theorem~\ref{thm:normality-breiman} follows.

\end{document}